\documentclass{amsart}
\usepackage[margin=3cm]{geometry}
\usepackage{amssymb,amsmath,amsfonts,mathtools,latexsym}
\usepackage[pdftex]{hyperref} 
\usepackage{tikz-cd} 
\usepackage{enumerate} 
\newtheorem{theorem}{Theorem}[section]
\newtheorem{corollary}[theorem]{Corollary}
\newtheorem{definition}[theorem]{Definition}
\newtheorem{example}[theorem]{Example}
\newtheorem{lemma}[theorem]{Lemma}
\newtheorem{problem}{Problem}
\newtheorem{proposition}[theorem]{Proposition}

\begin{document}

\title[$\mbox{Graded Imbeddings in Finite Dimensional Simple Graded Algebras}$]{Graded Imbeddings in Finite Dimensional Simple Graded Algebras}
\date{October 17, 2024}

\author[De Fran\c{c}a]{Antonio de Fran\c{c}a$^\dag$}
\address{Department of Mathematics, Federal University of Campina Grande, 58429-970 Campina Grande, Para\'iba, Brazil}
\email{\href{mailto: a.defranca@yandex.com}{a.defranca@yandex.com}}
\thanks{$^\dag$The author was partially supported by Para\'iba State Research Foundation (FAPESQ), Grant \#2023/2158.}


\begin{abstract}
Let $\mathbb{F}$ be a field and $\mathsf{G}$ a group. This work is inspired in the following problem: ``{\it given a division (simple) $\mathsf{G}$-graded $\mathbb{F}$-algebra, is there any other division (simple) $\mathsf{G}$-graded $\mathbb{F}$-algebra such that the former can be $\mathsf{G}$-imbedded in the latter?}''. In this work, we answer this question affirmatively for $\mathbb{F}$ algebraically closed, $\mathsf{G}$ finite abelian, and associative algebras of finite dimension. To prove this, we apply concepts and properties of Group Cohomology. We show $\mathcal{H}^2(H, \mathbb{F}^*)=\mathsf{res}^\mathsf{G}_H \left(\mathcal{H}^2(\mathsf{G}, \mathbb{F}^*)\right)$, where  $H$ is a subgroup of $\mathsf{G}$ and $\mathsf{res}^\mathsf{G}_H$ is the restriction homomorphism. Posteriorly, we prove that, given any $H_1,H_2\leq\mathsf{G}$ and $\sigma_i\in\mathcal{Z}^2(H_i,\mathbb{F}^*)$, $i=1,2$, are equivalent: i) $\mathbb{F}^{\sigma_1}[H_1] \stackrel{\mathsf{G}}{\hookrightarrow} \mathbb{F}^{\sigma_2}[H_2]$; ii) $H_1\leq H_2$ and $[\sigma_1]=[\sigma_2]_{H_1}$; iii) $\mathsf{T}^{\mathsf{G}}(\mathbb{F}^{\sigma_2}[H_2])\subseteq \mathsf{T}^{\mathsf{G}}(\mathbb{F}^{\sigma_1}[H_1])$, where $\mathsf{T}^{\mathsf{G}}(\mathbb{F}^{\sigma_i}[H_i])$ is the $\mathsf{G}$T-ideal of graded identities of $\mathbb{F}^{\sigma_i}[H_i]$. Furthermore, we prove that, given $\mathfrak{A}$ and $\mathfrak{B}$ two finite dimensional simple $\mathsf{G}$-graded $\mathbb{F}$-algebras, if $\mathbb{F}$ is algebraically closed, $\mathsf{char}(\mathbb{F}) = 0$ or $\mathsf{char} (\mathbb{F})$ is coprime with the order of each finite subgroup of $\mathsf{G}$, and any subgroup of $\mathsf{G}$ is normal, then $\mathsf{T}^{\mathsf{G}}(\mathfrak{A})\subseteq \mathsf{T}^{\mathsf{G}}(\mathfrak{B})$ iff $\mathfrak{B} \stackrel{\mathsf{G}}{\hookrightarrow} \mathfrak{A}$. As a consequence of this last result, we guarantee that $\mathsf{T}^{\mathsf{G}}(\mathfrak{A})= \mathsf{T}^{\mathsf{G}}(\mathfrak{B})$ iff $\mathfrak{A} \cong_{\mathsf{G}} \mathfrak{B}$.
\end{abstract}

\maketitle


\section{Introduction}

Let $\mathbb{F}$ be a field, $\mathsf{G}$ a group, and $\mathfrak{A}$ an associative $\mathbb{F}$-algebra with a $\mathsf{G}$-grading. Recall that $\mathfrak{A}$ is a $\mathsf{G}$-graded\footnote{To more details, as well as an overview, on graded algebras and other structures with gradings by groups, see the works \cite{BahtSehgZaic08}, \cite{ElduKoch13}, \cite{NastOyst04,NastOyst11} and their references, as well as the works \cite{Mardua01,Mardua02} from author of this work.} algebra when there exist subspaces $\mathfrak{A}_\xi$'s of $\mathfrak{A}$, indexed by $\xi\in\mathsf{G}$, such that $\mathfrak{A}=\bigoplus_{\xi\in\mathsf{G}}\mathfrak{A}_\xi$ and $\mathfrak{A}_\xi\mathfrak{A}_\zeta\subseteq\mathfrak{A}_{\xi\zeta}$ for any $\xi\,\zeta\in\mathsf{G}$. We say that $\mathfrak{A}$ is division (resp. simple) graded when $\mathfrak{A}$ is unitary and all nonzero homogeneous element is invertible in $\mathfrak{A}$ (resp. $\mathfrak{A}^2\neq\{0\}$ and $\mathfrak{A}$ has no non-trivial graded ideals). An ($\mathsf{G}$-graded) imbedding\footnote{According to the \textit{Oxford Advanced Learner's Dictionary}, ``imbed'' and ``embed'' are correct spellings with the same meaning. We have adopted the spelling with ``i'', although in some parts of our text, out of respect for the references, we have maintained the spelling with ``e''.} of ($\mathsf{G}$-graded) algebras is a ($\mathsf{G}$-graded) monomorphism of ($\mathsf{G}$-graded) algebras. So, the following problem naturally arises and it is our inspiration in this work:
\begin{problem}\label{prob1}
Given a division (resp. simple) $\mathsf{G}$-graded algebra $\mathfrak{B}$, is there a division (resp. simple) $\mathsf{G}$-graded algebra $\mathfrak{A}$ such that $\mathfrak{B}$ can be $\mathsf{G}$-imbedded into $\mathfrak{A}$?
\end{problem}

The importance of studying the simple and division (graded) algebras, and the relationships between them, comes from classical results and problems of Ring Theory, such as the graded version of Wedderburn-Malcev Theorem (see \cite{Malc42}, Theorem 72.19 in \cite{CurtRein62}, and Lemma 2 in \cite{Svir11}) and Specht's Problem (see works \cite{BeloRoweVish12}, \cite{Geno81}, \cite{Keme87,Keme91} and \cite{Popo81}).

In 1930, van der Waerden posed in your book \cite{Waer30} the following question: ``{\it which commutative rings have a quotient field?}''. This is equivalent to asking which rings can be imbedded in fields. So, in 1937, A.I. Mal'tcev proved in your work \cite{Malc37} that every commutative semigroup can be immersed into a group, and it presented a  noncommutative semigroup which cannot be immersed into a group. Obviously any commutative ring without divisors of zero can be immersed into a field. In this same work, A.I. Mal'tcev exhibits a noncommutative ring without divisors of zero which cannot be immersed into an algebraic field. The latter is one of the first works that addresses what is now known as the ``{\it embedding problem}''. Various mathematicians have studied this problem since 1930. We suggest the works \cite{Malc39,Malc40,Malc60}, due to A.I. Mal'tcev.

In 1956, A.I. Shirshov showed in \cite{Shir56} that any associative algebra $A$ of finite rank over a field $\mathbb{F}$ is isomorphic to a subalgebra of some associative $\mathbb{F}$-algebra $B$ of finite rank with two generators. See also \cite{Shir58.1}, due to A.I. Shirshov. Already in \cite{Boku63.1}, L.A. Bokut' (1963) proved that any associative algebra can be embedded in a simple associative algebra. Posteriorly, in \cite{Boku87}, L.A. Bokut' made a compilation of all (or almost all) works that addresses the embedding problem up to 1987. We also suggest the works \cite{Boku62.1,Boku62,Boku63,Boku67,Boku76}, due to L.A. Bokut'.

In 1974, G.M. Bergman showed in \cite{Berg74} that the endomorphism ring $B=\mathsf{End}_\mathbb{Z}(\mathbb{Z}_{p^2}\oplus\mathbb{Z}_p)$, $p$ prime, which is a finite ring with unity, cannot be embedded in matrices over any commutative ring. Afterwards, in 1985, A. Sychowicz proved in  \cite{Sych85} that the smallest ring which cannot be embedded in matrices over a commutative ring has order $32$. In turn, in \cite{Malc87}, Yu.N. Mal'tcev (1987) constructed an infinite sequence of finite rings $B$ and $B^{(m)}$, $m\geq2$, which cannot be embedded in rings of matrices over a commutative ring.


More recently, in 2004, J.Z. Gon{\c{c}}alves and D.S. Passman proved in \cite{GoncPass04} that if $G$ is a finite group and $p$ is a prime, then $U(\mathbb{Z}[G])$ contains the free product $\mathbb{Z}_p\ast\mathbb{Z}$ iff $G$ has a noncentral element of order $p$, where $\mathbb{Z}[G]$ is the integral group ring of $G$ and $U(\mathbb{Z}[G])$ its group of units. In turn, in 2007, V.M. Petrogradsky, Yu.P. Razmyslov and E.O. Shishkin proved in \cite{PetrRazmShis07} that if there exists an exact sequence $0\rightarrow H\rightarrow L\rightarrow G\rightarrow 0$ of Lie algebras, then $L$ can be embedded into the wreath product $H\wr G$. Already in \cite{RoweSalt13}, L. Rowen and D. Saltman (2013) exhibited a counterexample to the question, due to L. Small, ``{\it are two finite dimensional division algebras containing a common central subfield $\mathbb{F}$ embeddable in a common division algebra?}''. In addition, the authors answered affirmatively this problem for division algebras whose centers are finitely generated over a common perfect subfield. Finally, in \cite{AnquCortShes18}, L.A. Anquela, T. Cort{\'e}s and I. Shestakov (2018) obtained examples of Jordan systems (over arbitrary rings of scalars) that cannot be imbedded in nondegenerate systems.


Another problem which various mathematicians have studied in recent decades is to determine when two algebras which have the same {\it polynomial identities}\footnote{To further details, as well as an overview, on (graded) polynomial identities, we indicate the books \cite{Dren00}, \cite{DrenForm12}, \cite{GiamZaic05}, \cite{Jaco75} and \cite{Rowe80}, and also the work \cite{Mardua03} due to author of this work.} are isomorphic. The graded version of this problem has also been studied.

In 1952, S.A. Amitsur showed in \cite{Amit52} that every PI-ring which has no nilpotent ideals can be embedded in a matrix ring over some commutative ring. See also the works \cite{Amit71,Amit73}, due to A.S. Amitsur. In turn, A.Z. Annan'in (1980) showed in \cite{Anan80} that, given an algebra $R$, if $R^\#$ satisfies the identities $[x,y][z,t][x_1,\dots,x_k]\equiv0$, $[[x,y],z][x_1,\dots,x_k]\equiv0$ and $[x_1,y_1]\cdots[x_l,y_l]\equiv0$, then $R$ is embeddable in $UT_n(K)$, where $R^\#$ is the algebra $R$ with an identity adjoined and $UT_n(K)$ is the algebra of all upper triangular $n\times n$ matrices over some commutative algebra $K$. Recently, O. David (2012) showed in \cite{Davi12} that, given $G$ a finite abelian group, $\mathbb{F}$ an algebraically closed field of characteristic zero, and $A$ and $B$ two finite dimensional simple $G$-graded algebras over $\mathbb{F}$, then there is a $G$-graded embedding of $B$ in $A$ iff the $G$-graded identities of $A$ are contained in the $G$-graded identities of $B$.

In 2010, P. Koshlukov and M.V. Zaicev proved in \cite{KoshZaic10} that, given $G$ an arbitrary abelian group and $A$ and $B$ two finite dimensional simple $G$-graded algebras over an algebraically closed field $\mathbb{F}$ such that the orders of all finite subgroups of $G$ are invertible in $\mathbb{F}$, then $A$ and $B$ are isomorphic if and only if they satisfy the same $G$-graded identities. Afterwards, in \cite{AljaHail14}, E. Aljadeff and D. Haile (2014) generalized this last result for any group and $\mathbb{F}$ with characteristic zero. Already in \cite{ShesZaic11}, I. Shestakov and M. Zaicev (2011) proved that when $\mathbb{F}$ is an algebraically closed field and $A$ and $B$ are arbitrary finite dimensional simple (not necessary associative) algebras over $\mathbb{F}$, then $A$ and $B$ are isomorphic if and only if they satisfy the same identities.

In light of the above, and keeping in mind Problem \ref{prob1}, this work is dedicated to studying and to answering the following problem:
\begin{problem}\label{prob2}
How and when can two simple (division) graded algebras of finite dimension be imbedded into each other? And given a finite dimensional simple (division) graded algebra $\mathfrak{B}$, when is it possible to determine a finite dimensional simple (division) graded algebra $\mathfrak{A}$ such that $\mathfrak{B}$ is $\mathsf{G}$-imbedded in $\mathfrak{A}$?
\end{problem}

We begin the approach to Problem \ref{prob2} by recording in \S\ref{pre} the definitions and concepts necessary for the better development of this work. The reader familiar with objects as $\mathsf{G}$-graded algebras, $\mathsf{G}$-graded homomorphisms, $2$-cocycles on groups, twisted group algebra and matrices over these algebras, as well as its gradings by groups, and graded polynomial identities, can continue reading this work from \S\ref{sec3}, returning if necessary.

We dedicated \S\ref{sec3} of this work to the study of the question ``\textit{given a group $\mathsf{G}$ and two subgroups $H\leq N\leq \mathsf{G}$, if $\sigma$ is a $2$-cocycle on $H$, when can we ensure the existence of a $2$-cocycle on $N$ which extends $\sigma$?}''. Here, we present some notions and properties of the cohomology theory of groups. The reader familiar with these definitions and their main properties can continue from \S\ref{resGH}, page \pageref{resGH}, or even from Theorem \ref{2.07}, page \pageref{2.07}. The main result of this section, which also is our answer to above problem (which motivates this section), is the following:

	\vspace{0.15cm}
	\noindent{\bf Corollary \ref{2.09}}. 
	Let $H$ be a central subgroup of $\mathsf{G}$, $[\mathsf{G}:H]< \infty$, and $\mathsf{M}$ an abelian group with the trivial $\mathsf{G}$-action. If $[\mathsf{G}: H]$ is invertible in $\mathsf{M}$, then 
	\begin{equation}\nonumber
	\mathcal{H}^2(H, \mathsf{M})=\mathsf{res}^\mathsf{G}_H \left(\mathcal{H}^2(\mathsf{G}, \mathsf{M})\right)
	\ .
	\end{equation}

In \S\ref{sec4}, we present a study on $\mathsf{G}$-graded imbeddings of finite dimensional simple $\mathsf{G}$-graded $\mathbb{F}$-algebras, where $\mathbb{F}$ is a field and $\mathsf{G}$ is a group, both arbitrary. Here, the main goal is to answer the first part of Problem \ref{prob2}. In this section, the main results are as follows:

	\vspace{0.15cm}
	\noindent{\bf Lemma \ref{4.16}} (Imbedding). 
	Let $\mathbb{F}$ be a field and $\mathsf{G}$ a group. For any two subgroups $H_1$ and $H_2$ of $\mathsf{G}$ and any two $2$-cocycles $\sigma_1\in\mathcal{Z}^2(H_1,\mathbb{F}^*)$ and $\sigma_2\in\mathcal{Z}^2(H_2,\mathbb{F}^*)$, $\mathbb{F}^{\sigma_1}[H_1] \stackrel{\mathsf{G}}{\hookrightarrow} \mathbb{F}^{\sigma_2}[H_2]$ iff $(H_1,\sigma_1)\preceq (H_2, \sigma_2)$.
	\vspace{0.15cm}

A consequence of the Imbedding Lemma is Corollary \ref{4.05} which ensures (in algebraically closed fields), given a finite dimensional algebra $\mathfrak{B}$ with a division $\mathsf{G}$-grading $\Gamma$, if there exists a subgroup $H$ of $\mathsf{G}$ such that $\mathsf{Supp}(\Gamma)$ is a central subgroup of $H$, with $[H:\mathsf{Supp}(\Gamma)]< \infty$, then there exists a $2$-cocycle $\sigma\in\mathcal{Z}^2(H, \mathbb{F}^*)$ such that $\mathfrak{B}\stackrel{\mathsf{G}}{\hookrightarrow} \mathbb{F}^{\sigma}[H]$. This result partially answers the second part of Problem \ref{prob2}. Another consequence of the Imbedding Lemma is the theorem below:

	\vspace{0.15cm}
	\noindent{\bf Theorem \ref{1.39}} (Imbedding). 
Let $\mathbb{F}$ be a field and $\mathsf{G}$ a group. For $i=1,2$, consider $\mathfrak{B}_i=M_{k_i}(\mathbb{F}^{\sigma_i}[H_i])$ the ${k_i}\times {k_i}$ matrix algebra over $\mathbb{F}^{\sigma_i}[H_i]$ with an elementary-canonical $\mathsf{G}$-grading defined by a $k_i$-tuple $\theta_i=(\theta_{i1},\dots,\theta_{i{k_i}})\in\mathsf{G}^{k_i}$, where $H_i$ is a subgroup of $\mathsf{G}$, $\sigma_i\in \mathcal{H}^2(H_i,\mathbb{F}^*)$ is a $2$-cocycle. Suppose that $\theta_{i1},\dots,\theta_{ik_i}\in \mathsf{N}_{\mathsf{G}}(H_i)$, for $i=1,2$. Then $\mathfrak{B}_1 \stackrel{\mathsf{G}}{\hookrightarrow} \mathfrak{B}_2$ iff $k_1\leq k_2$, $\mathbb{F}^{\sigma_1}[H_1] \stackrel{\mathsf{G}}{\hookrightarrow} \mathbb{F}^{\sigma_2}[H_2]$ and there exist $\alpha\in Sym(k_2)$, $\delta\in \mathsf{N}_{\mathsf{G}}(H_2)$ and $\xi_1,\dots,\xi_{k_1}\in H_2$ such that $\theta_{1j}=\delta\xi_{j}\theta_{2 \alpha(j)}$ for all $j=1,\dots,k_1$.
	\vspace{0.15cm}

And so, combining Imbedding Theorem and Corollary \ref{2.09}, we have the corollary below, which completes our answer to Problem \ref{prob2}:

	\vspace{0.15cm}
	\noindent{\bf Corollary \ref{4.07}}. 
Let $\mathbb{F}$ be an algebraically closed field, $\mathsf{G}$ a group, $H$ a subgroup of $\mathsf{G}$, $\sigma\in \mathcal{H}^2(H,\mathbb{F}^*)$ a $2$-cocycle, and $\mathfrak{B}=M_{k}(\mathbb{F}^{\sigma}[H])$ the ${k}\times {k}$ matrix algebra with an elementary-canonical $\mathsf{G}$-grading defined by a $k$-tuple $\theta\in\mathsf{G}^{k}$. For any subgroup $N$ of $\mathsf{G}$ such that $H$ is a central subgroup of $N$ of finite index, there exists a $2$-cocycle $\rho\in\mathcal{Z}^2(N, \mathbb{F}^*)$ such that 
\begin{equation}\nonumber
   \begin{tikzcd}
\mathfrak{B} \arrow[r,hook]{l}{\mathsf{G}}\arrow[d,hook]{l}{\rotatebox{-90}{$\mathsf{G}$}} &
M_{k}(\mathbb{F}^{\rho}[N]) \arrow[d,hook]{l}{\rotatebox{-90}{$\mathsf{G}$}} \\
M_{t}(\mathbb{F}^{\sigma}[H]) \arrow[r,hook]{l}{\mathsf{G}} & M_{t}(\mathbb{F}^{\rho}[N])  
    \end{tikzcd}
\end{equation}
for all integer $t\geq k$, where $M_{k}(\mathbb{F}^{\rho}[N])$ is graded with the elementary-canonical $\mathsf{G}$-grading defined by $\theta$, and $M_{t}(\mathbb{F}^{\sigma}[H])$ and $M_{t}(\mathbb{F}^{\rho}[N])$ are graded with elementary-canonical $\mathsf{G}$-gradings defined by some $t$-tuple $\phi\in\mathsf{G}^t$.
	\vspace{0.15cm}

We finalize this work with \S\ref{sec5}. In this section, we exhibit some results that, based on the $\mathsf{G}$T-ideals of graded polynomial identities, ensure other conditions for the existence of $\mathsf{G}$-imbeddings between $\mathsf{G}$-simple algebras of finite dimension. Below, the main results of this section.

	\vspace{0.15cm}
	\noindent{\bf Lemma \ref{4.20}} (Imbedding). 
Let $\mathbb{F}$ be field, $\mathsf{G}$ a group, $H_1$ and $H_2$ two subgroups of $\mathsf{G}$ and $\sigma_1$ and $\sigma_2$ their $2$-cocycles, respectively. Then $\mathbb{F}^{\sigma_1}[H_1] \stackrel{\mathsf{G}}{\hookrightarrow} \mathbb{F}^{\sigma_2}[H_2]$ iff $\mathsf{T}^{\mathsf{G}}(\mathbb{F}^{\sigma_2}[H_2])\subseteq \mathsf{T}^{\mathsf{G}}(\mathbb{F}^{\sigma_1}[H_1])$.
	\vspace{0.15cm}

A consequence of the above lemma is Proposition \ref{4.11}, which ensures that, if $\mathfrak{A}$ and $\mathfrak{B}$ are two finite dimensional division $\mathsf{G}$-graded $\mathbb{F}$-algebras, with $\mathbb{F}$ algebraically closed, then $\mathsf{T}^{\mathsf{G}}(\mathfrak{A})\subseteq \mathsf{T}^{\mathsf{G}}(\mathfrak{B})$ iff $\mathfrak{B} \stackrel{\mathsf{G}}{\hookrightarrow} \mathfrak{A}$. In addition, $\mathsf{T}^{\mathsf{G}}(\mathfrak{A})= \mathsf{T}^{\mathsf{G}}(\mathfrak{B})$ iff $\mathfrak{A} \cong_{\mathsf{G}} \mathfrak{B}$.

	\vspace{0.15cm}
	\noindent{\bf Theorem \ref{4.21}} (Imbedding). 
Let $\mathsf{G}$ be a group, $\mathbb{F}$ an algebraically closed field, and $\mathfrak{A}$ and $\mathfrak{B}$ two finite dimensional simple $\mathsf{G}$-graded $\mathbb{F}$-algebras. Suppose that either $\mathsf{char}(\mathbb{F}) = 0$ or $\mathsf{char} (\mathbb{F})$ is coprime with the order of each finite subgroup of $\mathsf{G}$, and that any subgroup of $\mathsf{G}$ is normal. Then $\mathsf{T}^{\mathsf{G}}(\mathfrak{A})\subseteq \mathsf{T}^{\mathsf{G}}(\mathfrak{B})$ iff $\mathfrak{B} \stackrel{\mathsf{G}}{\hookrightarrow} \mathfrak{A}$.
	\vspace{0.15cm}

Therefore, the previous results are contributions to Problem \ref{prob2}. To conclude, it is important to note that Theorem \ref{4.21} guarantees $\mathsf{T}^{\mathsf{G}}(\mathfrak{A})= \mathsf{T}^{\mathsf{G}}(\mathfrak{B})$ iff $\mathfrak{A} \cong_{\mathsf{G}} \mathfrak{B}$.


\section{Preliminaries}\label{pre}

In this work, the word ``algebra'' always denotes an associative algebra over field $\mathbb{F}$, and $\mathsf{G}$ denotes a group.

In this section, we recall some concepts, although basic, which are essential for our study. We talk a little bit about the definitions of ``algebra graded by a group'', ``graded homomorphisms'' and ``graded imbeddings'', as well as other definitions and some results. Later, we recall the definitions and properties of ``twisted group algebra'' and ``matrices over a twisted group algebra''. We dedicate a part of the work to introduce ``free graded associative algebra'', and ``graded polynomial identities'', and some results connecting these. The reader familiar with these definitions and their main properties can continue from \S\ref{sec3}.

%

\subsection{Gradings on Algebras}

Let $\mathsf{G}$ be a group, $\mathbb{F}$ a field, and $\mathfrak{A}$ an (associative) algebra over $\mathbb{F}$. A $\mathsf{G}$-grading\footnote{To further reading, as well as an overview, on graded algebras, see the works \cite{GiamZaic05}, \cite{ElduKoch13}, \cite{NastOyst04,NastOyst11} and their references.} on $\mathfrak{A}$ is a decomposition $\Gamma: \mathfrak{A} = \bigoplus_{\xi\in \mathsf{G}} \mathfrak{A}_\xi$ that satisfies $\mathfrak{A}_\xi \mathfrak{A}_\zeta \subseteq \mathfrak{A}_{\xi\zeta}$, for any $\xi, \zeta \in \mathsf{G}$, where $\mathfrak{A}_\xi$'s are vector subspaces of $\mathfrak{A}$. In this case, we say ``$\mathfrak{A}$ is a $\mathsf{G}$-graded algebra'' or simply ``$\mathfrak{A}$ is $\mathsf{G}$-graded''. Given an ideal (resp. subalgebra) $I$ of $A$, we say that $I$ is a graded ideal (resp. a graded subalgebra) of $\mathfrak{A}$ if $I =\bigoplus_{\xi\in \mathsf{G}} I \cap \mathfrak{A}_\xi$. The support of $\Gamma$ is the subset of $\mathsf{G}$ given by $\mathsf{Supp}(\Gamma)=\{\xi\in\mathsf{G} : \mathfrak{A}_\xi\neq0\}$.

\begin{definition}
Let $\mathfrak{A}$ be a $\mathsf{G}$-graded algebra. We say that $\mathfrak{A}$ is {\bf simple $\mathsf{G}$-graded} (or simply {\bf $\mathsf{G}$-simple}) if $\mathfrak{A}^2\neq\{0\}$, and $\{0\}$ and $\mathfrak{A}$ are the unique graded ideals of $\mathfrak{A}$. Supposing that $\mathfrak{A}$ is unitary, we say that $\mathfrak{A}$ is {\bf division $\mathsf{G}$-graded} if any nonzero homogeneous element of $\mathfrak{A}$ is invertible in $\mathfrak{A}$.
\end{definition}

It is easy to verify that, given an $\mathbb{F}$-algebra $\mathfrak{B}$ with a $\mathsf{G}$-grading $\Gamma$, if $\mathfrak{B}$ is division $\mathsf{G}$-graded, then $\mathsf{Supp}(\Gamma)$ is a subgroup of $\mathsf{G}$, and $\mathsf{dim}_{\mathbb{F}}(\mathfrak{B}_\xi)=\mathsf{dim}_{\mathbb{F}}(\mathfrak{B}_e)$ for any $\xi\in\mathsf{Supp}(\Gamma)$, where $e$ is the neutral element of $\mathsf{G}$ (see the proof of Lemma 3, \cite{BahtSehgZaic08}). On the other hand, when $\mathfrak{A}$ is a simple $\mathsf{G}$-graded $\mathbb{F}$-algebra, not necessarily the support of your $\mathsf{G}$-grading is a subgroup of $\mathsf{G}$. In fact, consider $\mathfrak{A}=M_2(\mathbb{R})$ the $2\times2$ matrix algebra over $\mathbb{R}$ and $\mathsf{G}=\mathbb{Z}_4$ the cyclic group of order $4$. Now, consider the subspaces of $\mathfrak{A}$ given by $\mathfrak{A}_{\bar0}=\mathsf{span}_{\mathbb{R}}\{E_{11},E_{22}\in M_2(\mathbb{R})\}$, $\mathfrak{A}_{\bar1}=\mathsf{span}_{\mathbb{R}}\{E_{12}\in M_2(\mathbb{R})\}$ and $\mathfrak{A}_{\bar3}=\mathsf{span}_{\mathbb{R}}\{E_{21}\in M_2(\mathbb{R})\}$. Note that $\mathfrak{A}=\mathfrak{A}_{\bar0}\oplus\mathfrak{A}_{\bar1}\oplus\mathfrak{A}_{\bar3}$ determines a $\mathsf{G}$-grading on $\mathfrak{A}$ whose the support is not a subgroup of $\mathbb{Z}_4$.

\begin{definition}
Let $\mathbb{F}$ be a field, $\mathsf{G}$ a group and $\mathfrak{A}$ and $\mathfrak{B}$ two $\mathsf{G}$-graded $\mathbb{F}$-algebras. A homomorphism $\psi:\mathfrak{A}\rightarrow \mathfrak{B}$ is called a {\bf $\mathsf{G}$-graded homomorphism} if  $\psi(\mathfrak{A}_\xi)\subseteq\mathfrak{B}_\xi$ for any $\xi\in\mathsf{G}$. We say that $\mathfrak{A}$ and $\mathfrak{B}$ are {\bf $\mathsf{G}$-graded isomorphic}, and we denote this by $\mathfrak{A}\cong_\mathsf{G} \mathfrak{B}$, when there exists an isomorphism $\varphi$ from $\mathfrak{A}$ into $\mathfrak{B}$ such that $\varphi(\mathfrak{A}_\xi)=\mathfrak{B}_\xi$ for any $\xi\in\mathsf{G}$. When $\mathfrak{A}=\mathfrak{B}$, we say that $\varphi$ is a {\bf $\mathsf{G}$-graded automorphism} of $\mathfrak{A}$.
\end{definition}

Let $\mathfrak{A}$ be an $\mathbb{F}$-algebra, $\mathsf{G}$ a group, and $\Gamma: \mathfrak{A}=\sum_{\xi\in\mathsf{G}} \mathfrak{A}_\xi$ and $\widehat{\Gamma}: \mathfrak{A}=\sum_{\zeta\in\mathsf{G}} \widehat{\mathfrak{A}}_\zeta$ two $\mathsf{G}$-gradings on $\mathfrak{A}$. We say that $\Gamma$ and $\widehat{\Gamma}$ are \textbf{equivalent gradings} (on $\mathfrak{A}$) if there exists a $\mathsf{G}$-graded automorphism $\varphi$ of $\mathfrak{A}$ satisfying $\varphi(\mathfrak{A}_\xi) =\widehat{\mathfrak{A}}_\xi$ for any $\xi\in\mathsf{G}$.

The next result is the graded version of Isomorphisms and Correspondence Theorems.

\begin{lemma}\label{4.13}
Let $\mathsf{G}$ be a group, $\mathbb{F}$ a field and $\mathfrak{A}$ and $\mathfrak{B}$ two $\mathsf{G}$-graded $\mathbb{F}$-algebras. If $\varphi:\mathfrak{A}\rightarrow \mathfrak{B}$ is a $\mathsf{G}$-graded homomorphism of algebras, then $\displaystyle\frac{\mathfrak{A}}{\mathsf{ker}(\varphi)}\cong_\mathsf{G} \mathsf{im}(\varphi)$. In addition, if $I\subseteq \mathfrak{B}$ are two $\mathsf{G}$-graded ideals of $\mathfrak{A}$, then $\displaystyle\frac{\mathfrak{A}}{\mathfrak{B}}\cong_\mathsf{G} \displaystyle\frac{{\mathfrak{A}}/{I}}{{\mathfrak{B}}/{I}}$. Moreover, if $I$ is a $\mathsf{G}$-graded ideal of $\mathfrak{A}$, then any $\mathsf{G}$-graded ideal of the $\mathsf{G}$-graded quotient algebra $\mathfrak{A}/I$ is of the form $\mathfrak{B}/I=\{b+I: b\in \mathfrak{B}\}$, where $\mathfrak{B}$ is a $\mathsf{G}$-graded ideal of $\mathfrak{A}$ such that $I\subseteq \mathfrak{B}\subseteq \mathfrak{A}$. 
\end{lemma}
\begin{proof}
The proof is adapted from Theorems 2.136, 6.19, 6.21 and 6.22, in \cite{Rotm10}, pages 157, 406 and 407.
\end{proof}

%

\begin{definition}\label{1.54}
Let $\mathfrak{A}$ and $\mathfrak{B}$ be two $\mathsf{G}$-graded algebras and $\psi:\mathfrak{B} \rightarrow \mathfrak{A}$ a $\mathsf{G}$-graded homomorphism of algebras. We say that $\psi$ is a \textbf{$\mathsf{G}$-graded imbedding} (or {\bf$\mathsf{G}$-imbedding}) from $\mathfrak{B}$ to $\mathfrak{A}$ if $\psi$ is injective. We denote by $\mathfrak{B} \stackrel{\mathsf{G}}{\hookrightarrow} \mathfrak{A}$ when there exists a $\mathsf{G}$-graded imbedding from $\mathfrak{B}$ to $\mathfrak{A}$. Otherwise, we write $\mathfrak{B} \stackrel{\mathsf{G}}{\not\hookrightarrow} \mathfrak{A}$ when $\mathfrak{B}$ cannot be $\mathsf{G}$-imbedded in $\mathfrak{A}$.
\end{definition}

Note that if $\psi$ is a $\mathsf{G}$-graded imbedding from $\mathfrak{B}$ to $\mathfrak{A}$, then there exists a $\mathsf{G}$-graded subalgebra $\hat{\mathfrak{A}}$ of $\mathfrak{A}$, namely $\hat{\mathfrak{A}}=\mathsf{im}(\psi)$, such that $\hat{\mathfrak{A}}$ and $\mathfrak{B}$ are $\mathsf{G}$-graded isomorphic. Therefore, we can see $\mathfrak{B}$ as a $\mathsf{G}$-graded subalgebra of $\mathfrak{A}$ (or assume that $\mathfrak{A}$ has a $\mathsf{G}$-graded ``copy'' of $\mathfrak{B}$). In this case, observe that $\mathsf{Supp}(\Gamma_\mathfrak{B})\subseteq \mathsf{Supp}(\Gamma_{\mathfrak{A}})$.

\subsection{Matrix Algebra over a Twisted Group Algebra}

Here, let us denote by $\mathsf{G}$ a multiplicative group and $\mathbb{F}$ a field, both arbitrary.

\begin{definition}\label{1.05}
The mapping $\sigma: \mathsf{G}\times \mathsf{G} \longrightarrow \mathbb{F}^*$ which satisfies 
	\begin{equation}\nonumber
		\sigma(\xi,\zeta)\sigma(\xi\zeta,\varsigma)=\sigma(\xi,\zeta\varsigma)\sigma(\zeta,\varsigma) \mbox{ for all } \xi,\zeta,\varsigma\in \mathsf{G}
	\end{equation}
is called a \textbf{$2$-cocycle} on $\mathsf{G}$ with values in $\mathbb{F}^*$. The set of all $2$-cocycles on $\mathsf{G}$ with values in $\mathbb{F}^*$ is denoted by $\mathcal{Z}^2(\mathsf{G}, \mathbb{F}^*)$.
\end{definition}

\begin{example}
	The application $\sigma$ from $\mathsf{G}\times \mathsf{G}$ to $\mathbb{F}^*$ given by $\sigma(\xi,\zeta)=1$, for any $\xi,\zeta\in \mathsf{G}$, is a $2$-cocycle called the \textbf{trivial $2$-cocycle}.
\end{example}


%

\begin{example}\label{1.06}
	Let $\mathsf{G}=(\mathbb{Z}_2\times \mathbb{Z}_2, +)$. The map $\sigma:\mathsf{G}\times\mathsf{G}\rightarrow\mathbb{C}^*$ given by the following table 
	\begin{equation}\nonumber
		\begin{tabular}{|c|c|c|c|c|}
		\hline
		$\sigma$ & $(\bar{0},\bar{0})$ & $(\bar{0},\bar{1})$ & $(\bar{1},\bar{0})$ & $(\bar{1},\bar{1})$\\ \hline 
		$(\bar{0},\bar{0})$ & $1$ & $1$ & $1$ & $1$\\ \hline
		$(\bar{0},\bar{1})$ & $1$ & $1$ & $1$ & $1$ \\ \hline
		$(\bar{1},\bar{0})$ & $1$ & $-1$ & $1$ & $-1$ \\ \hline
		$(\bar{1},\bar{1})$ & $1$ & $-1$ & $1$ & $-1$ \\ \hline
		\end{tabular}
	\end{equation}
	defines a $2$-cocycle, i.e. $\sigma$ belongs to $\mathcal{Z}^2(\mathbb{Z}_2\times \mathbb{Z}_2, \mathbb{C}^*)$. 
\end{example}

Now, consider the group algebra $\mathbb{F}\mathsf{G}$. We have that $\mathbb{F}\mathsf{G}$ is an associative algebra with unity, and it has a natural $\mathsf{G}$-grading, defined by $(\mathbb{F}\mathsf{G})_\xi=\mathsf{span}_{\mathbb{F}}\{\eta_\xi\}$, $\xi\in\mathsf{G}$. Let us introduce the algebra $\mathbb{F}\mathsf{G}$ with a new product, called {\it twisted group algebra}, as follows.

\begin{definition}
Let $\mathsf{G}$ be a group, $\mathbb{F}$ a field, and $\sigma: \mathsf{G}\times \mathsf{G}\longrightarrow \mathbb{F}^*$ a $2$-cocycle on $\mathsf{G}$. Consider the $\mathbb{F}$-vector space 
	\begin{equation}\nonumber
	\mathbb{F}^\sigma[\mathsf{G}]=\left\{ \sum_{\xi\in \mathsf{G}} \lambda_\xi \eta_\xi: \lambda_\xi\in\mathbb{F}, \xi \in \mathsf{G} \right\} \ ,
	\end{equation}
where the set $\{\eta_\xi\}_{\xi\in\mathsf{G}}$ is linearly independent over $\mathbb{F}$. Consider on $\mathbb{F}^\sigma[\mathsf{G}]$ the multiplication which extends by linearity the product $\eta_\xi\eta_\zeta=\sigma(\xi,\zeta)\eta_{\xi\zeta}$, for any $\xi,\zeta\in \mathsf{G}$. With this multiplication, $\mathbb{F}^\sigma[\mathsf{G}]$ is an algebra, called a {\bf twisted group algebra}.
\end{definition}

Observe that the equality in Definition \ref{1.05} ensures that $\mathbb{F}^\sigma[\mathsf{G}]$ is an associative algebra. Obviously, if $\sigma$ is the trivial $2$-cocycle on $\mathsf{G}$, then $\mathbb{F}^\sigma[\mathsf{G}]=\mathbb{F} \mathsf{G}$. Furthermore, we have that $\mathfrak{A}=\mathbb{F}^\sigma[\mathsf{G}]$ is $\mathsf{G}$-graded with the natural grading, also called {\bf canonical}, defined by $\mathfrak{A}_\xi = \mathsf{span}_\mathbb{F}\{\eta_\xi\}$, for any $\xi\in\mathsf{G}$. We have that $\mathbb{F}^\sigma[\mathsf{G}]$ is a unitary algebra, where its unity is given by $\sigma(e,e)^{-1} \eta_e$, since $\sigma(e,e)=\sigma(e,\xi)=\sigma(\zeta,e)$ for any $\xi,\zeta\in\mathsf{G}$ (see Proposition \ref{1.07}). Moreover, $\mathbb{F}^{\sigma}[\mathsf{G}]$ is a division graded algebra. Obviously $\mathsf{dim}(\mathbb{F}^{\sigma}[\mathsf{G}])<\infty$ iff $\mathsf{G}$ is a finite group.

Let $\mathbb{F}$ be a field, $H$ a subgroup of group $\mathsf{G}$, $\sigma$ a $2$-cocycle on $H$ with values in $\mathbb{F}^*$ and $\mathfrak{A}=M_k (\mathbb{F}^\sigma [H])$ the $k\times k$ matrix algebra over the twisted group algebra $\mathbb{F}^\sigma [H]$. Let $\theta=(\theta_1,\dots, \theta_k )\in \mathsf{G}^k$ be a $k$-tuple. The $\mathsf{G}$-grading on $\mathfrak{A}$ determined by $\mathfrak{A}_\xi=\mathsf{span}_\mathbb{F}\{E_{ij}\eta_\zeta\in\mathfrak{A}: \xi=\theta_i^{-1}\zeta\theta_j\}$, $\xi\in\mathsf{G}$, is called \textbf{elementary-canonical $\mathsf{G}$-grading} defined by the $k$-tuple $\theta$. When the elementary-canonical $\mathsf{G}$-gradings on $\mathfrak{A}$ defined by $k$-tuples $\theta, \hat\theta \in\mathsf{G}^k$ are equivalent, we say that $\theta$ and $\hat\theta$ define equivalent elementary-canonical $\mathsf{G}$-gradings on $\mathfrak{A}$.

In what follows, we present some results which will be useful in our work.

\begin{lemma}[Theorem 2, \cite{BahtSehgZaic08}, or Theorem 2.13, \cite{ElduKoch13}]\label{teodivgradalg}
	Let $\mathbb{F}$ be an algebraically closed field and $D$ a $\mathsf{G}$-graded $\mathbb{F}$-algebra of finite dimension. Then $D$ is a division graded algebra with support $T\subseteq\mathsf{G}$ iff $D$ is isomorphic to the twisted group algebra $\mathbb{F}^\sigma[T]$ (with its natural $T$-grading regarded as a $\mathsf{G}$-grading) for some $\sigma\in\mathcal{Z}^2(T, \mathbb{F}^*)$, where $T$ is a finite subgroup of $\mathsf{G}$. Two twisted group algebras, $\mathbb{F}^{\sigma_1}[H_1]$ and $\mathbb{F}^{\sigma_2}[H_2]$, are isomorphic as $\mathsf{G}$-graded algebras if and only if $H_1=H_2$ and $[\sigma_1] = [\sigma_2]$. 
\end{lemma}

More generally, given $\mathsf{G}$ a group and $\mathbb{F}$ a field, both arbitrary, if $D$ is an $\mathbb{F}$-algebra with a $\mathsf{G}$-grading $\Gamma$ such that $\mathsf{dim}_\mathbb{F}(D_e)=1$, then $D$ is a division graded algebra iff $D$ is isomorphic to the twisted group algebra $\mathbb{F}^\sigma[T]$, $T=\mathsf{Supp}(\Gamma)$ and $\sigma\in \mathcal{H}^2(T,\mathbb{F}^*)$. To prove this fact, it is enough proceed as in the proofs of Lemma 3 and Theorem 2, in \cite{BahtSehgZaic08}. For example, it is well-known that when $\mathbb{F}$ is algebraically closed, the only division algebraic\footnote{An $\mathbb{F}$-algebra $\mathfrak{A}$ is called \textit{an algebraic algebra} (of bounded degree) if any element $a\in\mathfrak{A}$ satisfies a non-trivial equation $a^n+\lambda_{n-1} a^{n-1}+\cdots+\lambda_1a=0$ (for some $n$ fixed) with coefficients in $\mathbb{F}$.} $\mathbb{F}$-algebra is $\mathbb{F}$ itself (see the proof of Theorem 7.6 in Jaco12.1, p.451), and so the hypothesis ``of finite dimension'' in Lemma \ref{teodivgradalg} can be replaced by ``such that $D_e$ is an algebraic $\mathbb{F}$-algebra''.

\begin{lemma}[Theorem 3, \cite{BahtSehgZaic08}] \label{teosimpgradalg}
	Let $\mathfrak{A} =\bigoplus_{g\in \mathsf{G}} \mathfrak{A}_g$ be a finite dimensional algebra, over an algebraically closed field $\mathbb{F}$ that is graded by a group $\mathsf{G}$. Suppose that either $\mathsf{char}(\mathbb{F}) = 0$ or $\mathsf{char} (\mathbb{F})$ is coprime with the order of each finite subgroup of $\mathsf{G}$. Then $\mathfrak{A}$ is a simple graded algebra if and only if $\mathfrak{A}$ is isomorphic to the tensor product  $M_k (\mathbb{F})\otimes\mathbb{F}^\sigma [H] \cong M_k (\mathbb{F}^\sigma [H])$, that is, if and only if $\mathfrak{A}$ is a matrix algebra over the division graded algebra $\mathbb{F}^\sigma [H]$, where $H$ is a finite subgroup of $\mathsf{G}$ and $\sigma\in\mathcal{Z}^2(H,\mathbb{F}^*)$. The $\mathsf{G}$-grading  on $M_k (\mathbb{F}^\sigma [H])$ is an elementary-canonical grading defined by a $k$-tuple $(\theta_1,\dots, \theta_k )\in \mathsf{G}^k$.
\end{lemma}


\subsection{Graded Polynomial Identities}

Let $\mathsf{G}$ be a group, $\mathbb{F}$ a field, and $\mathbb{F}\langle X^{\mathsf{G}} \rangle$ the free $\mathsf{G}$-graded associative algebra\footnote{For more details on the free $\mathsf{G}$-graded associative $\mathbb{F}$-algebra $\mathbb{F}\langle X^{\mathsf{G}} \rangle$, see \cite{GiamZaic05}, p.66, and \cite{NastOyst04}, Proposition 2.3.1, p.22.} freely generated by $X^{\mathsf{G}}$, where $X^{\mathsf{G}}=\bigcup_{\xi\in\mathsf{G}} X_\xi$, $X_\xi=\{x_1^{(\xi)}, x_2^{(\xi)}, \dots\}$. The elements of $\mathbb{F}\langle X^{\mathsf{G}} \rangle$ are called \textit{graded polynomials}. It is well-known that $\mathbb{F}\langle X^{\mathsf{G}} \rangle$ satisfies the following universal property: given any $\mathsf{G}$-graded (associative) algebra $\mathfrak{A}$, any maps $f_\mathfrak{A}: X^{\mathsf{G}} \rightarrow \mathfrak{A}$ such that $f_\mathfrak{A}(X_\xi)\subseteq\mathfrak{A}_\xi$, $\xi\in\mathsf{G}$, can be extended (uniquely) to a $\mathsf{G}$-graded homomorphism $\varphi_\mathfrak{A}: \mathbb{F}\langle X^{\mathsf{G}} \rangle \rightarrow \mathfrak{A}$ of $\mathsf{G}$-graded algebras. The following diagram illustrates this property: 
\begin{equation}\label{4.17}
\begin{tikzcd}
X^{\mathsf{G}} \arrow[swap,""{name=X}]{d}{inc} \arrow{r}{f_{\mathfrak{A}}} & \mathfrak{A} \\
\mathbb{F}\langle X^{\mathsf{G}} \rangle \arrow[swap,dashed]{ru}{\varphi_{\mathfrak{A}}} \arrow[from=1-2,to=X,pos=0.6,phantom, "\circlearrowleft"] &  
\end{tikzcd} \ ,
\end{equation}
where $inc: X^{\mathsf{G}} \rightarrow \mathbb{F}\langle X^{\mathsf{G}} \rangle$ is the inclusion homomorphism.

Let $\mathfrak{A}$ be a $\mathsf{G}$-graded (associative) algebra. We say that a polynomial $f=f(x_{i_1}^{\xi_1},\dots,x_{i_n}^{\xi_n})\in\mathbb{F}\langle X^{\mathsf{G}} \rangle$ is a \textit{graded polynomial identity}\footnote{To more details, as well as an overview, on (graded) polynomial identities, we indicate the books \cite{Dren00}, \cite{DrenForm12}, \cite{GiamZaic05}, \cite{Jaco75} and \cite{Rowe80}, and their references.} of $\mathfrak{A}$, denoted by $f\equiv_{\mathsf{G}}0$ in $\mathfrak{A}$, if $f(b_{\xi_1},\dots,b_{\xi_n})=0$ for any $b_{\xi_1}\in\mathfrak{A}_{\xi_1}$, $\dots$, $b_{\xi_n}\in\mathfrak{A}_{\xi_n}$. We denote by $\mathsf{T}^{\mathsf{G}}(\mathfrak{A})$ the set of all the graded polynomial identities of $\mathfrak{A}$. It is well-known that $\mathsf{T}^{\mathsf{G}}(\mathfrak{A})$ is a $\mathsf{G}$-graded T-ideal of $\mathbb{F}\langle X^{\mathsf{G}} \rangle$, i.e. it is a graded ideal of $\mathbb{F}\langle X^{\mathsf{G}} \rangle$ which is invariant under all $\mathsf{G}$-graded endomorphisms of $\mathbb{F}\langle X^{\mathsf{G}}\rangle$. Clearly $f\in \mathbb{F}\langle X^{\mathsf{G}}\rangle$ is a graded polynomial identity for $\mathfrak{A}$ iff $f$ belongs to kernels of the all $\mathsf{G}$-graded homomorphisms of $\mathbb{F}\langle X^{\mathsf{G}}\rangle$ in $\mathfrak{A}$. Finally, using Lemma \ref{4.13}, being $\varphi_{\mathfrak{A}}$ the graded homomorphism given in (\ref{4.17}), it is easy to see that $\displaystyle\frac{\mathbb{F}\langle X^{\mathsf{G}} \rangle}{\mathsf{T}^{\mathsf{G}}(\mathfrak{A})}\cong_{\mathsf{G}}\mathsf{im}(\varphi_{\mathfrak{A}})$. In particular, if $\mathfrak{A}$ is finitely generated (as $\mathbb{F}$-algebra) (or $\mathfrak{A}$ has a generator set $S=\bigcup_{\xi\in \mathsf{G}} S_\xi$, where each $S_\xi$ is a countable subset of $\mathfrak{A}_\xi$), then $\mathfrak{A}\cong_{\mathsf{G}} \displaystyle\frac{\mathbb{F}\langle X^{\mathsf{G}} \rangle}{\mathsf{T}^{\mathsf{G}}(\mathfrak{A})}$.


\section{On the Restriction Homomorphism \texorpdfstring{$\mathsf{res}^\mathsf{G}_H$}{resGH} in Second Cohomology of Group}\label{sec3}

In this section we present some notions and properties of the cohomology theory of groups. Here, our goal is to determine suitable conditions to ensure that the restriction homomorphism from $\mathcal{H}^2(\mathsf{G}, \mathsf{M})$ into $\mathcal{H}^2(H, \mathsf{M})$ is surjective, where $H$ is a subgroup of a group $\mathsf{G}$. Although the many results in this section are well-known, we add them here in order to make this paper self-contained. The reader familiar with these definitions and their main properties can continue from \S\ref{resGH}, page \pageref{resGH}, or even from Theorem \ref{2.07}, page \pageref{2.07}. Unless otherwise stated, $\mathsf{G}$ denotes a multiplicative group and $\mathsf{M}$ denotes an abelian additive group that has a structure of a left $\mathsf{G}$-module. All modules in this section are assumed to be left modules.
 
\subsection{Definitions and Properties}

Let us define the \textbf{Second Cohomology Group}. Posteriorly, we will exhibit some important results. The following definition is a generalization of Definition \ref{1.05}. For further details, see \cite{Rotm08}, Section $9.1.2$, p.504.

Let $(\mathsf{G}, \cdot)$ be a multiplicative group, and $(\mathsf{M}, +)$ an abelian additive group. We say that $\mathsf{M}$ is a \textit{left $\mathsf{G}$-module} if there is a well-defined map from $\mathsf{G}\times\mathsf{M}$ into $\mathsf{M}$ which satisfies $r(m_1+m_2)=rm_1+rm_2$, $(r_1+r_2)m=r_1m+r_2m$, $(r_1r_2)m=r_1(r_2m)$, for any $r,r_1,r_2\in \mathbb{Z}\mathsf{G}$ and $m, m_1, m_2\in\mathsf{M}$, where $\mathbb{Z}\mathsf{G}$ is a group ring. We denote by $\eta_\xi$ the element of $\mathbb{Z}\mathsf{G}$ which corresponds to an element $\xi\in\mathsf{G}$. For convenience we assume that $\eta_e m=m$ for any $m\in\mathsf{M}$, where $e$ is the neutral element of $\mathsf{G}$ (i.e. $\mathsf{M}$ is a unitary left $\mathbb{Z}\mathsf{G}$-module).

\begin{definition}\label{1.10}
	Let $\mathsf{G}$ be a group and $\mathsf{M}$ a (left) $\mathsf{G}$-module. A map $\sigma:\mathsf{G}\times\mathsf{G}\rightarrow \mathsf{M}$ is said to be a \textbf{$2$-cocycle}\footnote{If we assume that $\mathsf{M}$ has a multiplicative notation, then $\sigma\in \mathcal{Z}^2(\mathsf{G}, \mathsf{M})$ when $\sigma(\xi,\zeta)\sigma(\xi \zeta,\varsigma)=(\eta_\xi \cdot \sigma(\zeta,\varsigma))\sigma(\xi,\zeta \varsigma)$, and $\varrho\in\mathcal{B}^2(\mathsf{G},\mathsf{M})$ when $\varrho(\xi,\zeta)=\displaystyle\frac{(\eta_\xi\cdot f(\zeta))f(\xi)}{f(\xi \zeta)}$ for some map $f:\mathsf{G}\rightarrow\mathsf{M}$.} if it satisfies the following relation:
	\begin{equation}\nonumber
	\sigma(\xi,\zeta)+\sigma(\xi \zeta,\varsigma)= \eta_\xi\sigma(\zeta,\varsigma)+\sigma(\xi,\zeta \varsigma)
	\ ,	\end{equation}
for any $\xi,\zeta,\varsigma\in\mathsf{G}$. We say that a $2$-cocycle $\rho:\mathsf{G}\times\mathsf{G}\rightarrow \mathsf{M}$ is a \textbf{$2$-coboundary} if there exists a function $f:\mathsf{G}\rightarrow \mathsf{M}$ such that
	\begin{equation}\nonumber
	\rho(\xi,\zeta)= \eta_\xi f(\zeta)-f(\xi \zeta)+f(\xi)
	\end{equation}
	for any $\xi,\zeta\in\mathsf{G}$.
\end{definition}

The previous definition generalizes Definition \ref{1.05}. Notice that if $\mathsf{M}$ is a trivial (left) $\mathsf{G}$-module (also we say ``$\mathsf{G}$ acts trivially on $\mathsf{M}$''), i.e. $\eta_\xi m=m$ for any $\xi\in\mathsf{G}$ and $m\in\mathsf{M}$, it follows that
	\begin{equation}\nonumber
\sigma(\xi,\zeta)+\sigma(\xi\zeta,\varsigma)= \sigma(\zeta,\varsigma)+\sigma(\xi,\zeta\varsigma)
	\end{equation}
for any $2$-cocycle $\sigma$ and $\xi,\zeta,\varsigma\in\mathsf{G}$. Similarly for any $2$-coboundary $\rho$, when $\mathsf{G}$ acts trivially on $\mathsf{M}$, we have that $\rho(\xi,\zeta)=f(\zeta)-f(\xi\zeta)+f(\xi)$ for any $\xi,\zeta\in \mathsf{G}$, for some maps $f:\mathsf{G}\rightarrow\mathsf{M}$.

\begin{proposition}\label{1.07}
	Let $\mathsf{G}$ be a group, $\mathsf{M}$ a (left) $\mathsf{G}$-module and $\sigma: \mathsf{G}\times \mathsf{G}\longrightarrow \mathsf{M}$ a $2$-cocycle. If $e\in \mathsf{G}$ is the neutral element (unit of $\mathsf{G}$), then
	\begin{equation}\nonumber
	\sigma(\xi,e)=\eta_\xi\sigma(e,e) \ \mbox{and} \ \sigma(e,\zeta)=\sigma(e,e),
	\end{equation}
for any $\xi,\zeta\in \mathsf{G}$. In particular, if $\mathsf{G}$ acts trivially on $\mathsf{M}$, then $\sigma(\xi,e)=\sigma(e,\zeta)=\sigma(e,e)$ and $\sigma(\xi,\xi^{-1})=\sigma(\xi^{-1},\xi)$ for any $\xi, \zeta\in \mathsf{G}$.
\end{proposition}
\begin{proof}
 To prove that $\sigma(e,\zeta)=\sigma(e,e)$ for any $\zeta\in \mathsf{G}$, it is enough to put $\xi=e$ in Definition \ref{1.10}. Now, take $\zeta=\varsigma=e$ in Definition \ref{1.10} to show that $\sigma(\xi,e)=\eta_\xi\sigma(e,e)$ for any $\zeta\in \mathsf{G}$. Finally, assuming that $\mathsf{G}$ acts trivially on $\mathsf{M}$, and putting $\zeta=\xi^{-1}$ and $\varsigma=\xi$ in Definition \ref{1.10}, we can ensure that $\sigma(\xi,\xi^{-1})=\sigma(\xi^{-1},\xi)$, for any $\xi\in \mathsf{G}$.
\end{proof}

\begin{definition}
	Let $\mathsf{G}$ be a group and $\mathsf{M}$ a $\mathsf{G}$-module. We define
	\begin{equation}\nonumber
	\mathcal{Z}^2(\mathsf{G}, \mathsf{M})=\{\mbox{all the 2-cocycles } \sigma:\mathsf{G}\times \mathsf{G} \rightarrow \mathsf{M} \} 
	\end{equation}
and
	\begin{equation}\nonumber
	\mathcal{B}^2(\mathsf{G}, \mathsf{M})=\{\mbox{all the 2-coboundaries } \rho:\mathsf{G}\times \mathsf{G} \rightarrow \mathsf{M} \} \ .
	\end{equation}
\end{definition}
	
Given $\sigma, \rho\in\mathcal{Z}^2(\mathsf{G}, \mathsf{M})$, we define\footnote{If we assume that $\mathsf{M}$ has a multiplicative notation, then $\sigma\rho: (\xi,\zeta)\mapsto\sigma(\xi,\zeta)\rho(\xi,\zeta)$ for any $\xi,\zeta\in\mathsf{G}$.}
	\begin{equation}\nonumber
\sigma + \rho: (\xi,\zeta)\mapsto  \sigma(\xi,\zeta) + \rho(\xi,\zeta)
	\end{equation}
for any $\xi,\zeta\in\mathsf{G}$. In \cite{Rotm08}, Proposition 9.11 ensures that $\mathcal{Z}^2(\mathsf{G}, \mathsf{M})$ is an abelian group, and $\mathcal{B}^2(\mathsf{G}, \mathsf{M})$ is a subgroup of $\mathcal{Z}^2(\mathsf{G}, \mathsf{M})$, with respect to this operation. Observe that, since $\mathcal{B}^2(\mathsf{G},\mathsf{M})$ is a group, the inverse element of a $2$-coboundary $\rho$ defined by $\rho(\xi,\zeta)= \eta_\xi f(\zeta)-f(\xi\zeta)+f(\xi)$ for some $f:\mathsf{G}\times\mathsf{G}\rightarrow \mathsf{M}$ is given by 
	\begin{equation}\nonumber
(-\rho)(\xi,\zeta)= -f(\xi)+ f(\xi\zeta)-\eta_\xi f(\zeta) \ .
	\end{equation}

\begin{definition}
	The \textbf{second cohomology group} of $\mathsf{G}$ is defined as a quotient group
	\begin{equation}\nonumber
	\mathcal{H}^2(\mathsf{G}, \mathsf{M}) \coloneqq \frac{\mathcal{Z}^2(\mathsf{G}, \mathsf{M})}{\mathcal{B}^2(\mathsf{G}, \mathsf{M})}
	\ .
	\end{equation}
\end{definition}

The elements of $\mathcal{H}^2(\mathsf{G}, \mathsf{M})$ are denoted by $[\sigma]$, where $\sigma\in \mathcal{Z}^2(\mathsf{G}, \mathsf{M})$. Note that, given $\sigma, \rho\in\mathcal{Z}^2(\mathsf{G}, \mathsf{M})$, $[\sigma]=[\rho]$ in $\mathcal{H}^2(\mathsf{G}, \mathsf{M})$ when there exists $\varrho\in\mathcal{B}^2(\mathsf{G}, \mathsf{M})$ such that $\sigma=\varrho+\rho$. In this case, we say that $\sigma$ and $\rho$ are \textbf{equivalent}.

Observe that Example \ref{1.06} provides a $2$-cocycle $\sigma\in \mathcal{Z}^2(\mathbb{Z}_2\times\mathbb{Z}_2, \mathbb{C}^*)$ which is not equivalent to the trivial $2$-cocycle, because $\sigma((\bar1,\bar0),(\bar1,\bar1))=-1\neq 1 = \sigma((\bar1,\bar1),(\bar1,\bar0))$. On the other hand, being $H=\{(\bar0,\bar0),(\bar0,\bar1)\}$ and $N=\{(\bar0,\bar0),(\bar1,\bar1)\}$, both subgroups of $\mathbb{Z}_2\times\mathbb{Z}_2$, it is obvious that $[\sigma]_{|H\times H}= [1]_{|H\times H}$ and $[\sigma_{N}]_{|N\times N}\neq [1]_{|N\times N}$.

Let us present some basic results, which relate the second cohomology group and the orders of groups and subgroups. These results and some other facts about the second cohomology group can be found in \cite{Brow12}, \cite{Grue06}, \cite{Jaco12.2}, \cite{Rotm08} and \cite{Verm03}.

\begin{lemma}[Theorem 6.14, \cite{Jaco12.2}]\label{2.11}
	Let $\mathsf{G}$ be a finite group and $\mathsf{M}$ a $\mathsf{G}$-module. Every element of $\mathcal{H}^2(\mathsf{G}, \mathsf{M})$ has finite order, which is a divisor of $|\mathsf{G}|$.
\end{lemma}

In \cite{Rotm08}, it is shown that if $\mathsf{M}$ is a finitely generated $\mathsf{G}$-module, then $\mathcal{H}^2(\mathsf{G}, \mathsf{M})$ is finite (see Corollary 9.41). It is also proved (see Corollary 9.90) the following items:
	\begin{enumerate}[{\it i)}]
		\item There is an injection $\psi:\mathcal{H}^2(\mathsf{G}, \mathsf{M}) \longrightarrow \bigoplus_p \mathcal{H}^2(\mathsf{G}_p, \mathsf{M})$, where $\mathsf{G}_p$ is a Sylow $p$-subgroup of $\mathsf{G}$, $p$ is a prime divisor of $|\mathsf{G}|$;
		\item If $\mathcal{H}^2(\mathsf{G}_p, \mathsf{M})=\{0\}$ for all Sylow $p$-subgroups, then $\mathcal{H}^2(\mathsf{G}, \mathsf{M})=\{0\}$.
	\end{enumerate}
Already in \cite{Verm03}, it is proven in Theorem 11.8.18 that if $H\unlhd\mathsf{G}$ with index $[\mathsf{G}: H]=m$ coprime to the order of $H$, then $\mathcal{H}^2(\mathsf{G}, \mathsf{M})\cong \mathcal{H}^2(H, \mathsf{M})^\mathsf{G}\oplus \mathcal{H}^2(\mathsf{G}/H, \mathsf{M}^H)$, where $\mathsf{M}^H=\{m\in\mathsf{M} : \eta_\zeta m=m, \forall \zeta\in H \}$, and $\mathcal{H}^2(H, \mathsf{M})^\mathsf{G}=\{\sigma\in\mathcal{H}^2(H, \mathsf{M}) : \eta_\xi \sigma=\sigma, \forall \xi\in \mathsf{G} \}$. Here, it is also proved (see Theorem 12.1.3) that if $H$ is a subgroup of $\mathsf{G}$ and $\mathsf{M}$ is an $H$-module, then $\mathcal{H}^2(\mathsf{G},\mathsf{Hom}_{\mathbb{Z} H}(\mathbb{Z} \mathsf{G}, \mathsf{M}))$ and $\mathcal{H}^2(H, \mathsf{M})$ are isomorphic, where $\mathsf{Hom}_{\mathbb{Z} H}(\mathbb{Z} \mathsf{G}, \mathsf{M})$ is the group of $\mathbb{Z} H$-homomorphisms from $\mathbb{Z} \mathsf{G}$ into $\mathsf{M}$. This last result is known as Shapiro Lemma (see \cite{Grue06}, Theorem 6.3.2, p.92).
%

The proposition below improves Exercise 6.10.3, in \cite{Jaco12.2}, p.369. Roughly speaking, for all $\lambda\in \mathbb{F}$, a field $\mathbb{F}$ contains $\sqrt[n]{\lambda}$ iff $\mathbb{F}$ contains a root of the polynomial $p_\lambda(x)=x^n-\lambda$. In particular, any algebraically closed field $\mathbb{F}$ contains $\sqrt[n]{\lambda}$ for all $\lambda\in \mathbb{F}$. Hence, we write $\gamma=\sqrt[n]{\lambda}$ to denote that $\gamma^n=\lambda$.

\begin{proposition}\label{2.01}
	Let $\mathsf{G}$ be a finite group of order $n$, $\mathbb{F}$ a field such that $\sqrt[n]{\lambda} \in \mathbb{F}$ for all $\lambda\in \mathbb{F}$, and $\mathcal{H}^2(\mathsf{G}, \mathbb{F}^*)$ the second cohomology group of $\mathsf{G}$ with coefficients in the multiplicative group $\mathbb{F}^*$, where $\mathsf{G}$ acts trivially on $\mathbb{F}^*$. For any $[\varrho]\in\mathcal{H}^2(\mathsf{G}, \mathbb{F}^*)$, the representative $2$-cocycle $\varrho$ can be chosen to have values that are $n$-th roots of unity. Therefore, $\mathcal{H}^2(\mathsf{G}, \mathbb{F}^*)$ is finite. 
\end{proposition}
\begin{proof}
	Let us assume that $\mathcal{H}^2(\mathsf{G}, \mathbb{F}^*)$ is a group with multiplicative notation. Take any $[\sigma]\in\mathcal{H}^2(\mathsf{G}, \mathbb{F}^*)$, where $\sigma\in\mathcal{Z}^2(\mathsf{G}, \mathbb{F}^*)$ is a $2$-cocycle. Since $\mathsf{G}$ is a group of order $n$, by Lemma \ref{2.11}, it follows that $[\sigma]^{n}$ is the neutral element of $\mathcal{H}^2(\mathsf{G}, \mathbb{F}^*)$, and so $[\sigma^{n}]=[\sigma]^{n}=[1]$ which implies $\sigma^n\in\mathcal{B}^2(\mathsf{G}, \mathbb{F}^*)$. Let $f: \mathsf{G}\rightarrow \mathbb{F}^*$ be a map and $\varrho\in\mathcal{B}^2(\mathsf{G}, \mathbb{F}^*)$ a $2$-coboundary such that 
	\begin{equation}\nonumber
\varrho(\xi,\zeta)=\displaystyle\frac{f(\xi\zeta)}{f(\xi)f(\zeta)} \quad\mbox{ and }\quad	\varrho(\xi,\zeta)=\sigma^{n}(\xi,\zeta)=(\sigma(\xi,\zeta))^n
	\end{equation}
for any $\xi,\zeta\in\mathsf{G}$. For each $\xi\in\mathsf{G}$, by hypothesis, it follows that $\sqrt[n]{f(\xi)}\in\mathbb{F}^*$, 
and thus, we can consider the map $\hat{f}:\mathsf{G}\rightarrow\mathbb{F}^*$ defined by $\hat{f}(\xi)=\sqrt[n]{f(\xi)}$ for any $\xi\in\mathsf{G}$. Put $\hat{\varrho}(\xi,\zeta)=\displaystyle\frac{\hat{f}(\xi)\hat{f}(\zeta)}{\hat{f}(\xi\zeta)}$. It is not difficult to see that $\hat{\varrho}\in\mathcal{B}^2(\mathsf{G}, \mathbb{F}^*)$, and
	\begin{equation}\label{2.13}
	((\hat{\varrho}\sigma)(\xi,\zeta))^n=\hat{\varrho}^{n}(\xi,\zeta)\sigma^{n}(\xi,\zeta)=(\varrho(\xi,\zeta))^{-1}\sigma^{n}(\xi,\zeta)=1
	\end{equation}
for any $\xi,\zeta\in\mathsf{G}$. Notice that $\hat{\varrho}\sigma$ is an element of $\mathcal{Z}^2(\mathsf{G}, \mathbb{F}^*)$, and $[\hat{\varrho}\sigma]=[\sigma]$, since $\hat{\varrho}\in\mathcal{B}^2(\mathsf{G}, \mathbb{F}^*)$. Consider now $\hat{\sigma}=\hat{\varrho}\sigma$. By (\ref{2.13}), for each $\xi,\zeta\in\mathsf{G}$, it follows that $\hat{\sigma}(\xi,\zeta)$ is an $n$-th root of unit. Therefore, we have $\hat{\sigma}$ is a representative of $[\sigma]$ which has values that are $n$-th roots of the unit.
	
	Finally, since $\mathbb{F}$ has at most $n$ roots of unit and $\mathsf{G}$ is a finite group, it follows that $\mathcal{H}^2(\mathsf{G}, \mathbb{F}^*)$ is a finite group.
\end{proof}

Under the assumptions of the previous proposition, since there are at most $n^{n^2}$ possibilities for functions from $\mathsf{G}\times\mathsf{G}$ into $\{\mbox{all }n\mbox{-th roots of unity}\}\subset\mathbb{F}^*$, it follows that $|\mathcal{H}^2(\mathsf{G}, \mathbb{F}^*)|\leq n^{n^2}$. In particular, assuming that $\mathsf{G}$ acts trivially on $\mathbb{F}^*$, and using the proof of the previous proposition and Proposition \ref{1.07}, we can improve this estimate to
	\begin{equation}\nonumber
|\mathcal{H}^2(\mathsf{G}, \mathbb{F}^*)|\leq n^{\frac{n(n-1)}{2}+1} \ .
	\end{equation}
In fact, being $\mathsf{G}=\{e, \xi_1,\dots, \xi_{n-1}\}$, fix $\sigma(e,e)=\lambda_0\in\mathbb{F}^*$ and $\sigma(\xi_i,\xi_j)=\lambda_{ij}\in\mathbb{F}^*$ for all $1\leq i\leq j \leq n-1$. By Proposition \ref{1.07}, we have $\sigma(\xi,e)=\sigma(e,\zeta)=\sigma(e,e)$ for any $\xi,\zeta\in \mathsf{G}$, and hence $\sigma(e,\xi_i)=\sigma(\xi_i,e)=\lambda_0$ for all $1\leq i , j \leq n-1$. From this, the multiplication table of $\sigma$ is (partially) filled as
	\begin{equation}\nonumber
		\begin{tabular}{c|ccccc}
$\sigma$ & $e$ & $\xi_1$ & $\xi_2$ & $\cdots$ & $\xi_{n-1}$\\ \hline 
$e$ & $\lambda_0$ & $\lambda_0$ & $\lambda_0$ & $\cdots$ & $\lambda_0$\\ 
$\xi_1$ & $\lambda_0$ & $\lambda_{11}$ & $\lambda_{12}$ & $\cdots$ & $\lambda_{1(n-1)}$\\ 
$\xi_2$ & $\lambda_0$ &  & $\lambda_{22}$ & $\cdots$ & $\lambda_{2(n-1)}$ \\ 
$\vdots$ & $\vdots$ &  &  & $\ddots$ & $\vdots$ \\ 
$\xi_{n-1}$ & $\lambda_0$ & & &  & $\lambda_{(n-1)(n-1)}$ \\ 
		\end{tabular} \ .
	\end{equation}
 By Proposition \ref{2.01}, we can assume, without loss of generality, that $\sigma(\xi,\zeta)\in\mathbb{F}^*$ is an $n$-th root of unit, for any $\xi,\zeta\in \mathsf{G}$. On the other hand, taking $\xi=\xi_j$, $\zeta=\xi_i$ and $\varsigma=(\xi_j \xi_i)^{-1}$ in Definition \ref{1.10}, we have that 
\begin{equation}\nonumber
	\sigma(\xi_j,\xi_i)=\frac{\sigma(\xi_j,\xi_j^{-1})}{\sigma(\xi_j \xi_i,(\xi_j \xi_i)^{-1})}\cdot\sigma(\xi_i,(\xi_j \xi_i)^{-1})=\frac{\sigma(\xi_j^{-1} , \xi_j)}{\sigma((\xi_j \xi_i)^{-1}, \xi_j \xi_i)}\cdot\sigma(\xi_i,(\xi_j \xi_i)^{-1})
	\ ,	\end{equation}
where we use that $\sigma(\xi,\xi^{-1})=\sigma(\xi^{-1},\xi)$ for any $\xi\in \mathsf{G}$ (Proposition \ref{1.07}), and so $\sigma(\xi_j,\xi_i)$ depends on $\lambda_{il}$'s, $\lambda_{ll}$'s and $\lambda_0$, for all $1\leq i< j \leq n-1$. Consequently, we can build the following table 

	\begin{equation}\nonumber
		\begin{tabular}{c|ccccc}
$\sigma$ & $e$ & $\xi_1$ & $\xi_2$ & $\cdots$ & $\xi_{n-1}$\\ \hline 
$e$ & $n$ & $1$ & $1$ & $\cdots$ & $1$\\ 
$\xi_1$ & 1 & $n$ & $n$ & $\cdots$ & $n$\\ 
$\xi_2$ & $1$ & $1$  & $n$ & $\cdots$ & $n$ \\ 
$\vdots$ & $\vdots$ & $\vdots$ & $\ddots$ & $\ddots$ & $\vdots$ \\ 
$\xi_{n-1}$ & $1$ & $1$ & $\cdots$ & $1$  & $n$ \\ 
		\end{tabular} \ ,
	\end{equation}
which represents the total number of possibilities (possible combinations) for each $\sigma(\xi,\zeta)$, for $\xi,\zeta\in\mathsf{G}$.


\subsection{Restriction \texorpdfstring{$\mathsf{res}^\mathsf{G}_H$}{resGH}}\label{resGH}

In the previous subsection we present some results which relate the second cohomology groups of $\mathsf{G}$ and its subgroups. In what follows, we study a way to ensure the existence of a surjective homomorphism from $\mathcal{H}^2(\mathsf{G},\mathsf{M})$ onto $\mathcal{H}^2(H,\mathsf{M})$, for a subgroup $H$ of $\mathsf{G}$. Here, to simplify notation, we write $\xi m\coloneqq \eta_\xi m$ for any $\xi\in\mathsf{G}$ and $m\in\mathsf{M}$.

Consider the pair $(H,\mathsf{M})$, where $H$ is a subgroup of a group $\mathsf{G}$ and $\mathsf{M}$ is a $\mathsf{G}$-module. Given an element $\xi\in\mathsf{G}$, consider the maps $\gamma_\xi: \zeta \mapsto \xi\zeta\xi^{-1}$ and $f_\xi: m\mapsto \xi^{-1} m$. Observe that $\gamma_\xi$ is an isomorphism of groups (of $H$ in $\xi H\xi^{-1}$), and $f_\xi$ is an automorphism of the abelian group $\mathsf{M}$, for any $\xi\in\mathsf{G}$. We denote by $c(\xi)$ the bijection
	\begin{equation}\nonumber
\begin{array}{cccl}
c(\xi)\coloneqq (\gamma_\xi, f_\xi):& (H, \mathsf{M}) &\longrightarrow &  (\xi H \xi^{-1}, \mathsf{M}) \\
& (\zeta, m) & \longmapsto &  c(\xi)((\zeta, m))=(\gamma_\xi(\zeta), f_\xi(m))
\end{array}
\ . 
	\end{equation}
This application defines an action of $\mathsf{G}$ on $(H, \mathsf{M})$, called \textbf{conjugation action}. Note that for any $\zeta\in H$ and $m\in\mathsf{M}$, we have:
	\begin{equation}\nonumber
f_\xi(\gamma_\xi(\zeta)m)=f_\xi(\xi\zeta\xi^{-1}m)=\xi^{-1}\xi\zeta\xi^{-1}m=\zeta\xi^{-1}m=\zeta f_\xi(m)
\ ,
	\end{equation}
In this case, we say that the pair $(\gamma_\xi, f_\xi)$ is a \textbf{compatible pair}.

It is well-known (see \S III.8 in \cite{Brow12}, or \S 9.5 in \cite{Rotm08}, or \S 11.8 in \cite{Verm03}) that there exists an isomorphism of groups
	\begin{equation}\nonumber
c(\xi)^*: \mathcal{H}^2(\xi H \xi^{-1},\mathsf{M}) \longrightarrow \mathcal{H}^2(H,\mathsf{M})
	\end{equation}
induced by the conjugation map $c(\xi)=(\gamma_\xi, f_\xi)$. For more details about how $c(\xi)$ induces a homomorphism, see \S 9.5 in \cite{Rotm08}.

Define a map from $\mathsf{G}\times\mathcal{H}^2(H,\mathsf{M})$ to $\mathcal{H}^2(\xi H \xi^{-1},\mathsf{M})$ by 
\begin{equation}\label{2.02}
\xi \cdot\sigma=(c(\xi)^*)^{-1} (\sigma) \in \mathcal{H}^2(\xi H \xi^{-1},\mathsf{M}) \ ,
\end{equation}
which is induced by the conjugation action of $\mathsf{G}$ on $(H,\mathsf{M})$. Unless otherwise stated, we denote by $\xi\sigma$ the product defined in (\ref{2.02}).

Now, by (\ref{2.02}), $H$ acts trivially on $\mathcal{H}^2(H,\mathsf{M})$, and if $H$ is a normal subgroup of $\mathsf{G}$, then the conjugation action of $\mathsf{G}$ on $(H,\mathsf{M})$ defined above induces an action of $\mathsf{G}/H$ on $\mathcal{H}^2(H, \mathsf{M})$. For more details, see \S 9.5 in \cite{Rotm08}, or  \S III.9 in \cite{Brow12}. More precisely, we have the following result:

\begin{lemma}[Corollaries 8.3 and 8.4, \S III.8, \cite{Brow12}]\label{2.03}
	Let $H$ be a subgroup of $\mathsf{G}$. Then the conjugation action of $H$ on $(H, \mathsf{M})$ induces an action of $H$ on $\mathcal{H}^2(H, \mathsf{M})$, which is trivial. In addition, if $H\unlhd \mathsf{G}$ and $\mathsf{M}$ is a $\mathsf{G}$-module, then the conjugation action of $\mathsf{G}$ on $(H,\mathsf{M})$ induces an action of $\mathsf{G}/H$ on $\mathcal{H}^2(H, \mathsf{M})$.
\end{lemma}

Consequently, the next result it is immediate.

\begin{lemma}[Exercise 1, \S III.8, \cite{Brow12}]\label{2.04}
	If $H$ is central in $\mathsf{G}$ and $\mathsf{M}$ is an abelian group with the trivial $\mathsf{G}$-action, then $\mathsf{G}/H$ acts trivially on $\mathcal{H}^2(H, \mathsf{M})$.
\end{lemma}

Other proofs of Lemmas \ref{2.03} and \ref{2.04} can also be found in \cite{Rotm08}, written as Lemma 9.82. 

Let us show now a relation between $\mathcal{H}^2(H, \mathsf{M})$ and $\mathcal{H}^2(\mathsf{G}, \mathsf{M})$ for any subgroup $H$ of a group $\mathsf{G}$. But firstly, we need some definitions. The next definition can be founded in \cite{Rotm08}, page 566.

\begin{definition}
	Let $H$ be a subgroup of a group $\mathsf{G}$, $\mathsf{M}$ a $\mathsf{G}$-module, $\mathfrak{i}$ the inclusion map from $H$ into $\mathsf{G}$, and $1_{\mathsf{M}}$ is the identity map on $\mathsf{M}$. The pair $(\mathfrak{i},1_\mathsf{M})$ is compatible, i.e. $1_\mathsf{M}(\mathfrak{i}(\zeta)m)=\zeta m=\zeta 1_\mathsf{M}(m)$ for any $\zeta\in H$ and $m\in\mathsf{M}$. The homomorphism induced by the pair $(\mathfrak{i},1_\mathsf{M})$ is denoted by $\mathsf{res}^\mathsf{G}_H: \mathcal{H}^2(\mathsf{G}, \mathsf{M}) \longrightarrow \mathcal{H}^2(H, \mathsf{M})$ and is called \textbf{restriction homomorphism}.
\end{definition}

It is well-known that the restriction homomorphism is defined as follows: if $\sigma: \mathsf{G}\times\mathsf{G} \longrightarrow \mathsf{M}$ is a $2$-cocycle, then $\sigma$ restricted to $H\times H$, which we denote by $\sigma_H$, is also a $2$-cocycle and $\mathsf{res}^\mathsf{G}_H([\sigma])=[\sigma_{H}]$ in $\mathcal{H}^2(H,\mathsf{M})$ (see \cite{Verm03}, \S11.8, or \cite{Rotm08}, \S9.5). From now on, let us write $\sigma_H$ to mean the restriction of $\sigma$ to $H\times H$, i.e. $\sigma_H(\xi,\zeta)=\sigma(\xi,\zeta)$ for any $\xi,\zeta\in H$.

\begin{lemma}[Lemma 11.8.15, \cite{Verm03}]\label{2.05}
	Let $\mathsf{G}$ be a group, $H$ a subgroup of $\mathsf{G}$ and $\mathsf{M}$ a $\mathsf{G}$-module. If $H\unlhd\mathsf{G}$, then 
	\begin{equation}\nonumber
	\mathsf{res}^\mathsf{G}_H(\mathcal{H}^2(\mathsf{G}, \mathsf{M}))\subseteq \mathcal{H}^2(H, \mathsf{M})^\mathsf{G}
	\ ,
	\end{equation}
where $\mathcal{H}^2(H, \mathsf{M})^\mathsf{G}=\{\sigma\in\mathcal{H}^2(H, \mathsf{M}) : \xi \cdot \sigma= \sigma, \forall \xi\in\mathsf{G}\}$, and $\xi\cdot \sigma$ is the action defined in (\ref{2.02}).
\end{lemma}

Another proof of the previous lemma is given in Corollary 9.83, \cite{Rotm08}.

Let us now use the next result to define a homomorphism of $\mathcal{H}^2(H, \mathsf{M})$ to $\mathcal{H}^2(\mathsf{G}, \mathsf{M})$, called \textbf{corestriction} (also called \textbf{transfer}).

\begin{lemma}\label{2.06}
	Let $H$ be a subgroup of finite index in a group $\mathsf{G}$, and $\mathsf{M}$ a $\mathsf{G}$-module. There exists a homomorphisms of groups
	\begin{equation}\nonumber
	\mathsf{cores}^\mathsf{G}_H: \mathcal{H}^2(H, \mathsf{M}) \longrightarrow \mathcal{H}^2(\mathsf{G}, \mathsf{M})
	\end{equation}
satisfying
	\begin{enumerate}[i)]
		\item $\mathsf{cores}^\mathsf{G}_H \mathsf{res}^\mathsf{G}_H \sigma= [\mathsf{G}:H] \sigma$, for all $\sigma\in \mathcal{H}^2(\mathsf{G}, \mathsf{M})$;
		\item If $H\unlhd\mathsf{G}$, then $\mathsf{res}^\mathsf{G}_H \mathsf{cores}^\mathsf{G}_H \rho= \sum_{\xi \in\mathsf{G}/H} \xi\rho$, for all $\rho\in \mathcal{H}^2(H, \mathsf{M})$,
	\end{enumerate}
where $\xi\rho=\xi\cdot \rho$ is the action defined in (\ref{2.02}).
\end{lemma}
\begin{proof}
The existence of $\mathsf{cores}$ is ensured by Proposition 9.87 in \cite{Rotm08} (see also \S III.9 in \cite{Brow12} or \S11.8 in \cite{Verm03}). The properties \textit{i)} and \textit{ii)} are proved in \cite{Brow12}, \S III.8, Proposition 9.5, p.82  (see also Theorem 9.88 in \cite{Rotm08}, or Theorem 11.8.6 in \cite{Verm03}).
\end{proof}

Observe that Lemma \ref{2.03} ensures that $\sum_{\xi \in\mathsf{G}/H} \xi\rho$ (item {\it ii)} in Lemma \ref{2.06}) is well defined, for all $\rho\in \mathcal{H}^2(H, \mathsf{M})$. For more details about the homomorphism $\mathsf{cores}$, see \S III.9 in \cite{Brow12}, or \S9.6 in \cite{Rotm08}, or \S11.8 in \cite{Verm03}.
	
Suppose that $\mathsf{M}$ is a $\mathsf{G}$-module. We say that $\lambda\in\mathbb{Z}$, $\lambda>0$, is \textbf{invertible} in $\mathsf{M}$ if the multiplication by $\lambda$ is an automorphism of $\mathsf{M}$, i.e. the map given by
	\begin{equation}\nonumber
\begin{array}{cccl}
d_\lambda:& \mathsf{M} &\longrightarrow &  \mathsf{M} \\
& m & \longmapsto & d_\lambda (m)=\underbrace{m+\cdots+m}_{\lambda-times}
\end{array}
	\end{equation}
is an isomorphism of $\mathsf{G}$-modules. In this case, by item \textit{i)} in the previous lemma, if $[\mathsf{G}:H]< \infty$ and $[\mathsf{G}: H]$ is invertible in $\mathsf{M}$, then $\mathsf{res}^\mathsf{G}_H$ is an injective map. Let us show now that, under suitable conditions, $\mathsf{res}^\mathsf{G}_H$ is surjective.

\begin{theorem}\label{2.07}
	If $[\mathsf{G}:H]< \infty$, $H$ is central in $\mathsf{G}$, and $\mathsf{M}$ is an abelian group with the trivial $\mathsf{G}$-action, then $\mathsf{res}^\mathsf{G}_H \mathsf{cores}^\mathsf{G}_H \rho= [\mathsf{G}: H] \rho$ for all $\rho\in \mathcal{H}^2(H, \mathsf{M})$. In addition, if $[\mathsf{G}: H]$ is invertible in $\mathsf{M}$, then  $\mathsf{res}^\mathsf{G}_H$ is a surjection from $\mathcal{H}^2(\mathsf{G}, \mathsf{M})$ into $\mathcal{H}^2(H, \mathsf{M})$.
\end{theorem}
\begin{proof}
	Since $H$ is central in $\mathsf{G}$, it follows that $H\unlhd\mathsf{G}$ and, by  item \textit{ii)} in Lemma \ref{2.06}, $\mathsf{res}^\mathsf{G}_H \mathsf{cores}^\mathsf{G}_H \rho= \sum_{\xi\in\mathsf{G}/H} \xi\rho$ for all $\rho\in \mathcal{H}^2(H, \mathsf{M})$. On the other hand, since $\mathsf{G}/H$ acts trivially on $\mathcal{H}^2(H, \mathsf{M})$ (Lemma \ref{2.04}), it follows that 
	\begin{equation}\nonumber
	\sum_{\xi\in\mathsf{G}/H} \xi\rho=\sum_{\xi\in\mathsf{G}/H} \rho=[\mathsf{G}: H]\rho
	\ , \ \rho\in \mathcal{H}^2(H, \mathsf{M})
	\ .
	\end{equation}
Therefore, $\mathsf{res}^\mathsf{G}_H \mathsf{cores}^\mathsf{G}_H \rho= [\mathsf{G}: H]\rho$ for all $\rho\in \mathcal{H}^2(H, \mathsf{M})$.
	
Suppose now that $[\mathsf{G}: H]$ is invertible in $\mathsf{M}$. Fixed $\sigma\in \mathcal{H}^2(H, \mathsf{M})$, we have $d_{[\mathsf{G}: H]}^{-1}(\sigma(\xi,\zeta))\in\mathsf{M}$ for any $\xi,\zeta\in H$. Consider the map
	\begin{equation}\nonumber
	\varrho \coloneqq d_{[\mathsf{G}: H]}^{-1}\sigma: (\xi,\zeta)\mapsto d_{[\mathsf{G}: H]}^{-1}(\sigma(\xi,\zeta))
	\ , 
	\end{equation}
from $H\times H$ to $\mathsf{M}$. As $\mathsf{G}/H$ acts trivially on $\mathcal{H}^2(H, \mathsf{M})$ and any $2$-cocycle satisfies the equality in Definition \ref{1.10}, we have 
	\begin{equation}\nonumber
	\varrho(\xi,\zeta)+\varrho(\xi\zeta,\varsigma)= \varrho(\zeta,\varsigma)+\varrho(\xi,\zeta\varsigma)
	\ , 
	\end{equation}
	for any $\xi,\zeta,\varsigma\in H$, since $\mathsf{M}$ is an abelian group, and so $d_{[\mathsf{G}: H]}^{-1}$ is a homomorphism of groups. Hence, it follows that $d_{[\mathsf{G}: H]}^{-1}\sigma=\varrho\in \mathcal{H}^2(H, \mathsf{M})$. From the first part of this proof, it follows that
	\begin{equation}\nonumber
	\mathsf{res}^\mathsf{G}_H \mathsf{cores}^\mathsf{G}_H \varrho= [\mathsf{G}: H]\varrho= [\mathsf{G}: H][\mathsf{G}: H]^{-1}\sigma=\sigma
	\ ,
	\end{equation}
	and so $\sigma=\mathsf{res}^\mathsf{G}_H (\mathsf{cores}^\mathsf{G}_H \varrho)\in\mathsf{res}^\mathsf{G}_H\left(\mathcal{H}^2(\mathsf{G}, \mathsf{M}) \right)$. Therefore, $\mathsf{res}^\mathsf{G}_H$ is surjective.
\end{proof}

Let $H$ be a central subgroup of $\mathsf{G}$, and $\mathsf{M}$ an abelian group. Observe that when $\mathsf{M}$ is a multiplicative group, then the last result means that $\mathsf{res}^\mathsf{G}_H \mathsf{cores}^\mathsf{G}_H \rho= \rho^{[\mathsf{G}: H]}$ for all $\rho\in \mathcal{H}^2(H, \mathsf{M})$, and in the case when $[\mathsf{G}: H]$ is invertible in $\mathsf{M}$, then we have that $d_{[\mathsf{G}: H]}^{-1}(\rho(\xi,\zeta))=\sqrt[{[\mathsf{G}: H]}]{\rho(\xi,\zeta)}\in\mathsf{M}$ for any $\xi,\zeta \in\mathsf{G}$. In particular, when $\mathsf{M}=\mathbb{F}^*$ is the multiplicative group of a field $\mathbb{F}$, where $\mathsf{G}$ acts trivially on $\mathsf{M}$, then $p=[\mathsf{G}:H]$ invertible in $\mathbb{F}$ means that $\sqrt[p]{\lambda}\in\mathbb{F}$ for any $\lambda\in\mathbb{F}$.

Now, by Lemma \ref{2.06} and Theorem \ref{2.07}, it follows that 
	\begin{equation}\nonumber
|\mathcal{H}^2(H, \mathsf{M})|=|\mathsf{res}^\mathsf{G}_H \left(\mathcal{H}^2(\mathsf{G}, \mathsf{M})\right)|\ ,
	\end{equation}
when $[\mathsf{G}:H]$ is invertible in $\mathsf{M}$. Consequently, we have
\begin{equation}\label{2.08}
\{[\sigma_H] : [\sigma]\in\mathcal{H}^2(\mathsf{G}, \mathsf{M})\}\subseteq \mathcal{H}^2(H, \mathsf{M})
\ .
\end{equation}
The above result gives enough conditions to ensure the equality in (\ref{2.08}), since $\mathsf{res}^\mathsf{G}_H \left(\mathcal{H}^2(\mathsf{G}, \mathsf{M})\right) \subseteq \mathcal{H}^2(H, \mathsf{M})$. Thus, we have the following result.

\begin{corollary}\label{2.09}
	Let $H$ be a central subgroup of $\mathsf{G}$, $[\mathsf{G}:H]< \infty$, and $\mathsf{M}$ an abelian group with the trivial $\mathsf{G}$-action. If $[\mathsf{G}: H]$ is invertible in $\mathsf{M}$, then 
	\begin{equation}\nonumber
	\mathcal{H}^2(H, \mathsf{M})=\mathsf{res}^\mathsf{G}_H \left(\mathcal{H}^2(\mathsf{G}, \mathsf{M})\right)
	\ .
	\end{equation}
\end{corollary}
\begin{proof}
	By Theorem \ref{2.07}, it follows that $\mathcal{H}^2(H, \mathsf{M})\subseteq\mathsf{res}^\mathsf{G}_H \left(\mathcal{H}^2(\mathsf{G}, \mathsf{M})\right)$. Therefore, the result follows, since $\mathsf{res}^\mathsf{G}_H \left(\mathcal{H}^2(\mathsf{G}, \mathsf{M})\right)\subseteq\mathcal{H}^2(H, \mathsf{M})$ (by the definition of homomorphism $\mathsf{res}^\mathsf{G}_H$).
\end{proof}

The previous corollary means that, when $H$ is a central subgroup of $\mathsf{G}$ of finite index and $\mathsf{G}$ acts trivially on an abelian group $\mathsf{M}$, given $[\sigma]\in \mathcal{H}^2(H, \mathsf{M})$, there exists $[\varrho]\in\mathcal{H}^2(\mathsf{G}, \mathsf{M})$ such that $\varrho_H=\sigma$. Therefore, we deduce that
	\begin{equation}\nonumber
\mathcal{H}^2(H, \mathsf{M})=\left\{[\varrho_H] : \varrho\in\mathcal{Z}^2(\mathsf{G}, \mathsf{M})\right\}
\ .
	\end{equation}
In particular, given $\sigma\in \mathcal{Z}^2(H, \mathsf{M})$, there exists $\varrho\in\mathcal{Z}^2(\mathsf{G}, \mathsf{M})$ such that $\varrho_H=\sigma$.

Consider a finite abelian group $\mathsf{G}$, and a field $\mathbb{F}$. Suppose that $\mathsf{G}$ acts trivially on $\mathbb{F}^*$. If $\mathbb{F}$ is algebraically closed, by Corollary \ref{2.09}, for any subgroup $H$ of $\mathsf{G}$, we have that 
	\begin{equation}\nonumber
	\mathcal{H}^2(H, \mathbb{F}^*)=\mathsf{res}^\mathsf{G}_H \left(\mathcal{H}^2(\mathsf{G}, \mathbb{F}^*)\right)
	\ .
	\end{equation}
In special, by Proposition \ref{2.01}, if $|\mathsf{G}|=n$, it follows that $\mathcal{H}^2(\mathsf{G},\mathbb{F}^*)$ is finite, and so there exist $\sigma_1,\dots,\sigma_r\in\mathcal{Z}^2(\mathsf{G},\sqrt[n]{1_\mathbb{F}})$ such that $\mathcal{H}^2(\mathsf{G},\mathbb{F}^*)=\{[\sigma_1],\dots, [\sigma_r]\}$, with $[\sigma_i]\neq[\sigma_j]$ for $i\neq j$, where $\sigma\in\mathcal{Z}^2(\mathsf{G},\sqrt[n]{1_\mathbb{F}})$ means that $(\sigma(\xi,\zeta))^n=1_\mathbb{F}$ for any $\xi,\zeta\in \mathsf{G}$.

%

\subsection{A partial order on \texorpdfstring{$\mathsf{G}$}{G}}\label{1.53}
Let us use $\mathcal{H}^2(H,\mathbb{F}^*)$, $H\leq\mathsf{G}$, to define a partial order on $\mathsf{G}$. Recall that a nonempty set $P$ together with a binary relation ``$\leq$'' is called a \textbf{partially ordered set} (also called \textbf{poset}) if satisfies the following axioms: for any $a, b, c\in P$,
	\begin{enumerate}[{\it i)}]
		\item $a\leq a$;
		\item if $a\leq b$ and $b\leq c$, then $a\leq c$;
		\item if $a\leq b$ and $b\leq a$, then $a=b$.
	\end{enumerate}
It is important to say that given $a,b\in P$, we do not necessarily have $a\leq b$ or $b\leq a$. For more details, see \cite{Halm17} or \cite{Jaco12.1}. 

	\begin{example}
	 Let $X$ be a nonempty set and $\mathcal{P}(X)=\{\beta: \beta \mbox{ is a subset of }X\}$. We have that $(\mathcal{P}(X), \subseteq)$ is a poset. We say that $X$ has the ordering by inclusion.
	\end{example}

	Let $\mathsf{G}$ be a group and $\mathbb{F}$ a field. For each subgroup $H\leq\mathsf{G}$, consider $\mathcal{H}^2(H,\mathbb{F}^*)=\{[\sigma] : \sigma\in\mathcal{Z}^2(\mathsf{G},\mathbb{F}^*) \}$, where $[\sigma]=[\rho]$ iff there is a $2$-coboundary $\varrho\in\mathcal{B}^2(H,\mathbb{F}^*)$ such that $\sigma=\rho\varrho$. Define $\mathcal{U}_{\mathsf{G},\mathbb{F}}=\{(H, \sigma) : H \leq \mathsf{G}, \sigma\in\mathcal{Z}^2(H,\mathbb{F}^*)\}=\bigcup_{H\leq\mathsf{G}}\left\{(H,\sigma): \sigma\in\mathcal{H}^2(H,\mathbb{F}^*)\right\}$. Note that if $\mathsf{G}$ is a finite abelian group and $\mathbb{F}$ is an algebraically closed field (so satisfying the hypotheses of Corollary \ref{2.09}), then $\mathcal{U}_{\mathsf{G},\mathbb{F}}=\{(H, \sigma_H) : H \leq \mathsf{G}, \sigma\in\mathcal{Z}^2(\mathsf{G},\mathbb{F}^*)\}$. Now, in $\mathcal{U}_{\mathsf{G},\mathbb{F}}$, we define the relation ``$\preceq$'' as follows: given two pairs $(H,\sigma)$ and $(N, \rho)$ in $\mathcal{U}_{\mathsf{G},\mathbb{F}}$, we write $(H, \sigma)\preceq (N,\rho)$ when $H\leq N$ and $[\sigma]=[\rho_H]$, i.e. there is $\varrho\in\mathcal{B}^2(H,\mathbb{F}^*)$ such that $\sigma(\xi,\zeta)=\varrho\rho(\xi,\zeta)$ for any $\xi,\zeta\in H$. In addition, we write $(H, \sigma)\sim(N, \rho)$ when $(H, \sigma)\preceq (N, \rho)$ and $(N, \rho)\preceq (H, \sigma)$. It is not difficult to see that ``$\preceq$'' is a partial order relation of the elements of $\mathcal{U}_{\mathsf{G},\mathbb{F}}$. Therefore, $(\mathcal{U}_{\mathsf{G},\mathbb{F}}, \preceq)$ is a poset. Consequently, we can say that $\mathsf{G}$ is ordering by ``$\preceq$''.

%

\section{Graded Imbeddings in Finite Dimensional Simple Graded Algebras}\label{sec4}

In this section, we study $\mathsf{G}$-graded imbeddings of finite dimensional simple $\mathsf{G}$-graded $\mathbb{F}$-algebras, where $\mathbb{F}$ is a field and $\mathsf{G}$ is a group, both arbitrary. Here, the main aim is to give necessary and sufficient conditions for the existence of a $\mathsf{G}$-graded imbedding between two finite dimensional simple $\mathsf{G}$-graded $\mathbb{F}$-algebras.

From now on, let us use the partial order ``$\preceq$'' defined in \S\ref{1.53}, page \pageref{1.53}. Under the condition that $\mathbb{F}$ is an algebraically closed field and $\mathsf{G}$ is a group, Lemma \ref{teodivgradalg} ensures that $\mathbb{F}^{\sigma_1}[H_1] \cong_{\mathsf{G}} \mathbb{F}^{\sigma_2}[H_2]$ iff $(H_1,\sigma_1)\sim(H_2,\sigma_2)$, where $H_1,H_2$ two finite subgroups of $\mathsf{G}$, and $\sigma_1\in\mathcal{Z}^2(H_1,\mathbb{F}^*)$ and $\sigma_2\in\mathcal{Z}^2(H_2,\mathbb{F}^*)$ two $2$-cocycles. It is interesting to note that this last result is not the most general case of this problem. Below we present this result for an arbitrary field.

\begin{proposition}\label{4.04}
Let $\mathbb{F}$ be a field, $\mathsf{G}$ a group, and $\mathbb{F}^{\sigma_1}[H_1]$ and $\mathbb{F}^{\sigma_2}[H_2]$ two twisted group algebras, where $H_i$ is a subgroup of $\mathsf{G}$ and $\sigma_i\in\mathcal{Z}^2(H_i,\mathbb{F}^*)$ is a $2$-cocycle on $H_i$, for $i=1,2$. Then $\mathbb{F}^{\sigma_1}[H_1] \cong_{\mathsf{G}} \mathbb{F}^{\sigma_2}[H_2]$ iff $(H_1,\sigma_1)\sim(H_2, \sigma_2)$.
\end{proposition}
\begin{proof}
Put $\mathfrak{B}_1=\mathbb{F}^{\sigma_1}[H_1]$ and $\mathfrak{B}_2=\mathbb{F}^{\sigma_2}[H_2]$. Suppose $\varphi:\mathfrak{B}_1\rightarrow\mathfrak{B}_2$ a $\mathsf{G}$-graded isomorphism. For each $\xi\in H_1$, we have that $\varphi((\mathfrak{B}_1)_\xi)\subseteq(\mathfrak{B}_2)_\xi$, and hence, since $(\mathfrak{B}_1)_\xi=\mathsf{span}_\mathbb{F}\{\eta_\xi\}$ and $(\mathfrak{B}_2)_\xi=\mathsf{span}_\mathbb{F}\{\tilde\eta_\xi\}$, it follows that $H_1\subseteq H_2$. As $\varphi^{-1}:\mathfrak{B}_2\rightarrow\mathfrak{B}_1$ is also a $\mathsf{G}$-graded isomorphism, we conclude that $H_1 = H_2$. Now, for each $\xi\in H_1$, there exists $\lambda_\xi\in\mathbb{F}^*$ such that $\varphi(\eta_\xi)=\lambda_\xi \tilde\eta_{\xi}$. Take $\varrho:H_1\times H_1\rightarrow \mathbb{F}^*$ defined by $\varrho(\xi,\zeta)=\frac{\lambda_{\xi\zeta}}{\lambda_\xi\lambda_\zeta}$, $\xi,\zeta\in H_1$. Note that  $\varrho\in\mathcal{B}^2(H_1,\mathbb{F}^*)$. Since $\varphi$ is a homomorphism of algebras, we have that 
		\begin{equation}\nonumber
		\begin{split}
	& \ \varphi(\eta_\xi\eta_\zeta)=\varphi(\sigma_1(\xi,\zeta)\eta_{\xi\zeta})=\sigma_1(\xi,\zeta)\lambda_{\xi\zeta}\tilde\eta_{\xi\zeta}\\
	& \ \ \ \ \ \ \rotatebox{90}{=} \\
	& \varphi(\eta_\xi)\varphi(\eta_\zeta) =\lambda_\xi\tilde\eta_\xi \lambda_\zeta\tilde\eta_\zeta =\lambda_\xi\lambda_\zeta\sigma_2(\xi,\zeta)\tilde\eta_{\xi\zeta}
		\end{split}
	\end{equation}
for any $\xi,\zeta\in H_1$, and hence, $\sigma_2=\varrho\sigma_1$. Consequently, $[\sigma_1]=[\sigma_2]$, and so $(H_1,\sigma_1)\sim(H_2,\sigma_2)$.
	
	Reciprocally, suppose $(H_1,\sigma_1)\sim(H_2, \sigma_2)$. Hence $H_1=H_2$ and $[\sigma_1]=[\sigma_2]$. Take a map $f:H_1\rightarrow \mathbb{F}^*$ such $\varrho:H_1\times H_1\rightarrow \mathbb{F}^*$ defined by $\varrho(\xi,\zeta)=\frac{f(\xi\zeta)}{f(\xi)f(\zeta)}$, $\xi,\zeta\in H_1$, satisfies $\sigma_2=\varrho\sigma_1$. Now, consider the linear map $\psi:\mathfrak{B}_1\rightarrow\mathfrak{B}_2$ which extends $\psi(\eta_\xi)=f(\xi)\tilde\eta_\xi$ for any $\xi\in H_1$. Notice that $\psi$ is bijective, because $H_1=H_2$ and $\{\eta_\xi: \xi\in H_1\}$ and  $\{\tilde\eta_\zeta: \zeta\in H_2\}=\{\psi(\eta_\zeta): \zeta\in H_1\}$ are basis of $\mathfrak{B}_1$ and $\mathfrak{B}_2$, respectively. Besides that, since
	\begin{equation}\nonumber
		\begin{split}
	& \ \ \ \ \tilde\eta_{\xi}\tilde\eta_{\zeta}=\frac{1}{f(\xi)}\psi(\eta_{\xi})\frac{1}{f(\zeta)}\psi(\eta_{\zeta})=\frac{1}{f(\xi)f(\zeta)}\psi(\eta_{\xi})\psi(\eta_{\zeta})	\\
	& \ \ \ \ \ \ \rotatebox{90}{=} \\
	& \sigma_2(\xi,\zeta)\tilde\eta_{\xi\zeta}=\frac{f(\xi\zeta)}{f(\xi)f(\zeta)}\sigma_1(\xi,\zeta)\tilde\eta_{\xi\zeta}=\frac{\sigma_1(\xi,\zeta)}{f(\xi)f(\zeta)}\psi(\eta_{\xi\zeta})=\frac{1}{f(\xi)f(\zeta)}\psi(\eta_{\xi}\eta_{\zeta}) \ ,
		\end{split}
	\end{equation}
for any $\xi,\zeta\in H_1$, we have that  $\psi$ is a homomorphism of algebras. Therefore, $\psi$ is a $\mathsf{G}$-graded isomorphism from $\mathfrak{B}_1$ into $\mathfrak{B}_2$. The result follows.
\end{proof}

Although the previous proposition is a small generalization of the second part of Lemma \ref{teodivgradalg}, it is necessary for our next result.

\begin{lemma}[Imbedding]\label{4.16}
	Let $\mathbb{F}$ be a field and $\mathsf{G}$ a group. For any two subgroups $H_1$ and $H_2$ of $\mathsf{G}$ and any two $2$-cocycles $\sigma_1\in\mathcal{Z}^2(H_1,\mathbb{F}^*)$ and $\sigma_2\in\mathcal{Z}^2(H_2,\mathbb{F}^*)$, $\mathbb{F}^{\sigma_1}[H_1] \stackrel{\mathsf{G}}{\hookrightarrow} \mathbb{F}^{\sigma_2}[H_2]$ iff $(H_1,\sigma_1)\preceq (H_2, \sigma_2)$.
\end{lemma}
\begin{proof}
	First, put $\mathfrak{B}_1=\mathbb{F}^{\sigma_1}[H_1]$ and $\mathfrak{B}_2=\mathbb{F}^{\sigma_2}[H_2]$. Suppose that  $H_1\leq H_2$ and $[\sigma_1]=[\sigma_2]_{H_1}$. Hence, there exists some $\varrho\in\mathcal{B}^2(H_1, \mathbb{F}^*)$ such that $(\sigma_2)_{H_1}=\varrho\sigma_1$. By Proposition \ref{4.04}, it is immediate that $\mathfrak{B}_1\cong_\mathsf{G} \mathbb{F}^{\tilde{\sigma}}[H_1]$, where $\tilde{\sigma}$ is the $2$-cocycle of $\mathcal{Z}^2(H_1, \mathbb{F}^*)$ defined by $\tilde{\sigma}(\xi,\zeta)=\varrho(\xi,\zeta)\sigma_1(\xi,\zeta)=\sigma_2(\xi,\zeta)$ for any $\xi,\zeta \in H_1$.

Since $H_1\leq H_2$, it follows from the definition of the twisted group algebra that $\mathbb{F}^{\tilde{\sigma}}[H_1]$ is a graded subspace of $\mathbb{F}^{\sigma_2}[H_2]$. Let us now show that $\mathbb{F}^{\tilde{\sigma}}[H_1]$ is an $H_2$-graded subalgebra of $\mathbb{F}^{\sigma_2}[H_2]$. In fact, let us denote by ``$*$'' the multiplication of $\mathbb{F}^{\sigma_2}[H_2]$ and by ``$\star$'' the multiplication of $\mathbb{F}^{\tilde{\sigma}}[H_1]$. Given $\eta_\zeta,\eta_\xi\in\mathbb{F}^{\tilde{\sigma}}[H_1]$, it follows that
	\begin{equation}\nonumber
		\eta_\zeta \star \eta_\xi=\tilde{\sigma}(\zeta,\xi)\eta_{\zeta\xi}=\varrho(\zeta,\xi)\sigma_1(\zeta,\xi)\eta_{\zeta\xi}=\sigma_2(\zeta,\xi)\eta_{\zeta\xi}=\eta_\zeta * \eta_\xi 
		\ .
	\end{equation}
	This shows that ``$*$'' and ``$\star$'' are equal in $\mathbb{F}^{\sigma_1}[H_1]$, and so we conclude that $\mathbb{F}^{\tilde{\sigma}}[H_1]$ is an $H_2$-graded subalgebra of $\mathbb{F}^{\sigma_2}[H_2]$. 
	Hence, we have that $\mathfrak{B}_1 \stackrel{\mathsf{G}}{\hookrightarrow} \mathfrak{B}_2$, because $\mathfrak{B}_1\cong_\mathsf{G} \mathbb{F}^{\tilde{\sigma}}[H_1]$.
	
	On the other hand, suppose that $\mathfrak{B}_1 \stackrel{\mathsf{G}}{\hookrightarrow} \mathfrak{B}_2$. Hence, by definition of ``$\stackrel{\mathsf{G}}{\hookrightarrow}$'' ($\mathsf{G}$-imbedding), there exists a graded homomorphism $\psi$ of $\mathfrak{B}_1$ to $\mathfrak{B}_2$ which is injective. Notice that $\psi(\eta_\xi)\in(\mathfrak{B}_2)_\xi$ is different to zero for any $\xi\in H_1$, because $\psi$ is injective. Hence, $H_1=\mathsf{Supp}(\Gamma_{\mathfrak{B}_1})\subseteq \mathsf{Supp}(\Gamma_{\mathfrak{B}_2})= H_2$, and so $H_1\leq H_2$. Applying Lemma \ref{4.13}, we have that $\mathfrak{B}_1\cong_\mathsf{G}  \mathsf{im}(\psi)$, where $\mathsf{im}(\psi)=\mathsf{span}_\mathbb{F}\{\tilde\eta_\xi \in\mathfrak{B}_2: \xi\in H_1\}$ is a graded subalgebra of $\mathfrak{B}_2$, and so $\mathbb{F}^{\sigma_1}[H_1]\cong_\mathsf{G}\mathbb{F}^{\hat{\sigma}}[H_1]$, where $\hat{\sigma}(\xi,\zeta)=\sigma_2(\xi,\zeta)$ for any $\xi,\zeta\in H_1$. By Lemma \ref{4.04}, it follows that $[\sigma_1]=[\hat{\sigma}]=[\sigma_2]_{H_1}$. 
Therefore, we conclude that $(H_1,\sigma_1)\preceq(H_2,\sigma_2)$. The result follows. %
\end{proof}

Observe that Corollary \ref{2.09} and Lemma \ref{4.16} play a central role for constructing $\mathsf{G}$-imbeddings of division graded algebras (of finite dimension). Let us now return to the following natural question: ``\textit{given a twisted group algebra $\mathfrak{A}=\mathbb{F}^{\sigma}[H]$, which is a division graded algebra, can we imbed $\mathfrak{A}$ in another division graded algebra?}''. The next two corollaries answer this question.

\begin{corollary}\label{4.05}
Let $\mathbb{F}$ be an algebraically closed field, $\mathsf{G}$ a group, and $\mathfrak{B}$ a finite dimensional algebra with a $\mathsf{G}$-grading $\Gamma$. Suppose that $\mathfrak{B}$ is a division graded algebra. If there exists a subgroup $H$ of $\mathsf{G}$ such that $\mathsf{Supp}(\Gamma)$ is a central subgroup of $H$, with $[H:\mathsf{Supp}(\Gamma)]< \infty$, then there exists a $2$-cocycle $\sigma\in\mathcal{Z}^2(H, \mathbb{F}^*)$ such that $\mathfrak{B}\stackrel{\mathsf{G}}{\hookrightarrow} \mathbb{F}^{\sigma}[H]$. In addition, if $\mathsf{Supp}(\Gamma)= H_1\subseteq H_2\subseteq\cdots\subseteq H_{n}\subseteq\mathsf{G}$ is a chain of subgroups of $\mathsf{G}$ such that $H_i$ is central in $H_{i+1}$ and $[H_{i+1}:H_i]< \infty$, $i=1,\dots,n-1$, then there exist $2$-cocycles $\sigma_1\in \mathcal{H}^2(H_1,\mathbb{F}^*)$, $\sigma_2\in \mathcal{H}^2(H_2,\mathbb{F}^*)$, $\dots$, $\sigma_n\in \mathcal{H}^2(H_n,\mathbb{F}^*)$ such that 
	\begin{equation}\nonumber
\mathfrak{B}\stackrel{\mathsf{G}}{\hookrightarrow} \mathbb{F}^{\sigma_1}[H_1]\stackrel{\mathsf{G}}{\hookrightarrow}\cdots\stackrel{\mathsf{G}}{\hookrightarrow} \mathbb{F}^{\sigma_{n-1}}[H_{n-1}]\stackrel{\mathsf{G}}{\hookrightarrow}\mathbb{F}^{\sigma_n}[H_n] \ 
	\end{equation}
where $[\sigma_{i+1}]_{H_i}=[\sigma_i]$ for all $i=1,\dots,n-1$.
\end{corollary}
\begin{proof}
By Lemma \ref{teodivgradalg}, there exist a subgroup $N$ of $\mathsf{G}$ and a $2$-cocycle $\sigma\in \mathcal{H}^2(N,\mathbb{F}^*)$ such that $\mathfrak{B}\cong_\mathsf{G} \mathbb{F}^\sigma[N]$. From this, we have that $N=\mathsf{Supp}(\Gamma)\subseteq H$, and hence, by Corollary \ref{2.09}, there exist a $2$-cocycle $\rho\in \mathcal{H}^2(H,\mathbb{F}^*)$ such that $[\rho]_N=[\sigma]$. Finally, by Lemma \ref{4.16}, it follows that $\mathfrak{B}\cong_\mathsf{G} \mathbb{F}^\sigma[N]\stackrel{\mathsf{G}}{\hookrightarrow} \mathbb{F}^{\sigma}[H]$. The second part of this result follows from first.
\end{proof}

By the previous lemma (and using Corollary \ref{2.09}), observe that, given $H_1$ and $H_2$ two central subgroups of finite index of a group $\mathsf{G}$ and $\sigma_1\in\mathcal{Z}^2(H_1,\mathbb{F}^*)$ and $\sigma_2\in\mathcal{Z}^2(H_2,\mathbb{F}^*)$ are $2$-cocycles, if there is a $2$-cocycle $\rho\in\mathcal{H}^2(\mathsf{G},\mathbb{F}^*)$ such that $[\rho]_{H_1}=[\sigma_1]$ and $[\rho]_{H_2}=[\sigma_2]$, then $\mathbb{F}^{\sigma_1}[H_1], \mathbb{F}^{\sigma_2}[H_2] \stackrel{\mathsf{G}}{\hookrightarrow} \mathbb{F}^{\rho}[\mathsf{G}]$, but not necessarily $\mathbb{F}^{\sigma_1}[H_1] \stackrel{\mathsf{G}}{\hookrightarrow} \mathbb{F}^{\sigma_2}[H_2]$ or $\mathbb{F}^{\sigma_2}[H_2] \stackrel{\mathsf{G}}{\hookrightarrow} \mathbb{F}^{\sigma_1}[H_1] $. This motivates the next result, which follows from Corollary \ref{4.05} and Lemma \ref{4.16}.

\begin{corollary}\label{4.06}
Let $\mathbb{F}$ be an algebraically closed field, $\mathsf{G}$ a finite abelian group, and $\mathfrak{D}_1$ and $\mathfrak{D}_2$ two finite dimensional division $\mathsf{G}$-graded algebras. If there is no $\sigma\in\mathcal{H}^2(\mathsf{G},\mathbb{F}^*)$ such that $\mathfrak{D}_1, \mathfrak{D}_2 \stackrel{\mathsf{G}}{\hookrightarrow} \mathbb{F}^{\sigma}[\mathsf{G}]$, then neither $\mathfrak{D}_1 \stackrel{\mathsf{G}}{\hookrightarrow} \mathfrak{D}_2$ nor $\mathfrak{D}_2 \stackrel{\mathsf{G}}{\hookrightarrow} \mathfrak{D}_1$.
\end{corollary}

Let us now study the graded imbeddings in simple graded algebras (of finite dimension). Let $\mathsf{G}$ be a group, $H$ a subgroup of $\mathsf{G}$, and $\mathsf{N}_\mathsf{G}(H)$ the normalizer of $H$ in $\mathsf{G}$, i.e. $\mathsf{N}_\mathsf{G}(H)=\{\xi\in\mathsf{G}:\xi^{-1}H\xi=H\}$. Now, given $\theta=(\theta_1,\dots,\theta_k)\in\mathsf{G}^k$, consider the family 
\begin{equation}\nonumber
\Lambda^H_\theta=\{\left(\delta\xi_1\theta_{\alpha(1)},\dots,\delta\xi_k\theta_{\alpha(k)}\right)\in\mathsf{G}^k: \delta\in\mathsf{N}_\mathsf{G}(H), \xi_1,\dots,\xi_k\in H, \alpha\in Sym(k)\}
\ ,
\end{equation}
where $Sym(k)$ is the symmetric group of order $k$. Note that $\theta\in\Lambda^H_{\theta}$, and if $\tilde\theta\in\Lambda^H_{\theta}$, then $\Lambda^H_{\tilde\theta}=\Lambda^H_{\theta}$. Indeed, suppose $\tilde\theta=(\delta\xi_{\alpha(1)}\theta_{\alpha(1)},\dots,\delta\xi_{\alpha(k)}\theta_{\alpha(k)})$ for some $\alpha\in Sym(k)$, $\delta\in\mathsf{N}_\mathsf{G}(H)$ and $\xi_1,\dots,\xi_k\in H$. Given any $\hat\theta\in\Lambda^H_{\tilde\theta}$, there are $\hat\alpha\in Sym(k)$, $\hat\delta\in\mathsf{N}_\mathsf{G}(H)$ and $\hat\xi_1,\dots,\hat\xi_k\in H$ such that 
\begin{equation}\nonumber
\hat\theta=(\hat\delta\hat\xi_{\hat\alpha(\alpha(1))}(\delta\xi_{\hat\alpha(\alpha(1))}\theta_{\hat\alpha(\alpha(1))}),\dots,\hat\delta\hat\xi_{\hat\alpha(\alpha(k))}(\delta\xi_{\hat\alpha(\alpha(k))}\theta_{\hat\alpha(\alpha(k))})) \ .
\end{equation}
Since $H\hat\delta=\hat\delta H$, it follows that $\hat\delta\hat\xi_{\hat\alpha(\alpha(i))}\delta\xi_{\hat\alpha(\alpha(i))}=(\hat\delta\delta)(\hat{\hat{\xi}}_{\hat\alpha(\alpha(i))}\xi_{\hat\alpha(\alpha(i))})$ for all $i=1,\dots,k$, for some $\hat{\hat{\xi}}_{\hat\alpha(\alpha(i))}\in H$. From this $\hat\theta\in \Lambda_{\theta}^H$, and so $\Lambda_{\tilde\theta}^H\subseteq \Lambda_{\theta}^H$. Now, since $\tilde\theta\in\Lambda^H_{\theta}$ iff $\theta\in\Lambda^H_{\tilde\theta}$, analogously we can show that $\Lambda_{\theta}^H\subseteq \Lambda_{\tilde\theta}^H$.

\begin{lemma}\label{4.09}
Let $\mathbb{F}$ be a field and $\mathsf{G}$ a group, both arbitrary. Given $H$ a subgroup of $\mathsf{G}$ and $\sigma\in\mathcal{Z}^2(H, \mathbb{F}^*)$, consider the algebra $\mathfrak{B}=M_k(\mathbb{F}^\sigma[H])$ of $k\times k$ matrices over $\mathbb{F}^\sigma[H]$ with an elementary-canonical $\mathsf{G}$-grading defined by a $k$-tuple $\theta\in \mathsf{G}^k$. For any $\tilde\theta\in\Lambda^H_\theta$, then either $\theta$ and $\tilde\theta$ define equivalent elementary-canonical $\mathsf{G}$-gradings on $\mathfrak{B}$ or there exists a $2$-cocycle $\rho$ on $H$ such that $\mathfrak{B}\cong_{\mathsf{G}} M_k(\mathbb{F}^\rho[H])$, where $M_k(\mathbb{F}^\rho[H])$ is graded with the elementary-canonical $\mathsf{G}$-grading defined by $\tilde\theta$. In addition, if $\theta=(\theta_1, \dots, \theta_k)$ and $\theta_{i_0}\in\mathsf{N}_{\mathsf{G}}(H)$ for some $i_0\in\{1,\dots,k\}$, then there exists a $2$-cocycle $\varrho$ on $H$ such that $\mathfrak{B}\cong_{\mathsf{G}} M_k(\mathbb{F}^\varrho[H])$, where $M_k(\mathbb{F}^\varrho[H])$ is graded with the elementary-canonical $\mathsf{G}$-grading defined by $(e,\theta_{i_0}^{-1}\theta_{\alpha(2)}, \dots, \theta_{i_0}^{-1}\theta_{\alpha(k)})$, for some $\alpha\in Sym(k)$.
\end{lemma}
\begin{proof}
First, fix $\theta=(\theta_1, \dots, \theta_k)\in\mathsf{G}^k$. Let us show the first part of this lemma in three steps: 1) for any $\alpha\in Sym(k)$, $\theta$ and $(\theta_{\alpha(1)}, \dots, \theta_{\alpha(k)})$ define equivalent elementary-canonical $\mathsf{G}$-gradings on $\mathfrak{B}$; 2) given any $\xi_1,\dots,\xi_k\in H$, $\theta$ and $(\xi_1\theta_1, \dots, \xi_k\theta_k)$ define equivalent elementary-canonical $\mathsf{G}$-gradings on $\mathfrak{B}$; and 3) given $\delta\in\mathsf{N}_{\mathsf{G}}(H)$, there exists $\rho\in\mathcal{Z}^2(H,\mathbb{F}^*)$ such that $\mathfrak{B}\cong_{\mathsf{G}} M_k(\mathbb{F}^\rho[H])$, where $M_k(\mathbb{F}^\rho[H])$ is graded with the elementary-canonical $\mathsf{G}$-grading defined by $(\delta\theta_1, \dots, \delta\theta_k)$.

{\it Step 1:} Take any $\alpha\in Sym(k)$, and consider the algebra $\widehat{\mathfrak{B}}=M_k(\mathbb{F}^\sigma[H])$ with the elementary-canonical $\mathsf{G}$-grading defined by $\left(\theta_{\alpha(1)},\dots,\theta_{\alpha(k)}\right)$. It is easy to see that the linear application $\psi$ which extends the map $E_{ij}\eta_\zeta\mapsto E_{\alpha^{-1}(i)\alpha^{-1}(j)}\tilde{\eta}_{\zeta}$, for any $i,j\in\{1,\dots,k\}$ and $\zeta\in H$, is a $\mathsf{G}$-graded isomorphism of $\mathfrak{B}$ in $\widehat{\mathfrak{B}}$.

{\it Step 2:} Let $\xi_1,\dots,\xi_k\in H$. First, consider the algebra $\mathfrak{B}_1=M_k(\mathbb{F}^\sigma[H])$ with the elementary-canonical $\mathsf{G}$-grading defined by $\left(\xi_1\theta_1,\theta_2, \dots,\theta_k \right)$. Consider the linear application $\psi$ that extends the following relations: {\it i)} $E_{1j}\eta_\zeta \mapsto \sigma(\xi_1,\zeta) E_{1j}\tilde{\eta}_{\xi_1 \zeta}$ for $j\neq1$; 
{\it ii)} $E_{i1}\eta_\zeta \mapsto \frac{\sigma(\zeta,\xi_1^{-1})}{\sigma(\xi_1^{-1},\xi_1)\sigma(e,e)} E_{i1}\tilde{\eta}_{\zeta \xi_1^{-1}}$ for $i\neq1$; 
{\it iii)} $E_{11}\eta_\zeta \mapsto \frac{\sigma(\xi_1,\zeta)\sigma(\xi_1\zeta,\xi_1^{-1})}{\sigma(\xi_1^{-1},\xi_1)\sigma(e,e)} E_{11}\tilde{\eta}_{\xi_1 \zeta \xi_1^{-1}}$; and 
{\it iv)} $E_{ij}\eta_\zeta \mapsto E_{ij}\tilde{\eta}_\zeta$ for $i,j\neq1$, 
for all $i,j=1,\dots,k$ and $\zeta\in H$. It is easy to check that $\psi$ is a $\mathsf{G}$-graded isomorphism of $\mathfrak{B}$ in $\mathfrak{B}_1$. Now, consider the algebra $\mathfrak{B}_2=M_k(\mathbb{F}^\sigma[H])$ with the elementary-canonical $\mathsf{G}$-grading defined by $\left(\xi_2\theta_2,\xi_1\theta_1,\theta_3, \dots,\theta_k \right)$. By {\it Step 1}, we have that $\left(\xi_1\theta_1,\theta_2, \dots,\theta_k \right)$ and $\left(\theta_2,\xi_1\theta_1,\theta_3, \dots,\theta_k \right)$ define equivalent elementary-canonical $\mathsf{G}$-grading on $\mathfrak{B}_1$. So, as in the first part of this step, we can ensure that $\mathfrak{B}_1$ and $\mathfrak{B}_2$ are $\mathsf{G}$-graded isomorphism. Proceeding by induction, we conclude that $\left(\theta_1,\theta_2, \dots,\theta_k \right)$ and $\left(\xi_1\theta_1,\xi_2\theta_2, \dots,\xi_k\theta_k \right)$ define equivalent elementary-canonical $\mathsf{G}$-gradings on $\mathfrak{B}$.

{\it Step 3:} In \S\ref{resGH}, for each $\delta\in \mathsf{G}$, we present the maps $c(\xi)^*: \mathcal{H}^2(\delta H \delta^{-1},\mathbb{F}^*) \rightarrow \mathcal{H}^2(H,\mathbb{F}^*)$ (here, $\mathsf{M}=\mathbb{F}^*$), which is an isomorphism of groups, and so there exists a $2$-cocycle $\rho\in \mathcal{Z}^2(\delta H \delta^{-1},\mathbb{F}^*)$ such that $\rho(\delta \xi\delta^{-1},\delta \zeta\delta^{-1})=\sigma(\xi,\zeta)$ for any $\xi,\zeta\in H$. From this, taking $\delta\in\mathsf{N}_{\mathsf{G}}(H)$, we have that $\mathcal{Z}^2(\delta H \delta^{-1},\mathbb{F}^*)=\mathcal{Z}^2(H,\mathbb{F}^*)$. Now, consider the algebra $\widetilde{\mathfrak{B}}=M_k(\mathbb{F}^\rho[H])$ with the elementary-canonical $\mathsf{G}$-grading defined by $\left(\delta\theta_1,\dots,\delta\theta_k \right)$. It is not difficult to see that the linear map $\psi$ which extends the map $E_{ij}\eta_\xi\mapsto E_{ij}\tilde{\eta}_{\delta\xi \delta^{-1}}$, for any $i,j\in\{1,\dots,k\}$ and $\xi\in H$, is a $\mathsf{G}$-graded isomorphism of $\mathfrak{B}$ in $\widetilde{\mathfrak{B}}$.

Finally, assume that $\theta_{i_0}\in\mathsf{N}_{\mathsf{G}}(H)$ for some $i_0\in\{1,\dots,k\}$. Take $\alpha=(1 \ i_0)\in Sym(k)$ a transposition, i.e. a permutation that satisfies $\alpha(1)=i_0$, $\alpha(i_0)=1$ and $\alpha(j)=j$ when $j\in\{1,\dots,k\}\setminus\{1,i_0\}$. Note that $\theta_{i_0}^{-1}\theta_{\alpha(1)}=e$ and $(\theta_{i_0}^{-1}\theta_{\alpha(1)}, \dots, \theta_{i_0}^{-1}\theta_{\alpha(k)})\in\Lambda^H_\theta$, and so the result follows from the first part of the lemma.
\end{proof}

From the discussion above, being $\mathfrak{B}=M_k(\mathbb{F}^\sigma[H])$ an algebra with an elementary-canonical $\mathsf{G}$-grading $\Gamma$ defined by a $k$-tuple $(\theta_1, \dots, \theta_k)\in\mathsf{G}^k$, when at lest one of $\theta_1, \dots, \theta_k$ belongs to normalizer of $H$ in $\mathsf{G}$, we can assume, without loss of generality, that $\theta_1=e$. Consequently, it is easy to see that $\theta_i,\theta_i^{-1},\theta_i^{-1}\theta_j\in\mathsf{Supp}(\Gamma)$ for all $i,j=1,\dots,k$. On the other hand, let $\mathsf{G}$ be a group that contains a non-normal subgroup $H$ (for example, the group $Sym(3)$ has this property). Take $\delta\in \mathsf{G}$ such that $H\neq H^\delta$, i.e. $\delta\notin\mathsf{N}_\mathsf{G}(H)$. Considering $\hat\theta_1=\hat\theta_2=\delta$, take $\hat\theta=(\hat\theta_1,\hat\theta_2)$ and $\tilde\theta=(\delta^{-1}\hat\theta_1,\delta^{-1}\hat\theta_2)=(e,e)$ two $2$-tuples in $\mathsf{G}^2$. \textbf{Claim:} $\hat\theta$ and $\tilde\theta$ not define equivalent elementary-canonical $\mathsf{G}$-gradings on $M_2(\mathbb{F}[H])$. In fact, assume that $\widehat\Gamma$ and $\widetilde\Gamma$ are the elementary-canonical $\mathsf{G}$-gradings on $M_2(\mathbb{F}[H])$ defined by $\hat\theta$ and $\tilde\theta$, respectively. We have that $\mathsf{Supp}(\widehat\Gamma)=\{\hat\theta_i^{-1}\xi\hat\theta_j: i,j=1,2,\ \xi\in H\}=\{\delta^{-1}\xi\delta:\xi\in H\}=H^\delta$, and $\mathsf{Supp}(\widetilde\Gamma)=\{\tilde\theta_r^{-1}\zeta\tilde\theta_s: r,s=1,2,\ \zeta\in H\}=\{\zeta:\zeta\in H\}=H$. Thus, since $H\neq H^\delta$, it follows that $\mathsf{Supp}(\widehat\Gamma)\neq\mathsf{Supp}(\widetilde\Gamma)$, and so $\widehat\Gamma$ and $\widetilde\Gamma$ are not equivalent elementary-canonical $\mathsf{G}$-gradings. Therefore, the condition ``$\delta\in\mathsf{N}_\mathsf{G}(H)$'' in the second part of Lemma \ref{4.09} is necessary.

The next result is key in our work. It allows us to determine, when given two matrices algebras $\mathfrak{A}=M_{k_1}(\mathbb{F}^{\sigma_1}[H_1])$ and $\mathfrak{B}=M_{k_2}(\mathbb{F}^{\sigma_2}[H_2])$ under suitable conditions, a (graded) homomorphism of $\mathfrak{A}$ in $\mathfrak{B}$ such that the diagonal matrices $E_{ii}\eta_{\zeta}$'s and  $E_{jj}\tilde\eta_{\xi}$'s of $\mathfrak{A}$ and $\mathfrak{B}$, respectively, are associated.

\begin{lemma}\label{4.18}
Let $\mathsf{G}$ be a group, $\mathbb{F}$ a field and $\mathfrak{A}$ and $\mathfrak{B}$ two algebras with $\mathsf{G}$-gradings. Suppose that $\mathfrak{A}$ and $\mathfrak{B}$ have generator sets $S=\bigcup_{\xi\in \mathsf{G}} S_\xi$ and $R=\bigcup_{\xi\in \mathsf{G}} R_\xi$, respectively, where $S_\xi\subset \mathfrak{A}_\xi$ and $R_\xi\subset \mathfrak{B}_\xi$ are countable sets, for all $\xi\in\mathsf{G}$. If $\mathsf{T}^{\mathsf{G}}(\mathfrak{A})\subseteq \mathsf{T}^{\mathsf{G}}(\mathfrak{B})$, then there is a $\mathsf{G}$-graded epimorphism of $\mathfrak{A}$ in $\mathfrak{B}$.
\end{lemma}
\begin{proof}
For each $\xi\in \mathsf{G}$, we have that $R_\xi$ and $S_\xi$ are countable, and so there exist bijective maps $\tau_{R,\xi}: X_\xi\rightarrow R_\xi$ and $\tau_{S,\xi}: X_\xi \rightarrow S_\xi$. Consequently, the maps $\tau_R: X^\mathsf{G}\rightarrow R$ and $\tau_S: X^\mathsf{G}\rightarrow S$ defined by $\tau_R(x_i^{(\xi)})=\tau_{R,\xi}(x_i^{(\xi)})$ and $\tau_S(x_i^{(\xi)})=\tau_{S,\xi}(x_i^{(\xi)})$, for any $x_i^{(\xi)}\in X^\mathsf{G}$, are bijective.

Let $inc:X^{\mathsf{G}}\rightarrow \mathbb{F}\langle X^{\mathsf{G}} \rangle$ be the inclusion map, $f_{\mathfrak{A}}: X^{\mathsf{G}}\rightarrow \mathfrak{A}$ defined by $f_{\mathfrak{A}}(x_i^{(\xi)})= \tau_R(x_i^{(\xi)})$ and $f_{\mathfrak{B}}:X^{\mathsf{G}}\rightarrow \mathfrak{B}$ defined by $f_{\mathfrak{B}}(x_i^{(\xi)})= \tau_S(x_i^{(\xi)})$. By the universal property of $\mathbb{F}\langle X^{\mathsf{G}} \rangle$, there exist $\mathsf{G}$-graded homomorphisms of algebras $\varphi_{\mathfrak{A}}:\mathbb{F}\langle X^{\mathsf{G}} \rangle \rightarrow \mathfrak{A}$ and $\varphi_{\mathfrak{B}}:\mathbb{F}\langle X^{\mathsf{G}} \rangle \rightarrow \mathfrak{B}$ which extend $f_{\mathfrak{A}}$ and $f_{\mathfrak{B}}$, respectively. Similarly to diagram (\ref{4.17}), we have the following diagram:
\begin{equation}\label{4.14}
\begin{tikzcd}
& X^{\mathsf{G}} \arrow[swap]{d}{inc} \arrow{rd}{f_{\mathfrak{B}}} \arrow[swap]{dl}{f_{\mathfrak{A}}} & \\
\mathfrak{A}  & \mathbb{F}\langle X^{\mathsf{G}} \rangle \arrow[swap,dashed]{r}{\varphi_{\mathfrak{B}}} \arrow[dashed]{l}{\varphi_{\mathfrak{A}}} & \mathfrak{B} 
\end{tikzcd} \ .
\end{equation}
As $R$ and $S$ generate $\mathfrak{A}$ and $\mathfrak{B}$, respectively, we have that $\varphi_\mathfrak{A}$ and $\varphi_\mathfrak{B}$ are also epimorphisms.

By Lemma \ref{4.13}, there exist $\mathsf{G}$-graded isomorphisms $\bar\varphi_\mathfrak{A}: \displaystyle\frac{\mathbb{F}\langle X^{\mathsf{G}} \rangle}{\mathsf{T}^{\mathsf{G}}(\mathfrak{A})}\rightarrow \mathfrak{A}$, $\bar\varphi_\mathfrak{B}: \displaystyle\frac{\mathbb{F}\langle X^{\mathsf{G}} \rangle}{\mathsf{T}^{\mathsf{G}}(\mathfrak{B})}\rightarrow \mathfrak{B}$, and $\psi: \displaystyle\frac{\mathbb{F}\langle X^{\mathsf{G}} \rangle}{\mathsf{T}^{\mathsf{G}}(\mathfrak{B})}\rightarrow \displaystyle\frac{\frac{\mathbb{F}\langle X^{\mathsf{G}} \rangle}{\mathsf{T}^{\mathsf{G}}(\mathfrak{A})}}{\frac{\mathsf{T}^{\mathsf{G}}(\mathfrak{B})}{\mathsf{T}^{\mathsf{G}}(\mathfrak{A})}}$, since $\mathsf{T}^{\mathsf{G}}(\mathfrak{A})\subseteq \mathsf{T}^{\mathsf{G}}(\mathfrak{B})$. Let $\pi$ be the projection homomorphism of $\displaystyle\frac{\mathbb{F}\langle X^{\mathsf{G}} \rangle}{\mathsf{T}^{\mathsf{G}}(\mathfrak{A})}$ in $\displaystyle\frac{\frac{\mathbb{F}\langle X^{\mathsf{G}} \rangle}{\mathsf{T}^{\mathsf{G}}(\mathfrak{A})}}{\frac{\mathsf{T}^{\mathsf{G}}(\mathfrak{B})}{\mathsf{T}^{\mathsf{G}}(\mathfrak{A})}}$, $\pi_1$ the projection homomorphism of $\mathbb{F}\langle X^{\mathsf{G}} \rangle$ in $\displaystyle\frac{\mathbb{F}\langle X^{\mathsf{G}} \rangle}{\mathsf{T}^{\mathsf{G}}(\mathfrak{A})}$, and $\pi_2$ the projection homomorphism of $\mathbb{F}\langle X^{\mathsf{G}} \rangle$ in $\displaystyle\frac{\mathbb{F}\langle X^{\mathsf{G}} \rangle}{\mathsf{T}^{\mathsf{G}}(\mathfrak{B})}$. Hence, we have the following diagram
\begin{equation}\label{4.15}
\begin{tikzcd}
&  & \mathfrak{B} \\
& \mathbb{F}\langle X^{\mathsf{G}} \rangle \arrow{ru}{\varphi_{\mathfrak{B}}} \arrow[swap]{dl}{\varphi_{\mathfrak{A}}} \arrow[swap]{r}{\pi_2} \arrow{d}{\pi_1} & \frac{\mathbb{F}\langle X^{\mathsf{G}} \rangle}{\mathsf{T}^{\mathsf{G}}(\mathfrak{B})} \arrow{d}{\psi} \arrow[swap,dashed]{u}{\bar\varphi_{\mathfrak{B}}}  \\
\mathfrak{A} \arrow[bend left=50,dashed]{rruu}{\varphi} & \frac{\mathbb{F}\langle X^{\mathsf{G}} \rangle}{\mathsf{T}^{\mathsf{G}}(\mathfrak{A})} \arrow{r}{\pi} \arrow[swap,dashed]{l}{\bar\varphi_{\mathfrak{A}}} & \displaystyle\frac{\frac{\mathbb{F}\langle X^{\mathsf{G}} \rangle}{\mathsf{T}^{\mathsf{G}}(\mathfrak{A})}}{\frac{\mathsf{T}^{\mathsf{G}}(\mathfrak{B})}{\mathsf{T}^{\mathsf{G}}(\mathfrak{A})}} 
\end{tikzcd}
\end{equation}
where $\varphi: \mathfrak{A} \rightarrow \mathfrak{B}$ is the homomorphism given by $\varphi \coloneqq \bar\varphi_\mathfrak{B}\circ\psi^{-1}\circ\pi\circ\bar\varphi_\mathfrak{A}^{-1}$. It is easy to see that $\varphi$ is a $\mathsf{G}$-graded epimorphism of $\mathfrak{A}$ in $\mathfrak{B}$.
\end{proof}

In what follows, let us use the idea of the proof of Lemma \ref{4.18} to prove that, given $\mathfrak{B}_1=M_{k_1}(\mathbb{F}^{\sigma_1}[H_1])$ and $\mathfrak{B}_2=M_{k_2}(\mathbb{F}^{\sigma_2}[H_2])$ with elementary-canonical $\mathsf{G}$-gradings and under suitable conditions, we can ``choose'' a maps ``$\varphi$'' from $\mathfrak{B}_1$ into $\mathfrak{B}_2$ satisfying, for any $\xi\in H_2$ and $i,j\in\{1,\dots,k_2\}$,
\begin{equation}\nonumber
E_{\alpha(i)\alpha(j)}\eta_{\zeta_{i,j,\xi}} \mapsto \lambda_{i,j,\xi}  E_{ij}\tilde\eta_\xi \ ,
\end{equation}
for some $\alpha\in Sym(k_1)$, $\zeta_{i,j,\xi}\in H_1$ and $\lambda_{i,j,\xi}\in\mathbb{F}^*$.

\begin{lemma}\label{4.08}
Let $\mathsf{G}$ be a group, $\mathbb{F}$ a field, and $\mathfrak{B}_1=M_{k_1}(\mathbb{F}^{\sigma_1}[H_1])$ and $\mathfrak{B}_2=M_{k_2}(\mathbb{F}^{\sigma_2}[H_2])$ two matrices algebras with elementary-canonical $\mathsf{G}$-gradings, where $H_1$ and $H_2$ are subgroups of $\mathsf{G}$, and $\sigma_i\in \mathcal{H}^2(H_i,\mathbb{F}^*)$ is a $2$-cocycle. Suppose $\mathsf{T}^{\mathsf{G}}(\mathfrak{B}_1)\subseteq \mathsf{T}^{\mathsf{G}}(\mathfrak{B}_2)$. If $k_1,k_2>1$, then there exists a $\mathsf{G}$-graded epimorphism $\varphi: \mathfrak{B}_1 \rightarrow \mathfrak{B}_2$ satisfying:
	\begin{enumerate}[{\it i)}]
	\item $\varphi(E_{\alpha(1)\alpha(i)}\eta_{\zeta_{1,i,\xi}}) = E_{1i}\tilde\eta_\xi$ and $\varphi(E_{\alpha(i)\alpha(1)}\eta_{\zeta_{i,1,e}}) = E_{i1}\tilde\eta_e$, for any $i=2,\dots,k_2$ and $\xi\in H_2$,
	\item $\varphi(E_{\alpha(j)\alpha(j)}\eta_{\zeta_{j,j,\xi}}) = \lambda_{j,j,\xi} E_{jj}\tilde\eta_\xi$  for any $\xi\in H_2$ and $j=1,2,\dots,k_2$, where $\zeta_{1,1,\xi}=\zeta_{1,2,\xi}\zeta_{2,1,e}$, $\lambda_{1,1,\xi}=\frac{\sigma_2(e,e)}{\sigma_1(\zeta_{1,2,\xi}, \zeta_{2,1,e})}$, $\zeta_{i,i,\xi}=\zeta_{i,1,e}\zeta_{1,i,\xi}$ and $\lambda_{i,i,\xi}=\frac{\sigma_2(e,e)}{\sigma_1(\zeta_{i,1,e}, \zeta_{1,i,\xi})}$, for $i=2,\dots,k_2$,
	\item $\varphi(E_{\alpha(l)\alpha(l)}\eta_e) = \frac{\sigma_2(e,e)}{\sigma_1(e,e)}  E_{ll}\tilde\eta_e$ for all $l=1,2,\dots,k_2$,
	\end{enumerate}
for some $\alpha\in Sym(k_1)$ and $\zeta_{1,i,\xi}, \zeta_{i,1,e}\in H_1$ which depend on $\varphi$.

Moreover, if $k_1=1$, then $k_2=1$. In addition, if $k_2=1$, then there exists a $\mathsf{G}$-graded epimorphism $\psi: \mathfrak{B}_1 \rightarrow \mathfrak{B}_2$ such that $\psi(E_{rr}\eta_{\zeta_\xi}) = \tilde\eta_\xi$, $\xi\in H_2$, for some $r\in\{1,\dots,k_1\}$ and $\zeta_\xi\in H_1$.
\end{lemma}
\begin{proof}
Let us use the diagrams (\ref{4.14}) and (\ref{4.15}) from Lemma \ref{4.18} to construct the desired graded epimorphism. Constructing $\bar\varphi_{\mathfrak{B}_2}$, $\bar\varphi_{\mathfrak{B}_1}$, $\psi$ and $\pi$ similar to $\bar\varphi_{\mathfrak{B}}$, $\bar\varphi_{\mathfrak{A}}$, $\psi$ and $\pi$ from the proof of Lemma \ref{4.18}, respectively, we have that
\begin{equation}\nonumber
\mathfrak{B}_1 \stackrel{\bar\varphi_{\mathfrak{B}_1}^{-1}}{\longrightarrow} \displaystyle\frac{\mathbb{F}\langle X^{\mathsf{G}} \rangle}{\mathsf{T}^{\mathsf{G}}(\mathfrak{B}_1)} \stackrel{\pi}{\longrightarrow} \displaystyle\frac{\frac{\mathbb{F}\langle X^{\mathsf{G}} \rangle}{\mathsf{T}^{\mathsf{G}}(\mathfrak{B}_1)}}{\frac{\mathsf{T}^{\mathsf{G}}(\mathfrak{B}_2)}{\mathsf{T}^{\mathsf{G}}(\mathfrak{B}_1)}} \stackrel{\psi^{-1}}{\longrightarrow} \displaystyle\frac{\mathbb{F}\langle X^{\mathsf{G}} \rangle}{\mathsf{T}^{\mathsf{G}}(\mathfrak{B}_2)} \stackrel{\bar\varphi_{\mathfrak{B}_2}}{\longrightarrow} \mathfrak{B}_2 \ ,
\end{equation}
and so the maps $\varphi: \mathfrak{B}_1\rightarrow \mathfrak{B}_2$ defined by $\varphi \coloneqq \bar\varphi_{\mathfrak{B}_2}\circ\psi^{-1}\circ\pi\circ\bar\varphi_{\mathfrak{B}_1}^{-1}$ is a $\mathsf{G}$-graded epimorphism of $\mathsf{G}$-graded algebras.

Let us now construct another $\mathsf{G}$-graded epimorphism of $\mathfrak{B}_1$ in $\mathfrak{B}_2$ which satisfies the desired conclusions of the lemma.

First, assume $\mathfrak{B}_1$ with an elementary-canonical $\mathsf{G}$-grading defined by a $k_1$-tuple $\theta_1=(\theta_{11},\dots,\theta_{1k_1})\in \mathsf{G}^{k_1}$, and $\mathfrak{B}_2$ with an elementary-canonical $\mathsf{G}$-grading defined by a $k_2$-tuple $\theta_2=(\theta_{21},\dots,\theta_{2k_2})\in \mathsf{G}^{k_2}$. Observe that $\mathfrak{B}_1=\langle E_{1i}\eta_\zeta, E_{i1}\eta_e : i=2,\dots,k_1, \zeta\in H_1\rangle_\mathbb{F}$ and $\mathfrak{B}_2=\langle E_{1j}\tilde\eta_\xi, E_{j1}\tilde\eta_e : j=2,\dots,k_2, \xi\in H_2\rangle_\mathbb{F}$, where ``$\langle S \rangle_\mathbb{F}$'' means ``the subalgebra generated by the subset $S$''. Take a maps $f_2: X^\mathsf{G}\rightarrow \mathfrak{B}_2$ such that $f_2(X_{\varsigma})\subset(\mathfrak{B}_2)_\varsigma$ for any $\varsigma\in \mathsf{G}$, and $f_2$ satisfies $x_i^{(\theta_{21}^{-1}\xi\theta_{2i})}\mapsto E_{1i}\tilde\eta_\xi$ and $x_{k_2+j}^{(\theta_{2j}^{-1}\theta_{21})}\mapsto E_{j1}\tilde\eta_e $, for any $\xi\in H_2$ and $i,j\in\{2,\dots,k_2\}$. By the universal property of $\mathbb{F}\langle X^\mathsf{G} \rangle$, there exists a $\mathsf{G}$-graded homomorphism $\varphi_2: \mathbb{F}\langle X^\mathsf{G} \rangle \rightarrow \mathfrak{B}_2$ which extends $f_2$. Observe that $\varphi_2$ is surjective.

On the other hand, by first part of this proof, because $\varphi$ is a graded epimorphism, there exist homogeneous $a_{1,2,\xi},\dots,a_{1,k_2,\xi},a_{2,1,e},\dots,a_{k_2, 1,e}\in\mathfrak{B}_1$, $\xi\in H_2$, such that $\varphi(a_{1,i,\xi})=E_{1i}\tilde\eta_\xi$ and $\varphi(a_{j,1,e})=E_{j1}\tilde\eta_e$, for any $\xi\in H_2$ and $i,j=2,\dots,k_2$. As the set $\{E_{1i}\tilde\eta_\xi, E_{j1}\tilde\eta_e : i,j=2,\dots,k_2, \xi\in H_2\}\subset\mathfrak{B}_2$ is linearly independent over $\mathbb{F}$, we have that the subset $\{a_{1,i,\xi}+\mathsf{ker}(\varphi), a_{j,1,e}+\mathsf{ker}(\varphi): i,j=2,\dots,k_2, \xi\in H_2\}$ of $\mathfrak{B}_1/\mathsf{ker}(\varphi)$ is linearly independent over $\mathbb{F}$. From this, for each $\xi\in H_2$ and $i,j=2,\dots,k_2$, there exist $r_{i,\xi},s_{i,\xi},p_j,q_j\in \{1,\dots,k_1\}$ and $\zeta_{i,\xi}, \mu_j \in H_1$ such that $\mathsf{deg}_\mathsf{G}(E_{r_{i,\xi} s_{i,\xi}}\eta_{\zeta_{i,\xi}})=\mathsf{deg}_\mathsf{G}(E_{1 i}\tilde\eta_{\xi})$, $\mathsf{deg}_\mathsf{G}(E_{p_j q_j}\eta_{\mu_j})=\mathsf{deg}_\mathsf{G}(E_{j 1}\tilde\eta_e)$, and $\{E_{r_{i,\xi} s_{i,\xi}}\eta_{\zeta_{i,\xi}}+\mathsf{ker}(\varphi),E_{p_j q_j}\eta_{\mu_j}+\mathsf{ker}(\varphi): i,j=2,\dots,k_2 , \xi\in H_2 \}\subset\mathfrak{B}_1/\mathsf{ker}(\varphi)$ is linearly independent over $\mathbb{F}$. Now, take a maps $f_1: X^\mathsf{G}\rightarrow \mathfrak{B}_1$ such that $f_1(X_{\varsigma})\subset(\mathfrak{B}_1)_\varsigma$ for any $\varsigma\in \mathsf{G}$, and $f_1$ satisfies $x_i^{(\theta_{21}^{-1}\xi\theta_{2i})}\mapsto E_{r_{i,\xi} s_{i,\xi}}\eta_{\zeta_{i,\xi}}$ and $x_{k_2+j}^{(\theta_{2j}^{-1}\theta_{21})}\mapsto E_{p_j q_j}\eta_{\mu_j} $, for any $\xi\in H_2$ and $i,j\in\{2,\dots,k_2\}$, and $x_{2k_2+t}^{(\theta_{11}^{-1}\zeta\theta_{1t})}\mapsto E_{1t}\eta_\zeta$ and $x_{2k_2+k_1+t}^{(\theta_{1t}^{-1}\theta_{11})}\mapsto E_{t1}\eta_e $, for any $\zeta\in H_1$ and $t\in\{2,\dots,k_1\}$. By the universal property of $\mathbb{F}\langle X^\mathsf{G} \rangle$, there exists a $\mathsf{G}$-graded homomorphism $\varphi_1: \mathbb{F}\langle X^\mathsf{G} \rangle \rightarrow \mathfrak{B}_1$ which extends $f_1$. Observe that $\varphi_1$ is surjective.

Finally, as in the diagram (\ref{4.15}) (and using Lemma \ref{4.13}), let $\bar\varphi_1: \displaystyle\frac{\mathbb{F}\langle X^{\mathsf{G}} \rangle}{\mathsf{T}^{\mathsf{G}}(\mathfrak{B}_1)}\rightarrow \mathfrak{B}_1$, $\bar\varphi_2: \displaystyle\frac{\mathbb{F}\langle X^{\mathsf{G}} \rangle}{\mathsf{T}^{\mathsf{G}}(\mathfrak{B}_2)}\rightarrow \mathfrak{B}_2$, and $\psi: \displaystyle\frac{\mathbb{F}\langle X^{\mathsf{G}} \rangle}{\mathsf{T}^{\mathsf{G}}(\mathfrak{B}_2)}\rightarrow \displaystyle\frac{\frac{\mathbb{F}\langle X^{\mathsf{G}} \rangle}{\mathsf{T}^{\mathsf{G}}(\mathfrak{B}_1)}}{\frac{\mathsf{T}^{\mathsf{G}}(\mathfrak{B}_2)}{\mathsf{T}^{\mathsf{G}}(\mathfrak{B}_1)}}$ be $\mathsf{G}$-graded isomorphisms that satisfy $\bar\varphi_1(a+\mathsf{T}^{\mathsf{G}}(\mathfrak{B}_1))=\varphi_1(a)$, $\bar\varphi_2(a+\mathsf{T}^{\mathsf{G}}(\mathfrak{B}_2))=\varphi_2(a)$, and $\psi(a+\mathsf{T}^{\mathsf{G}}(\mathfrak{B}_2))=(a+\mathsf{T}^{\mathsf{G}}(\mathfrak{B}_1))+ \frac{\mathsf{T}^{\mathsf{G}}(\mathfrak{B}_2)}{\mathsf{T}^{\mathsf{G}}(\mathfrak{B}_1)}$ for any $a\in \mathbb{F}\langle X^{\mathsf{G}} \rangle$, which are well defined because $\mathsf{ker}(\varphi_1)=\mathsf{T}^\mathsf{G}(\mathfrak{B}_1)$, $\mathsf{ker}(\varphi_2)=\mathsf{T}^\mathsf{G}(\mathfrak{B}_2)$ and $\mathsf{T}^\mathsf{G}(\mathfrak{B}_1)\subseteq\mathsf{T}^\mathsf{G}(\mathfrak{B}_2)$, and let $\pi:\displaystyle\frac{\mathbb{F}\langle X^{\mathsf{G}} \rangle}{\mathsf{T}^{\mathsf{G}}(\mathfrak{B}_1)}\rightarrow\displaystyle\frac{\frac{\mathbb{F}\langle X^{\mathsf{G}} \rangle}{\mathsf{T}^{\mathsf{G}}(\mathfrak{B}_1)}}{\frac{\mathsf{T}^{\mathsf{G}}(\mathfrak{B}_2)}{\mathsf{T}^{\mathsf{G}}(\mathfrak{B}_1)}}$ be the projection homomorphism. Let $\hat\varphi: \mathfrak{B}_1\rightarrow \mathfrak{B}_2$ be the maps defined by $\hat\varphi \coloneqq \bar\varphi_2\circ\psi^{-1}\circ\pi\circ\bar\varphi_1^{-1}$. By this construction, $\hat\varphi$ is a $\mathsf{G}$-graded epimorphism of $\mathsf{G}$-graded algebras.

By definitions of the maps $\hat\varphi$, $\varphi_1$ and $\varphi_2$, it is not difficult to see that $\hat\varphi(E_{p_jq_j}\eta_{\mu_j})=E_{j1}\tilde\eta_e$ and $\hat\varphi(E_{r_{i,\xi} s_{i,\xi}}\eta_{\zeta_{i,\xi}})=E_{1i}\tilde\eta_\xi$, for any $\xi\in H_2$ and $i,j=2,\dots,k_2$. Consequently, for any $\xi\in H_2$ and $i,j=1,\dots,k_2$, we have that 
\begin{equation}\nonumber
	\begin{split}
E_{ji}\tilde\eta_\xi &= (\sigma_2(e,\xi))^{-1} (E_{j1}\tilde\eta_e) (E_{1i}\tilde\eta_\xi)=(\sigma_2(e,\xi))^{-1}\hat\varphi(E_{p_j q_j}\eta_{\mu_j}) \hat\varphi(E_{r_{i,\xi} s_{i,\xi}}\eta_{\zeta_{i,\xi}}) \\
	&=(\sigma_2(e,\xi))^{-1}\hat\varphi(E_{p_j q_j}\eta_{\mu_j} \cdot E_{r_{i,\xi} s_{i,\xi}}\eta_{\zeta_{i,\xi}}) \ .
	\end{split}
\end{equation}
From this, since $E_{ji}\tilde\eta_\xi \neq 0$, we must have that $q_j=r_{i,\xi}$ for all $i,j=2,\dots,k_2$ and $\xi\in H_2$. And hence, $E_{ji}\tilde\eta_\xi=(\sigma_2(e,\xi))^{-1}\sigma_1(\mu_j,\zeta_{i,\xi})\hat\varphi(E_{p_j s_{i,\xi}}\eta_{\mu_j\zeta_{i,\xi}})$. In addition, as $(E_{ii}\tilde\eta_\xi)^2=\sigma_2(\xi,\xi) E_{ii}\tilde\eta_{\xi^2}\neq0$, we must have $p_i=s_{i,\xi}$ for all $i=2,\dots,k_2$ and $\xi\in H_2$, and so $E_{ii}\tilde\eta_\xi=(\sigma_2(e,\xi))^{-1}\sigma_1(\mu_i,\zeta_{i,\xi})\hat\varphi(E_{p_i p_i}\eta_{\mu_i\zeta_{i,\xi}})$ for $i=2,\dots,k_2$ and $\xi\in H_2$; and 
\begin{equation}\nonumber
	\begin{split}
	E_{11}\tilde\eta_\xi &=(\sigma_2(\xi,e))^{-1}(E_{12}\tilde\eta_\xi) (E_{21}\tilde\eta_e)=(\sigma_2(\xi,e))^{-1} \hat\varphi(E_{r_{2,\xi} s_{2,\xi}}\eta_{\zeta_{2,\xi}})  \hat\varphi(E_{p_2 q_2}\eta_{\mu_2}) \\
	&=(\sigma_2(\xi,e))^{-1} \sigma_1(\zeta_{2,\xi},\mu_2) \hat\varphi(E_{q_2 q_2} \eta_{\zeta_{2,\xi}\mu_2}) 
	\ .
	\end{split}
\end{equation}
because $p_2=s_{2,\xi}$ and $q_2=r_{2,\xi}$ for any $\xi\in H_2$. From this, when $\xi=e$, it is easy to see that $E_{11}\tilde\eta_e=(\sigma_2(e,e))^{-1} \sigma_1(\mu_2^{-1},\mu_2) \hat\varphi(E_{q_2 q_2} \eta_e)$, where $\zeta_{2,e}=\mu_2^{-1}$ and $\sigma_1(\mu_2^{-1},\mu_2)=\sigma_1(e,e)$; and for $i=2,\dots,k_2$, $E_{ii}\tilde\eta_e=(\sigma_2(e,e))^{-1}\sigma_1(\mu_i,\mu_i^{-1})\hat\varphi(E_{p_i p_i}\eta_e)$, where $\mu_i^{-1}=\zeta_{i,e}$ and $\sigma_1(\mu_i,\mu_i^{-1})=\sigma_1(e,e)$. Consequently $p_2\neq q_2$, and so there exists $\alpha\in Sym(k_1)$ such that $\hat\varphi(E_{\alpha(i)\alpha(i)}\eta_e) = \sigma_1(e,e)(\sigma_2(e,e))^{-1}E_{ii}\tilde\eta_e$. Therefore, the first part of the result follows.

To prove the last part of the lemma, suppose first that $k_1=1$. Hence, $\mathfrak{B}_1=\mathbb{F}^{\sigma_1}[H_1]$ is a division $\mathsf{G}$-graded algebra. By the first part of this proof, we have that $\mathfrak{B}_1/\mathsf{ker}(\varphi)\cong_\mathsf{G}\mathfrak{B}_2$. Consequently, $\mathfrak{B}_2$ is a graded division algebra, and so $k_2=1$, otherwise $E_{22}\tilde\eta_e$ should be an invertible matrix in $\mathfrak{B}_2$. Now, suppose that $k_2=1$. Thus, $\mathfrak{B}_2=\mathbb{F}^{\sigma_2}[H_2]$ is a division $\mathsf{G}$-graded algebra. Consider a map $f_{2'}: X^{\mathsf{G}}\rightarrow \mathfrak{B}_2$ such that $f_{2'}(X_{\varsigma})\subset(\mathfrak{B}_2)_{\varsigma}$ for any $\varsigma\in \mathsf{G}$, and $f_{2'}$ satisfies $f_{2'}(x_1^{(\xi)})=\tilde\eta_\xi$ for any $\xi\in H_2$. In turn, since $\varphi$ is a graded epimorphism, for each $\xi\in H_2$, there exist $r_\xi,s_\xi\in\{1,\dots,k_1\}$ and $\zeta_\xi\in H_1$ such that $E_{r_\xi s_\xi}\eta_{\zeta_\xi}\in (\mathfrak{B}_1)_\xi$. Take a maps $f_{1'}: X^\mathsf{G}\rightarrow \mathfrak{B}_1$ such that $f_{1'}(X_{\varsigma})\subset(\mathfrak{B}_1)_\varsigma$ for any $\varsigma\in \mathsf{G}$, and $f_{1'}$ satisfies $x_1^{(\xi)}\mapsto E_{r_{\xi} s_{\xi}}\eta_{\zeta_{\xi}}$ for any $\xi\in H_2$, and $x_{t}^{(\theta_{11}^{-1}\zeta\theta_{1t})}\mapsto E_{1t}\eta_\zeta$ and $x_{k_1+t}^{(\theta_{1t}^{-1}\theta_{11})}\mapsto E_{t1}\eta_e $, for any $\zeta\in H_1$ and $t\in\{2,\dots,k_1\}$. As discussed previously, in consequence of the universal property of $\mathbb{F}\langle X^\mathsf{G} \rangle$, there exist $\mathsf{G}$-graded isomorphisms $\bar\varphi_\mathfrak{2'}: \frac{\mathbb{F}\langle X^{\mathsf{G}} \rangle}{\mathsf{T}^{\mathsf{G}}(\mathfrak{B}_2)}\rightarrow \mathfrak{B}_2$ and $\bar\varphi_{1'}: \frac{\mathbb{F}\langle X^{\mathsf{G}} \rangle}{\mathsf{T}^{\mathsf{G}}(\mathfrak{B}_1)}\rightarrow \mathfrak{B}_1$ which satisfy $\bar\varphi_\mathfrak{2'}(x_1^{(\xi)}+\mathsf{T}^{\mathsf{G}}(\mathfrak{B}_2))=\tilde\eta_\xi$ for any $\xi\in H_2$ and $\bar\varphi_\mathfrak{1'}(x_1^{(\xi)}+\mathsf{T}^{\mathsf{G}}(\mathfrak{B}_1))=E_{r_{\xi} s_{\xi}}\eta_{\zeta_{\xi}}$ for any $\xi\in H_2$. Let $\breve\varphi: \mathfrak{B}_1\rightarrow \mathfrak{B}_2$ be the maps defined by $\breve\varphi \coloneqq \bar\varphi_{2'}\circ\psi^{-1}\circ\pi\circ\bar\varphi_{1'}^{-1}$. It is easy to see that $\breve\varphi$ is a $\mathsf{G}$-graded epimorphism of $\mathsf{G}$-graded algebras such that $\breve\varphi(E_{r_{\xi} s_{\xi}}\eta_{\zeta_{\xi}})=\tilde\eta_\xi$. Note that $r_{\xi}=s_{\hat\xi}$ for any $\xi,{\hat\xi}\in H_2$, since $\breve\varphi(E_{r_{\xi} s_{\xi}}\eta_{\zeta_{\xi}}\cdot E_{r_{\hat\xi} s_{\hat\xi}}\eta_{\zeta_{\hat\xi}})=\breve\varphi(E_{r_{\xi} s_{\xi}}\eta_{\zeta_{\xi}})\cdot\breve\varphi(E_{r_{\hat\xi} s_{\hat\xi}}\eta_{\zeta_{\hat\xi}})=\tilde\eta_\xi\cdot \tilde\eta_{\hat\xi}=\sigma_2(\xi,{\hat\xi})\tilde\eta_{\xi{\hat\xi}}\neq0$. The result follows.
\end{proof}

By Lemma \ref{4.08}, it is easy to conclude that $\varphi(E_{\alpha(i)\alpha(j)}\eta_{\zeta_{i,1,e}\zeta_{1,j,\xi}}) = \frac{\sigma_2(e,e)}{\sigma_1(\zeta_{i,1,e}, \zeta_{1,j,\xi})}  E_{ij}\tilde\eta_\xi$ for any $i,j=2,\dots,k_2$ and $\xi\in H_2$. Furthermore, in Lemmas \ref{4.18} and \ref{4.08}, the hypothesis ``$\mathsf{T}^{\mathsf{G}}(\mathfrak{A})\subseteq \mathsf{T}^{\mathsf{G}}(\mathfrak{B})$'' can be replaced by ``$\mathfrak{B}\stackrel{\mathsf{G}}{\hookrightarrow} \mathfrak{A}$'', since the latter implies the former.

\begin{lemma}\label{4.01}
Let $\mathbb{F}$ be a field and $\mathsf{G}$ a group. For $i=1,2$, consider $\mathfrak{B}_i=M_{k_i}(\mathbb{F}^{\sigma_i}[H_i])$ the ${k_i}\times {k_i}$ matrix algebra over $\mathbb{F}^{\sigma_i}[H_i]$ with an elementary-canonical $\mathsf{G}$-grading defined by a $k_i$-tuple $\theta_i=(\theta_{i1},\dots,\theta_{i{k_i}})\in\mathsf{G}^{k_i}$, where $H_i$ is a subgroup of $\mathsf{G}$, $\sigma_i\in \mathcal{H}^2(H_i,\mathbb{F}^*)$ is a $2$-cocycle. Suppose that all $\theta_{ij}$'s belong to normalizer of $H_i$ in $\mathsf{G}$. If $\mathfrak{B}_1 \stackrel{\mathsf{G}}{\hookrightarrow} \mathfrak{B}_2$, then $k_1\leq k_2$, $\mathbb{F}^{\sigma_1}[H_1] \stackrel{\mathsf{G}}{\hookrightarrow} \mathbb{F}^{\sigma_2}[H_2]$ and there exist $\alpha\in Sym(k_2)$, $\delta\in \mathsf{N}_{\mathsf{G}}(H_2)$ and $\xi_1,\dots,\xi_{k_2}\in H_2$ such that $\theta_{1j}=\delta\xi_{j}\theta_{2 \alpha(j)}$ for all $j=1,\dots,k_1$.
\end{lemma}
\begin{proof}
Let us first suppose $k_1,k_2>1$. Since $\mathfrak{B}_1 \stackrel{\mathsf{G}}{\hookrightarrow} \mathfrak{B}_2$, we have that $\mathsf{T}^{\mathsf{G}}(\mathfrak{B}_2)\subseteq\mathsf{T}^{\mathsf{G}}(\mathfrak{B}_1)$. Hence, by Lemma \ref{4.08}, there exists a $\mathsf{G}$-graded epimorphism $\varphi: \mathfrak{B}_2 \rightarrow \mathfrak{B}_1$ satisfying the 3 conditions below:
	\begin{enumerate}[{\it i)}]
	\item $\varphi(E_{\alpha(1)\alpha(i)}\tilde\eta_{\zeta_{1,i,\xi}}) = E_{1i}\eta_\xi$ and $\varphi(E_{\alpha(i)\alpha(1)}\tilde\eta_{\zeta_{i,1,e}}) = E_{i1}\eta_e$, for any $i=2,\dots,k_1$ and $\xi\in H_1$;
	\item $\varphi(E_{\alpha(j)\alpha(j)}\tilde\eta_{\zeta_{j,j,\xi}}) = \lambda_{j,j,\xi} E_{jj}\eta_\xi$  for any $\xi\in H_1$ and $j=1,2,\dots,k_1$, where $\zeta_{1,1,\xi}=\zeta_{1,2,\xi}\zeta_{2,1,e}$, $\lambda_{1,1,\xi}=\frac{\sigma_1(e,e)}{\sigma_2(\zeta_{1,2,\xi}, \zeta_{2,1,e})}$, $\zeta_{i,i,\xi}=\zeta_{i,1,e}\zeta_{1,i,\xi}$ and $\lambda_{i,i,\xi}=\frac{\sigma_1(e,e)}{\sigma_2(\zeta_{i,1,e}, \zeta_{1,i,\xi})}$, for $i=2,\dots,k_1$;
	\item $\varphi(E_{\alpha(l)\alpha(l)}\tilde\eta_e) = \frac{\sigma_1(e,e)}{\sigma_2(e,e)} E_{ll}\eta_e$ for all $l=1,2,\dots,k_1$,
	\end{enumerate}
for some $\alpha\in Sym(k_2)$ and $\zeta_{1,i,\xi}, \zeta_{i,1,e}\in H_2$ which depend on $\varphi$.

The item $iii)$ above ensures that $k_2\geq k_1$. By item $ii)$ above, we have that $\varphi(E_{\alpha(1)\alpha(1)}\tilde\eta_{\zeta_{1,1,\xi}}) = \frac{\sigma_1(e,e)}{\sigma_2(\zeta_{1,2,\xi}, \zeta_{2,1,e})} E_{11}\eta_\xi$ for any $\xi\in H_1$, and so $\theta_{11}^{-1}\xi\theta_{11}=\theta_{2\alpha(1)}^{-1}\zeta_{1,1,\xi}\theta_{2\alpha(1)}\in H_2$ because $\theta_{2\alpha(1)}\in\mathsf{N}_{\mathsf{G}}(H_2)$ and $\zeta_{1,1,\xi}=\zeta_{1,2,\xi}\zeta_{2,1,e}\in H_2$. Consequently, $\theta_{11}^{-1} H_1 \theta_{11}\subseteq H_2$. As $\theta_{11}\in\mathsf{N}_{\mathsf{G}}(H_1)$, we have that $\theta_{11}^{-1}H_1 \theta_{11}=H_1$, and thus, $H_1\subseteq H_2$.

On the other hand, by Lemma \ref{4.09} and its proof, we can assume, without loss of generality, that $\theta_{2 \alpha(1)}=\theta_{1 1}=e$, and so $\zeta_{1,1,\xi}=\xi$, i.e. $\varphi(E_{\alpha(1)\alpha(1)}\tilde\eta_\xi) = \lambda_{1,1,\xi} E_{11}\eta_\xi$ for any $\xi\in H_1$, where $\lambda_{1,1,\xi}=\frac{\sigma_1(e,e)}{\sigma_2(\zeta_{1,2,\xi}, \zeta_{2,1,e})}$ (see item {\it ii)} above). Hence, for any $\xi,{\hat\xi}\in H_1$, we have that
\begin{equation}\nonumber
	\begin{split}
	\varphi(E_{\alpha(1)\alpha(1)}\tilde\eta_\xi \cdot E_{\alpha(1)\alpha(1)}\tilde\eta_{\hat\xi}) &=\sigma_2(\xi,{\hat\xi})\varphi(E_{\alpha(1)\alpha(1)}\tilde\eta_{\xi{\hat\xi}})=\sigma_2(\xi,{\hat\xi})\lambda_{1,1,\xi{\hat\xi}} E_{11}{\eta}_{\xi{\hat\xi}} \\
	&=\sigma_2(\xi,{\hat\xi})\lambda_{1,1,\xi{\hat\xi}}\sigma_1(\xi,{\hat\xi})^{-1}\lambda_{1,1,{\hat\xi}}^{-1}\lambda_{1,1,\xi}^{-1}\left(\lambda_{1,1,\xi} E_{11}{\eta}_{\xi}\right)\left(\lambda_{1,1,{\hat\xi}}\ E_{11}{\eta}_{{\hat\xi}}\right) \\
	&=\frac{\lambda_{1,1,\xi{\hat\xi}}}{\lambda_{1,1,\xi}\lambda_{1,1,{\hat\xi}}} \cdot \frac{\sigma_2(\xi,{\hat\xi})}{\sigma_1(\xi,{\hat\xi})} \varphi(E_{\alpha(1)\alpha(1)}\tilde\eta_\xi) \cdot \varphi(E_{\alpha(1)\alpha(1)}\tilde\eta_{\hat\xi}) \ .
	\end{split}
\end{equation}
As $\varphi$ is a homomorphism of algebras, it follows that $\displaystyle\frac{\sigma_2(\xi,{\hat\xi})}{\sigma_1(\xi,{\hat\xi})}=\frac{\lambda_{1,1,\xi}\lambda_{1,1,{\hat\xi}}}{\lambda_{1,1,\xi{\hat\xi}}}$ for any $\xi,{\hat\xi}\in H_1$, i.e. $[\sigma_2]_{H_1}=[\sigma_1]$. Thus, by Lemma \ref{4.16}, we conclude that $\mathbb{F}^{\sigma_1}[H_1]\stackrel{\mathsf{G}}{\hookrightarrow} \mathbb{F}^{\sigma_2}[H_2]$. Furthermore, by item $i)$ above, we have that $\varphi(E_{\alpha(i)\alpha(1)}\tilde\eta_{\zeta_{i,1,e}}) = E_{i1}\eta_e$ for all $i=2,\dots,k_1$. Hence, it follows that $\theta_{2\alpha(i)}^{-1}\zeta_{i,1,e}\theta_{2\alpha(1)} = \theta_{1i}^{-1}\theta_{11}$, and so $\theta_{1i}=\theta_{11}\theta_{2\alpha(1)}^{-1}\zeta_{i,1,e}^{-1}\theta_{2\alpha(i)}$ for all $i=2,\dots,k_1$. Put $\delta=\theta_{11}\theta_{2\alpha(1)}^{-1}=e\in\mathsf{N}_{\mathsf{G}}(H_2)$ and $\xi_i=\zeta_{i,1,e}^{-1}\in H_1$ for all $i=2,\dots,k_1$, and taking $\xi_1=e$, we have that $\theta_{1i}=\delta\xi_i\theta_{2 \alpha(i)}$ for all $i=1,2,\dots,k_1$. Therefore, the result is true when $k_1,k_2>1$.

Finally, again by Lemma \ref{4.08}, if $k_2=1$, then $k_1=1$, and so the result is obvious. Assume that $k_1=1$. Once more by Lemma \ref{4.08}, there exists a $\mathsf{G}$-graded epimorphism $\psi: \mathfrak{B}_2 \rightarrow \mathfrak{B}_1$ such that $\psi(E_{r r}\tilde\eta_{\zeta_\xi}) = \eta_\xi$, $\xi\in H_1$, for some $\zeta_\xi\in H_2$ and $r\in\{1,\dots,k_2\}$. In this case, the result follows similarly to 
the reasoning above.
\end{proof}
%

A generalization of Proposition \ref{4.04} is given below. The next theorem is an (almost) immediate consequence from Lemmas \ref{4.09} and \ref{4.01} and Proposition \ref{4.04}.

\begin{theorem}\label{4.02}
Let $\mathbb{F}$ be a field and $\mathsf{G}$ a group. For $i=1,2$, consider $\mathfrak{B}_i=M_{k_i}(\mathbb{F}^{\sigma_i}[H_i])$ the ${k_i}\times {k_i}$ matrix algebra over $\mathbb{F}^{\sigma_i}[H_i]$ with an elementary-canonical $\mathsf{G}$-grading defined by a $k_i$-tuple $\theta_i=(\theta_{i1},\dots,\theta_{i{k_i}})\in\mathsf{G}^{k_i}$, where $H_i$ is a subgroup of $\mathsf{G}$, $\sigma_i\in \mathcal{H}^2(H_i,\mathbb{F}^*)$ is a $2$-cocycle. Suppose that $\theta_{i1},\dots,\theta_{ik_i}\in \mathsf{N}_{\mathsf{G}}(H_i)$, for $i=1,2$. Then $\mathfrak{B}_1 \cong_\mathsf{G} \mathfrak{B}_2$ iff $k_1=k_2$, $\mathbb{F}^{\sigma_1}[H_1]\cong_\mathsf{G} \mathbb{F}^{\sigma_2}[H_2]$, and $\Lambda^{H_1}_{\theta_1}=\Lambda^{H_2}_{\theta_2}$.
\end{theorem}
\begin{proof}
First, assume that $\mathfrak{B}_1 \cong_{\mathsf{G}} \mathfrak{B}_2$. Hence, $\mathfrak{B}_1 \stackrel{\mathsf{G}}{\hookrightarrow} \mathfrak{B}_2$ and $\mathfrak{B}_2 \stackrel{\mathsf{G}}{\hookrightarrow} \mathfrak{B}_1$. By Lemmas \ref{4.09} and \ref{4.01} and Proposition \ref{4.04}, it is immediate that $k_1=k_2$, $\mathbb{F}^{\sigma_1}[H_1]\cong_\mathsf{G} \mathbb{F}^{\sigma_2}[H_2]$, and $\Lambda^{H_1}_{\theta_1}=\Lambda^{H_2}_{\theta_2}$.

Reciprocally, suppose that $k_1=k_2$, $\mathbb{F}^{\sigma_1}[H_1]\cong_\mathsf{G} \mathbb{F}^{\sigma_2}[H_2]$, and $\Lambda^{H_1}_{\theta_1}=\Lambda^{H_2}_{\theta_2}$. Hence, $H_1=H_2$, $\sigma_2=\varrho\sigma_1$ for some $2$-coboundary $\varrho$ on $H_1$, and $\theta_1\in\Lambda^{H_2}_{\theta_2}$ (and also $\theta_2\in\Lambda^{H_1}_{\theta_1}$). Put $k=k_1=k_2$ and $H=H_1=H_2$. By Lemma \ref{4.09}, we have that $\theta_1$ and $\theta_2$ define equivalent elementary-canonical $\mathsf{G}$-gradings on $M_{k}(\mathbb{F}^{\sigma_2}[H])$ or there exists a $2$-cocycle $\rho$ on $H$ such that $M_{k}(\mathbb{F}^{\sigma_2}[H])\cong_{\mathsf{G}} M_k(\mathbb{F}^\rho[H])$, where $M_k(\mathbb{F}^\rho[H])$ is graded with the elementary-canonical $\mathsf{G}$-grading defined by $\tilde\theta$. Thus, we can assume, without loss of generality, that $\theta_1=\theta_2$, namely equal to $\theta$. Hence, $M_{k}(\mathbb{F}^{\sigma_1}[H])$ and $M_{k}(\mathbb{F}^{\sigma_2}[H])$ are graded by an elementary-canonical $\mathsf{G}$-gradings defined by $\theta$. Now, as $\varrho$ is a $2$-coboundary on $H$, let $f:H\rightarrow\mathbb{F}^*$ be a map such that $\varrho(\xi,\zeta)=\frac{f(\xi\zeta)}{f(\xi)f(\zeta)}$, for any $\xi,\zeta\in H$. Consider the linear map $\psi: M_{k}(\mathbb{F}^{\sigma_1}[H]) \rightarrow M_{k}(\mathbb{F}^{\sigma_2}[H])$ which extends the relation $E_{ij}\eta_\xi\mapsto f(\xi)E_{ij}\tilde\eta_\xi$ for any $i,j=1,\dots,k$ and $\xi\in H$. Observe that $\psi$ is a bijection. Since $\sigma_2=\varrho\sigma_1$, it is easy to prove that $\psi$ is a $\mathsf{G}$-graded homomorphism of algebras (see the proof of Proposition \ref{4.04}). The result follows.
\end{proof}

Under the conditions of Theorem \ref{4.02}, and also using Proposition \ref{4.04}, we have that $M_{k_1}(\mathbb{F}^{\sigma_1}[H_1])\cong_\mathsf{G} M_{k_2}(\mathbb{F}^{\sigma_2}[H_2])$ iff $k_1=k_2$, $H_1=H_2$, $[\sigma_1]=[\sigma_2]$ in $\mathcal{H}^2(H_1,\mathbb{F}^*)$, and $\theta_{1}=\left(\delta\xi_1\theta_{2\alpha(1)}, \dots, \delta\xi_{k_1}\theta_{2\alpha(k_1)} \right)$ for some $\delta\in\mathsf{N}_\mathsf{G}(H_1)$, $\xi_1,\dots,\xi_{k_1}\in H_1$, $\alpha\in Sym(k_1)$.

The main result of this section is presented below. It is a generalization of Imbedding Lemma (Lemma \ref{4.16}).

\begin{theorem}[Imbedding]\label{1.39}
Let $\mathbb{F}$ be a field and $\mathsf{G}$ a group. For $i=1,2$, consider $\mathfrak{B}_i=M_{k_i}(\mathbb{F}^{\sigma_i}[H_i])$ the ${k_i}\times {k_i}$ matrix algebra over $\mathbb{F}^{\sigma_i}[H_i]$ with an elementary-canonical $\mathsf{G}$-grading defined by a $k_i$-tuple $\theta_i=(\theta_{i1},\dots,\theta_{i{k_i}})\in\mathsf{G}^{k_i}$, where $H_i$ is a subgroup of $\mathsf{G}$, $\sigma_i\in \mathcal{H}^2(H_i,\mathbb{F}^*)$ is a $2$-cocycle. Suppose that $\theta_{i1},\dots,\theta_{ik_i}\in \mathsf{N}_{\mathsf{G}}(H_i)$, for $i=1,2$. Then $\mathfrak{B}_1 \stackrel{\mathsf{G}}{\hookrightarrow} \mathfrak{B}_2$ iff $k_1\leq k_2$, $\mathbb{F}^{\sigma_1}[H_1] \stackrel{\mathsf{G}}{\hookrightarrow} \mathbb{F}^{\sigma_2}[H_2]$ and there exist $\alpha\in Sym(k_2)$, $\delta\in \mathsf{N}_{\mathsf{G}}(H_2)$ and $\xi_1,\dots,\xi_{k_1}\in H_2$ such that $\theta_{1j}=\delta\xi_{j}\theta_{2 \alpha(j)}$ for all $j=1,\dots,k_1$.
\end{theorem}
\begin{proof}
By Lemma \ref{4.01}, if $M_{k_1}(\mathbb{F}^{\sigma_1}[H_1]) \stackrel{\mathsf{G}}{\hookrightarrow} M_{k_2}(\mathbb{F}^{\sigma_2}[H_2])$, then $k_1\leq k_2$, $\mathbb{F}^{\sigma_1}[H_1] \stackrel{\mathsf{G}}{\hookrightarrow} \mathbb{F}^{\sigma_2}[H_2]$ and there exist $\alpha\in Sym(k_2)$, $\delta\in \mathsf{N}_{\mathsf{G}}(H_2)$ and $\xi_1,\dots,\xi_{k_1}\in H_2$ such that $\theta_{1j}=\delta\xi_{j}\theta_{2 \alpha(j)}$ for all $j=1,\dots,k_1$.

	On the other hand, suppose that $k_1\leq k_2$, $\mathbb{F}^{\sigma_1}[H_1] \stackrel{\mathsf{G}}{\hookrightarrow} \mathbb{F}^{\sigma_2}[H_2]$ and there exist $\alpha\in Sym(k_2)$, $\delta\in \mathsf{N}_{\mathsf{G}}(H_2)$ and $\xi_1,\dots,\xi_{k_1}\in H_2$ such that $\theta_{1j}=\delta\xi_{j}\theta_{2 \alpha(j)}$ for $j=1,\dots,k_1$. From this, since $[\sigma_1]=[\sigma_2]_{H_1}$, there exists a map $f:H_1\rightarrow \mathbb{F}^*$ such that $\rho(\xi,\zeta)=\frac{f(\xi)f(\zeta)}{f(\xi\zeta)}$ defines a $2$-coboundary $\rho$ which satisfies $\sigma_1=\rho\sigma_2$, and by Lemma \ref{4.09}, we can assume, without restriction of generality, that $\theta_2=(\theta_{11},\dots,\theta_{1k_1},\theta_{2(k_1+1)},\theta_{2{k_2}})$. The homomorphism $\psi$ from the proof of Theorem \ref{4.02} gives us an idea of what a graded monomorphism from $\mathfrak{B}_1$ to $\mathfrak{B}_2$ should be like. Therefore, it is easy to see that the linear maps $\psi:\mathfrak{B}_1\rightarrow \mathfrak{B}_2$ which extends the map $E_{ij}\eta_\xi \mapsto f(\xi) E_{ij}\tilde\eta_{\xi}$ is a $\mathsf{G}$-graded monomorphism of algebras. The result follows.
\end{proof}

Let us conclude this section by presenting two results which combine Theorem \ref{1.39} and Corollary \ref{2.09}. The second of these is for simple graded algebras of finite dimension, the first not necessarily. Furthermore, the second result generalizes Corollary \ref{4.05}.

\begin{corollary}\label{4.07}
Let $\mathbb{F}$ be an algebraically closed field, $\mathsf{G}$ a group, $H$ a subgroup of $\mathsf{G}$, $\sigma\in \mathcal{H}^2(H,\mathbb{F}^*)$ a $2$-cocycle, and $\mathfrak{B}=M_{k}(\mathbb{F}^{\sigma}[H])$ the ${k}\times {k}$ matrix algebra with an elementary-canonical $\mathsf{G}$-grading defined by a $k$-tuple $\theta\in\mathsf{G}^{k}$. For any subgroup $N$ of $\mathsf{G}$ such that $H$ is a central subgroup of $N$ of finite index, there exists a $2$-cocycle $\rho\in\mathcal{Z}^2(N, \mathbb{F}^*)$ such that 
\begin{equation}\nonumber
   \begin{tikzcd}
\mathfrak{B} \arrow[r,hook]{l}{\mathsf{G}}\arrow[d,hook]{l}{\rotatebox{-90}{$\mathsf{G}$}} &
M_{k}(\mathbb{F}^{\rho}[N]) \arrow[d,hook]{l}{\rotatebox{-90}{$\mathsf{G}$}} \\
M_{t}(\mathbb{F}^{\sigma}[H]) \arrow[r,hook]{l}{\mathsf{G}} & M_{t}(\mathbb{F}^{\rho}[N])  
    \end{tikzcd}
\end{equation}
for all integer $t\geq k$, where $M_{k}(\mathbb{F}^{\rho}[N])$ is graded with the elementary-canonical $\mathsf{G}$-grading defined by $\theta$, and $M_{t}(\mathbb{F}^{\sigma}[H])$ and $M_{t}(\mathbb{F}^{\rho}[N])$ are graded with elementary-canonical $\mathsf{G}$-gradings defined by some $t$-tuple $\phi\in\mathsf{G}^t$.
\end{corollary}
\begin{proof}
Corollary \ref{2.09} ensures the existence of $\rho\in\mathcal{Z}^2(N, \mathbb{F}^*)$ such that $\mathbb{F}^{\sigma}[H] \stackrel{\mathsf{G}}{\hookrightarrow} \mathbb{F}^{\rho}[N]$. Therefore, the result follows from Theorem \ref{1.39}.
\end{proof}

\begin{corollary}\label{4.10}
Let $\mathbb{F}$ be an algebraically closed field, $\mathsf{G}$ a group, and $\mathfrak{B}$ a finite dimensional algebra with a $\mathsf{G}$-grading $\Gamma$. Suppose that $\mathfrak{B}$ is simple graded. If there exists a finite subgroup $H$ of $\mathsf{G}$ such that $\mathsf{Supp}(\Gamma)$ is a central subset of $H$, then there exists a $2$-cocycle $\sigma\in\mathcal{Z}^2(H, \mathbb{F}^*)$ such that $\mathfrak{B}\stackrel{\mathsf{G}}{\hookrightarrow} M_k(\mathbb{F}^{\sigma}[H])$, for some $k\in\mathbb{N}$, where the $\mathsf{G}$-grading on $M_k(\mathbb{F}^{\sigma}[H])$ is the elementary-canonical $\mathsf{G}$-grading defined by some $k$-tuple $\theta\in \mathsf{G}^k$. In addition, if $H= H_1\subseteq H_2\subseteq\cdots\subseteq H_{n}\subseteq\mathsf{G}$ is a chain of subgroups of $\mathsf{G}$ such that $H_i$ is central in $H_{i+1}$ and $[H_{i+1}:H_i]< \infty$ for $i=1,\dots,n-1$, then there exist $2$-cocycles $\sigma_1\in \mathcal{H}^2(H_1,\mathbb{F}^*)$, $\dots$, $\sigma_n\in \mathcal{H}^2(H_n,\mathbb{F}^*)$ such that
\begin{equation}\nonumber
   \begin{tikzcd}
 \mathbb{F} \arrow[r,hook]{l}{\mathsf{G}} \arrow[d,hook]{l}{\rotatebox{-90}{$\mathsf{G}$}} &
\mathbb{F}^{\sigma_1}[H_1] \arrow[r,hook]{l}{\mathsf{G}} \arrow[d,hook]{l}{\rotatebox{-90}{$\mathsf{G}$}} &
 \cdots \arrow[r,hook]{l}{\mathsf{G}} &
\mathbb{F}^{\sigma_{n}}[H_{n}] \arrow[d,hook]{l}{\rotatebox{-90}{$\mathsf{G}$}} \\
M_{k_1}(\mathbb{F})\arrow[r,hook]{l}{\mathsf{G}} \arrow[bend right=60,dd,hook]{l}{\rotatebox{-90}{$\mathsf{G}$}} &
M_{k_1}(\mathbb{F}^{\sigma_1}[H_1]) \arrow[r,hook]{l}{\mathsf{G}} \arrow[d,hook]{l}{\rotatebox{-90}{$\mathsf{G}$}} &
 \cdots \arrow[r,hook]{l}{\mathsf{G}} &
M_{k_1}(\mathbb{F}^{\sigma_{n}}[H_{n}]) \arrow[d,hook]{l}{\rotatebox{-90}{$\mathsf{G}$}} \\
 \mathfrak{B} \arrow[r,hook]{l}{\mathsf{G}} &
M_{k}(\mathbb{F}^{\sigma_1}[H_1]) \arrow[r,hook]{l}{\mathsf{G}} \arrow[d,hook]{l}{\rotatebox{-90}{$\mathsf{G}$}} &
 \cdots \arrow[r,hook]{l}{\mathsf{G}} &
M_{k}(\mathbb{F}^{\sigma_{n}}[H_{n}]) \arrow[d,hook]{l}{\rotatebox{-90}{$\mathsf{G}$}} \\
M_{k_2}(\mathbb{F}) \arrow[r,hook]{l}{\mathsf{G}} &
M_{k_2}(\mathbb{F}^{\sigma_1}[H_1]) \arrow[r,hook]{l}{\mathsf{G}} &
 \cdots \arrow[r,hook]{l}{\mathsf{G}} &
M_{k_2}(\mathbb{F}^{\sigma_{n}}[H_{n}])
    \end{tikzcd}
\end{equation}
for all $1\leq k_1 \leq k \leq k_2$, where each $M_{l}(\mathbb{F}^{\sigma_i}[H_i])$, $l\in\{k,k_1,k_2\}$ and $i=1,\dots,n$, is graded with an elementary-canonical $\mathsf{G}$-grading defined by some $l$-tuple $\theta_l\in \mathsf{G}^{l}$.
\end{corollary}
\begin{proof}
Analogous to the proof of Corollary \ref{4.07} and Corollary \ref{4.05}.
\end{proof}

We conclude by stating that a result similar to Corollary \ref{4.06} is also true for the general case of finite dimensional simple graded algebras, but the reciprocal is not true. Indeed, let $\mathbb{F}$ be an arbitrary field, $\mathsf{G}$ a non-trivial finite group, and $\mathfrak{A}=\mathbb{F}[\mathsf{G}]$ and $\mathfrak{B}=M_2(\mathbb{F})$ two finite dimensional algebras. Suppose $\mathfrak{A}$ with its canonical $\mathsf{G}$-grading and $\mathfrak{B}$ with an elementary-canonical $\mathsf{G}$-grading defined by some $2$-tuple $\theta\in\mathsf{G}^2$. Obviously $\mathfrak{A}$ and $\mathfrak{B}$ are simple graded. Observe that neither $\mathfrak{A} \stackrel{\mathsf{G}}{\hookrightarrow} \mathfrak{B}$ nor $\mathfrak{B} \stackrel{\mathsf{G}}{\hookrightarrow} \mathfrak{A}$ (see Theorem \ref{1.39}), but $M_{2}(\mathbb{F}[\mathsf{G}])$ with some elementary-canonical $\mathsf{G}$-grading is such that $\mathfrak{A}, \mathfrak{B} \stackrel{\mathsf{G}}{\hookrightarrow} M_{2}(\mathbb{F}[\mathsf{G}])$.


\section{Graded Imbeddings and Graded Polynomial Identities of Finite Dimensional Simple Graded Algebras}\label{sec5}

In the previous section, Theorem \ref{1.39} presents conditions that guarantee when two matrix algebras over twisted group algebras can be $\mathsf{G}$-imbedding into each other, where such conditions are related to the form of these matrix algebras. In what follows, let us exhibit some results that, based on the $\mathsf{G}$T-ideals of graded polynomial identities, ensure other conditions for the existence of $\mathsf{G}$-imbeddings between $\mathsf{G}$-simple algebras of finite dimension.

\begin{lemma}[Imbedding]\label{4.20}
Let $\mathbb{F}$ be field, $\mathsf{G}$ a group, $H_1$ and $H_2$ two subgroups of $\mathsf{G}$ and $\sigma_1$ and $\sigma_2$ their $2$-cocycles, respectively. Then $\mathbb{F}^{\sigma_1}[H_1] \stackrel{\mathsf{G}}{\hookrightarrow} \mathbb{F}^{\sigma_2}[H_2]$ iff $\mathsf{T}^{\mathsf{G}}(\mathbb{F}^{\sigma_2}[H_2])\subseteq \mathsf{T}^{\mathsf{G}}(\mathbb{F}^{\sigma_1}[H_1])$.
\end{lemma}
\begin{proof}
Put $\mathfrak{A}=\mathbb{F}^{\sigma_2}[H_2]$ and $\mathfrak{B}=\mathbb{F}^{\sigma_1}[H_1]$. Obviously $\mathfrak{B} \stackrel{\mathsf{G}}{\hookrightarrow} \mathfrak{A}$ implies $\mathsf{T}^{\mathsf{G}}(\mathfrak{A})\subseteq \mathsf{T}^{\mathsf{G}}(\mathfrak{B})$.

Reciprocally, assume $\mathsf{T}^{\mathsf{G}}(\mathfrak{A})\subseteq \mathsf{T}^{\mathsf{G}}(\mathfrak{B})$. By Lemma \ref{4.18}, there exists a $\mathsf{G}$-graded epimorphism $\varphi: \mathfrak{A}\rightarrow\mathfrak{B}$. Consequently, for any $\xi\in H_1$, there exist $\lambda_\xi\in\mathbb{F}^*$ such that $\varphi(\eta_\xi)=\lambda_\xi\tilde\eta_\xi$. Hence, it follows that $H_1\subseteq H_2$ and
	\begin{equation}\nonumber
		\begin{split}
	& \ \varphi(\eta_\xi\eta_\zeta)=\varphi(\sigma_2(\xi,\zeta)\eta_{\xi\zeta})=\sigma_2(\xi,\zeta)\lambda_{\xi\zeta}\tilde\eta_{\xi\zeta}\\
	& \ \ \ \ \ \ \rotatebox{90}{=} \\
	& \varphi(\eta_\xi)\varphi(\eta_\zeta) =\lambda_\xi\tilde\eta_\xi \lambda_\zeta\tilde\eta_\zeta =\lambda_\xi\lambda_\zeta\sigma_1(\xi,\zeta)\tilde\eta_{\xi\zeta}
		\end{split}
	\end{equation}
for any $\xi,\zeta\in H_1$, and so $\sigma_2(\xi,\zeta)=\rho(\xi,\zeta)\sigma_1(\xi,\zeta)$ for any $\xi,\zeta\in H_1$, where $\rho(\xi,\zeta)=\frac{\lambda_\xi\lambda_\zeta}{\lambda_{\xi\zeta}}$ is a $2$-coboundary in $\mathcal{B}^2(H_1,\mathbb{F}^*)$. Therefore, by Lemma \ref{4.16}, we can conclude that $\mathfrak{B} \stackrel{\mathsf{G}}{\hookrightarrow} \mathfrak{A}$.
\end{proof}

Combining the Lemmas of Imbedding (Lemmas \ref{4.16} and \ref{4.20}), we have that the following conditions are equivalent:
\begin{enumerate}[{\it i)}] 
\item $\mathbb{F}^{\sigma_1}[H_1] \stackrel{\mathsf{G}}{\hookrightarrow} \mathbb{F}^{\sigma_2}[H_2]$;
\item $H_1\leq H_2$ and $[\sigma_1]=[\sigma_2]_{H_1}$;
\item $\mathsf{T}^{\mathsf{G}}(\mathbb{F}^{\sigma_2}[H_2])\subseteq \mathsf{T}^{\mathsf{G}}(\mathbb{F}^{\sigma_1}[H_1])$,
\end{enumerate}
where $\mathbb{F}$ is a field, $\mathsf{G}$ is a group, both arbitrary, $H_1$ and $H_2$ are two subgroups of $\mathsf{G}$ and $\sigma_1$ and $\sigma_2$ their $2$-cocycles, respectively.

Now, recall that, in finite dimension and $\mathbb{F}$ algebraically closed, Lemma \ref{teodivgradalg} guarantees the one-to-one relationship between division graded algebras and twisted group algebras. Let us use this fact to prove the proposition below.
 
\begin{proposition}\label{4.11}
Let $\mathfrak{A}$ and $\mathfrak{B}$ be two finite dimensional division $\mathsf{G}$-graded $\mathbb{F}$-algebras, with $\mathbb{F}$ algebraically closed. Then $\mathsf{T}^{\mathsf{G}}(\mathfrak{A})\subseteq \mathsf{T}^{\mathsf{G}}(\mathfrak{B})$ iff $\mathfrak{B} \stackrel{\mathsf{G}}{\hookrightarrow} \mathfrak{A}$. In addition, $\mathsf{T}^{\mathsf{G}}(\mathfrak{A})= \mathsf{T}^{\mathsf{G}}(\mathfrak{B})$ iff $\mathfrak{A} \cong_{\mathsf{G}} \mathfrak{B}$.
\end{proposition}
\begin{proof}
By Lemma \ref{teodivgradalg}, there are finite subgroups $H_\mathfrak{A}$ and $H_\mathfrak{B}$ of $\mathsf{G}$ and $2$-cocycles $\sigma_\mathfrak{A}$ on $H_\mathfrak{A}$ and $\sigma_\mathfrak{B}$ on $H_\mathfrak{B}$, such that $\mathfrak{A}\cong_\mathsf{G}\mathbb{F}^{\sigma_\mathfrak{A}}[H_\mathfrak{A}]$ and $\mathfrak{B}\cong_\mathsf{G}\mathbb{F}^{\sigma_\mathfrak{B}}[H_\mathfrak{B}]$. So the result follows from Lemma \ref{4.20}.
\end{proof}

Using once again the second Imbedding Lemma (Lemmas \ref{4.20}), and Lemma \ref{4.01} and Imbedding Theorem (Theorem \ref{1.39}), in the next lemma we analyse when two matrix algebras (over twisted group algebras) are $\mathsf{G}$-imbeddable into each other from their graded polynomial identities.


\begin{lemma}[Imbedding]\label{4.19}
Let $\mathbb{F}$ be a field and $\mathsf{G}$ a group. For $i=1,2$, consider $\mathfrak{B}_i=M_{k_i}(\mathbb{F}^{\sigma_i}[H_i])$ the ${k_i}\times {k_i}$ matrix algebra over $\mathbb{F}^{\sigma_i}[H_i]$ with an elementary-canonical $\mathsf{G}$-grading defined by a $k_i$-tuple $\theta_i=(\theta_{i1},\dots,\theta_{i{k_i}})\in\mathsf{G}^{k_i}$, where $H_i$ is a subgroup of $\mathsf{G}$, $\sigma_i\in \mathcal{H}^2(H_i,\mathbb{F}^*)$ is a $2$-cocycle. Suppose that $\theta_{i1},\dots,\theta_{ik_i}\in \mathsf{N}_{\mathsf{G}}(H_i)$, for $i=1,2$. Then $\mathfrak{B}_1 \stackrel{\mathsf{G}}{\hookrightarrow} \mathfrak{B}_2$ iff $\mathsf{T}^{\mathsf{G}}(\mathfrak{B}_2)\subseteq \mathsf{T}^{\mathsf{G}}(\mathfrak{B}_1)$.
\end{lemma}
\begin{proof}
First, obviously $\mathfrak{B}_1 \stackrel{\mathsf{G}}{\hookrightarrow} \mathfrak{B}_2$ implies $\mathsf{T}^{\mathsf{G}}(\mathfrak{B}_2)\subseteq \mathsf{T}^{\mathsf{G}}(\mathfrak{B}_1)$. On the other hand, suppose $\mathsf{T}^{\mathsf{G}}(\mathfrak{B}_2)\subseteq \mathsf{T}^{\mathsf{G}}(\mathfrak{B}_1)$. Assume $k_1,k_2>1$. By the first part of the proof of Lemma \ref{4.01}, the following conditions are true: 1) $k_1\leq k_2$; 2) $H_1\subseteq H_2$; 3) $[\sigma_2]_{H_1}=[\sigma_1]$; 4) there exist $\delta\in H_1$, $\alpha\in Sym(k_1)$ and $\xi_1,\dots,\xi_{k_2}\in H_2$ such that $\theta_{2i}=\delta \xi_i\theta_{1\alpha(i)}$. Consequently, by Imbedding Theorem (Theorem \ref{1.39}), it follows that $\mathfrak{B}_1 \stackrel{\mathsf{G}}{\hookrightarrow} \mathfrak{B}_2$. Now, when $k_2=1$, we have that $k_1=1$ (see final part of the proof of Lemma \ref{4.01}), and so the result follows from the second Imbedding Lemma (Lemma \ref{4.20}). If $k_1=1$, the proof is similar to first part of this proof.
\end{proof}

Under the same hypotheses of Lemma \ref{4.19}, we can combine Theorem \ref{1.39} and Lemma \ref{4.19} to prove that the following statements are equivalent:
	\begin{enumerate}[{\it i)}]
	\item $\mathsf{T}^{\mathsf{G}}(M_{k_2}(\mathbb{F}^{\sigma_2}[H_2]))\subseteq \mathsf{T}^{\mathsf{G}}(M_{k_1}(\mathbb{F}^{\sigma_1}[H_1]))$;
	\item $M_{k_1}(\mathbb{F}^{\sigma_1}[H_1]) \stackrel{\mathsf{G}}{\hookrightarrow} M_{k_2}(\mathbb{F}^{\sigma_2}[H_2])$;
	\item $k_1\leq k_2$, $H_1\subseteq H_2$, $[\sigma_1]=[\sigma_2]_{H_1}$ and $\theta_{1i}=\delta\xi_i\theta_{2\alpha(i)}$ for some $\delta\in\mathsf{N}_\mathsf{G}(H_2)$, $\xi_1,\dots,\xi_{k_1}\in H_2$ and $\alpha\in Sym(k_2)$.
	\end{enumerate}

Recall that a non-abelian group is called a \textit{Hamiltonian group} if every subgroup is normal. It is well-known that the quaternion group is the smallest Hamiltonian group. For further details about the Hamiltonian groups, we suggest the book \cite{Rotm10}, p.44, or \cite{Robi96}, p.143. With this in mind, let us go to the next result.

\begin{theorem}[Imbedding]\label{4.21}
Let $\mathsf{G}$ be a group, $\mathbb{F}$ an algebraically closed field, and $\mathfrak{A}$ and $\mathfrak{B}$ two finite dimensional simple $\mathsf{G}$-graded $\mathbb{F}$-algebras. Suppose that either $\mathsf{char}(\mathbb{F}) = 0$ or $\mathsf{char} (\mathbb{F})$ is coprime with the order of each finite subgroup of $\mathsf{G}$, and that any subgroup of $\mathsf{G}$ is normal. Then $\mathsf{T}^{\mathsf{G}}(\mathfrak{A})\subseteq \mathsf{T}^{\mathsf{G}}(\mathfrak{B})$ iff $\mathfrak{B} \stackrel{\mathsf{G}}{\hookrightarrow} \mathfrak{A}$.
\end{theorem}
\begin{proof}
By Lemma \ref{teosimpgradalg}, there are finite subgroups $H_1$ and $H_2$ of $\mathsf{G}$, $2$-cocycles $\sigma_1$ on $H_1$ and $\sigma_2$ on $H_2$, and $k_1,k_2\in\mathbb{N}$ such that $\mathfrak{A}\cong_\mathsf{G} M_{k_1}(\mathbb{F}^{\sigma_1}[H_1])$ and $\mathfrak{B}\cong_\mathsf{G} M_{k_2}(\mathbb{F}^{\sigma_2}[H_2])$, where the $\mathsf{G}$-grading on $M_{k_i} (\mathbb{F}^{\sigma_i} [H_i])$ is an elementary-canonical grading defined by some $k_i$-tuple $\theta_i\in \mathsf{G}^{k_i}$, for $i=1,2$. Therefore, the result follows from Lemma \ref{4.19}.
\end{proof}

Consequently, we have that

\begin{corollary}
Under the hypotheses of Theorem \ref{4.21}, $\mathsf{T}^{\mathsf{G}}(\mathfrak{A})= \mathsf{T}^{\mathsf{G}}(\mathfrak{B})$ iff $\mathfrak{A} \cong_{\mathsf{G}} \mathfrak{B}$.
\end{corollary}

Finally, combining Corollary \ref{4.10} and Theorem \ref{4.21}, when the group $\mathsf{G}$ is abelian, we have that
\begin{corollary}
Let $\mathsf{G}$ be an abelian group, $\mathbb{F}$ an algebraically closed field, and $\mathfrak{A}$ a finite dimensional simple $\mathsf{G}$-graded $\mathbb{F}$-algebra. Suppose that either $\mathsf{char}(\mathbb{F}) = 0$ or $\mathsf{char} (\mathbb{F})$ is coprime with the order of each finite subgroup of $\mathsf{G}$. Then there exist $k\in\mathbb{N}$ and a $2$-cocycle $\sigma\in\mathcal{Z}^2(\mathsf{G},\mathbb{F}^*)$ such that $\mathfrak{A} \stackrel{\mathsf{G}}{\hookrightarrow} M_{k}(\mathbb{F}^{\sigma}[\mathsf{G}])$ and $\mathsf{T}^{\mathsf{G}}(M_{k}(\mathbb{F}^{\sigma}[\mathsf{G}]))\subseteq \mathsf{T}^{\mathsf{G}}(\mathfrak{A})$.
\end{corollary}
%
%



\subsection{The Semisimple Case}

To finalize this text, let us briefly analyse the case of $\mathsf{G}$-imbeddings of semisimple $\mathsf{G}$-graded algebras of finite dimension. Firstly, given any $\mathsf{G}$-graded $\mathbb{F}$-algebra $\mathfrak{A}$ of finite dimension, it is easy to see that $\mathfrak{A}\times\mathfrak{A}\not\stackrel{\mathsf{G}}{\hookrightarrow} \mathfrak{A}$, since $\mathsf{dim}_\mathbb{F}(\mathfrak{A}\times\mathfrak{A})=2\cdot\mathsf{dim}_\mathbb{F}(\mathfrak{A})$, but $\mathsf{T^G}(\mathfrak{A}\times\mathfrak{A}) = \mathsf{T^G}(\mathfrak{A})$ because $\mathsf{T^G}(\mathfrak{A}\times\mathfrak{A})=\mathsf{T^G}(\mathfrak{A})\cap \mathsf{T^G}(\mathfrak{A})$. With this in mind, we present in the next (and last) result with a generalization of Theorem \ref{4.21} for a semisimple graded case:

\begin{proposition}
Let $\mathsf{G}$ be a group, $\mathbb{F}$ an algebraically closed field, $\mathfrak{A}_1, \dots, \mathfrak{A}_r$ and $\mathfrak{B}_1,\dots, \mathfrak{B}_s$ finite dimensional simple $\mathsf{G}$-graded $\mathbb{F}$-algebras. Suppose that either $\mathsf{char}(\mathbb{F}) = 0$ or $\mathsf{char} (\mathbb{F})$ is coprime with the order of each finite subgroup of $\mathsf{G}$, and that any subgroup of $\mathsf{G}$ is normal. If $\mathfrak{B}_i \not\stackrel{\mathsf{G}}{\hookrightarrow} \mathfrak{B}_j$ for all $i\neq j$, then are equivalent:
	\begin{enumerate}[i)]
	\item $\mathsf{T}^{\mathsf{G}}(\mathfrak{A}_1 \times \cdots \times \mathfrak{A}_r)\subseteq \mathsf{T}^{\mathsf{G}}(\mathfrak{B}_1 \times \cdots \times  \mathfrak{B}_s)$;
	\item $\mathfrak{B}_1 \times \cdots \times  \mathfrak{B}_s \stackrel{\mathsf{G}}{\hookrightarrow} \mathfrak{A}_1 \times \cdots \times \mathfrak{A}_r$;
	\item there exist $i_1,\dots,i_r\in\{1,\dots,s\}$, not necessarily distinct, such that $\mathfrak{B}_j \stackrel{\mathsf{G}}{\hookrightarrow} \mathfrak{A}_{i_j}$ for $j=1,\dots,s$.
	\end{enumerate}
\end{proposition}
\begin{proof}
First, by Lemma \ref{teosimpgradalg}, there are finite subgroups $H_1,\dots, H_r, N_1,\dots,N_s$ of $\mathsf{G}$, $2$-cocycles $\sigma_i$ on $H_i$, $i=1,\dots,r$, and $\rho_j$ on $H_j$, $j=1,\dots,s$, and $k_1,\dots,k_r,l_1,\dots,l_s\in\mathbb{N}$ such that $\mathfrak{A}_i\cong_\mathsf{G} M_{k_i}(\mathbb{F}^{\sigma_i}[H_i])$ and $\mathfrak{B}_j\cong_\mathsf{G} M_{l_j}(\mathbb{F}^{\rho_j}[N_j])$, for $i=1,\dots,r$ and $j=1,\dots,s$, where $M_{k_i} (\mathbb{F}^{\sigma_i} [H_i])$ and $M_{l_j}(\mathbb{F}^{\rho_j}[N_j])$ are graded with elementary-canonical gradings defined by tuples $\theta_i\in \mathsf{G}^{k_i}$ and $\hat\theta_j\in \mathsf{G}^{l_j}$.

Obviously {\it iii)} implies {\it ii)}, and {\it ii)} implies {\it i)}. Let us now show that {\it i)} implies {\it iii)}. Indeed, put $\mathfrak{A}=\mathfrak{A}_1, \dots, \mathfrak{A}_r$ and $\mathfrak{B}=\mathfrak{B}_1 \times \cdots \times  \mathfrak{B}_s$, and assume that $\mathsf{T}^{\mathsf{G}}(\mathfrak{A})\subseteq \mathsf{T}^{\mathsf{G}}(\mathfrak{B})$. Recall that $\mathsf{T}^{\mathsf{G}}(\mathfrak{B})=\bigcap_{i=1}^s\mathsf{T}^{\mathsf{G}}(\mathfrak{B}_i)$, and so $\mathsf{T}^{\mathsf{G}}(\mathfrak{B}) \subseteq \mathsf{T}^{\mathsf{G}}(\mathfrak{B}_i)$ for all $i=1,\dots,s$. Consequently, by Lemma \ref{4.18}, for each $i=1,\dots,s$, there exists a $\mathsf{G}$-graded epimorphism $\varphi_i: \mathfrak{A} \rightarrow \mathfrak{B}_i$. Finally, proceeding as in the proof of Lemma \ref{4.08}, and observing that $\varphi_i(ab)=0$ when $a\in\mathfrak{A}_p$ and $b\in\mathfrak{A}_q$ for $p\neq q$, the result follows.
\end{proof}
%
%
%
\section*{ACKNOWLEDGMENTS}

The author is grateful to Diogo Diniz and deeply thankful to Irina Sviridova and Igor Lima for their support in this work. 


\bibliographystyle{amsplain}
%

\end{document}